\documentclass{amsart}
\usepackage{fullpage,graphicx,color,url}
\usepackage{mathrsfs}
\usepackage{stmaryrd} % for double brackets (nearest integer)
\def\eq{\begin{equation}}
\def\endeq{\end{equation}}
\def\bpm{\begin{pmatrix}}
\def\epm{\end{pmatrix}}
\def\bbm{\begin{bmatrix}}
\def\ebm{\end{bmatrix}}
\def\pq{\mathcal{Q}}
\def\pp{\mathcal{P}}
\def\pu{\mathcal{U}}
\def\pv{\mathcal{V}}
\def\zc{z_*}
\def\rc{r_*}
\def\bc{\beta_*}
\def\tc{t_*}
\def\ec{\eta_*}
\def\Odot{{\bf \dot{O}}}

\renewcommand{\Re}{\mathrm{Re}}
\renewcommand{\Im}{\mathrm{Im}}

\theoremstyle{plain}
\newtheorem{rhp}{Riemann-Hilbert Problem}
\newtheorem{lemma}{Lemma}
\newtheorem{theorem}{Theorem}
\newtheorem{proposition}{Proposition}

\usepackage{tikz}
\DeclareRobustCommand*\circled[1]{\tikz[baseline=(char.base)]{
    \node[shape=circle,draw,inner sep=2pt] (char) {#1};}}
\newcommand\xqed[1]{%
  \leavevmode\unskip\penalty9999 \hbox{}\nobreak\hfill
  \quad\hbox{#1}}
\newcommand\myendrmk{\xqed{$\triangleright$}}

\theoremstyle{definition}
\newtheorem{remark}{$\triangleleft$ Remark}

\renewcommand{\theequation}{\arabic{section}-\arabic{equation}}
\makeatletter
\@addtoreset{equation}{section}
\makeatother

\title{Large-degree asymptotics of rational Painlev\'e-II functions.  II.}
\author{Robert J. Buckingham}
\address{Department of Mathematical Sciences, University of Cincinnati, Cincinnati, PO Box 210025, Cincinnati, OH 45221}
\email{buckinrt@uc.edu}
\urladdr{http://homepages.uc.edu/~buckinrt}
\author{Peter D. Miller}
\address{Department of Mathematics, University of Michigan, East Hall, 530 Church St., Ann Arbor, MI 48109}
\email{millerpd@umich.edu}
\urladdr{http://www.math.lsa.umich.edu/~millerpd}

\begin{document}
\begin{abstract}
This paper is a continuation of our analysis, begun in \cite{Buckingham-Miller-rational-noncrit}, of
the rational solutions of the inhomogeneous Painlev\'e-II equation and associated rational solutions of 
the homogeneous coupled Painlev\'e-II system in the limit of large degree.  In this paper we establish asymptotic formulae valid near a certain curvilinear triangle in the complex plane that was previously shown to separate two distinct types of asymptotic behavior.  Our results display both a trigonometric degeneration of the rational Painlev\'e-II functions and also a degeneration to the tritronqu\'ee solution of the Painlev\'e-I equation.  Our rigorous analysis is based on the steepest descent method applied to a Riemann-Hilbert representation of the rational Painlev\'e-II functions, and supplies leading-order formulae as well as error estimates.
\end{abstract}
\maketitle

\section{Introduction}
\label{section:intro}
Here we continue our investigation, begun in 
\cite{Buckingham-Miller-rational-noncrit}, of the large-degree asymptotic 
behavior of rational solutions to 
the inhomogeneous Painlev\'e-II equation
\eq
\label{PII}
p_{yy} = 2p^3 + \frac{2}{3}y p-\frac{2}{3}m, \quad p:\mathbb{C}\to\mathbb{C} \text{ with parameter } m\in\mathbb{C}
\endeq
and the coupled Painlev\'e-II system
\eq
\label{PII-system}
\left.\begin{matrix}
\vspace{.1in}\displaystyle u_{yy} + 2u^2v + \frac{1}{3}yu = 0 \\
\displaystyle v_{yy} + 2uv^2 + \frac{1}{3}yv = 0
\end{matrix}\right\}, 
\quad u,v:\mathbb{C}\to\mathbb{C}.
\endeq
The Painlev\'e-II equation \eqref{PII} has a rational solution if and only if 
$m\in\mathbb{Z}$ \cite{Airault:1979}, and when this rational solution exists 
it is unique \cite{Murata:1985}.  These rational solutions arise in the study 
of fluid vortices \cite{Clarkson:2009}, string theory \cite{Johnson:2006}, 
and transition behavior for the semiclassical sine-Gordon equation 
\cite{BuckinghamMcritical}.  
The rational solutions can be constructed as follows.  Define 
\eq
\label{backlund-initial-condition}
\pu_0(y):=1 \quad \text{and} \quad \pv_0(y):=-\frac{1}{6}y.
\endeq
Then define the rational functions $\pu_m(y)$ and $\pv_m(y)$ iteratively 
for positive integers $m$ by 
\eq
\label{backlund-positive}
\pu_{m+1}(y) := -\frac{1}{6}y\pu_m(y) - \frac{\pu_m'(y)^2}{\pu_m(y)} + \frac{1}{2}\pu_m''(y) \quad \text{and} \quad \pv_{m+1}(y) := \frac{1}{\pu_m(y)},
\endeq
and for negative integers $m$ by
\eq
\label{backlund-negative}
\pu_{m-1}(y) := \frac{1}{\pv_m(y)} \quad \text{and} \quad \pv_{m-1}(y) := \frac{1}{2}\pv_m''(y) - \frac{\pv_m'(y)^2}{\pv_m(y)} - \frac{1}{6}y\pv_m(y).
\endeq
Then $\{u,v\}=\{\pu_m,\pv_m\}$ solves the coupled Painlev\'e-II system 
\eqref{PII-system} for each choice of $m\in\mathbb{Z}$.  Furthermore, if 
we define 
\eq
\label{log-derivative}
\pp_m(y) := \frac{\pu_m'(y)}{\pu_m(y)} \quad \text{and} \quad \pq_m(y) := \frac{\pv_m'(y)}{\pv_m(y)}, \quad m\in\mathbb{Z},
\endeq
then $\pp_m$ satisfies \eqref{PII} with parameter $m$ while 
$\pq_m=\pp_{1-m}$.  It is sufficient to assume that $m>0$ (see \cite[Remark 2]{Buckingham-Miller-rational-noncrit}) and we will do so for the rest of this paper.

In \cite{Buckingham-Miller-rational-noncrit}, it was shown that, in the 
scaled coordinate 
\begin{equation}
\label{x}
x:=(m-\tfrac{1}{2})^{-2/3}y,
\end{equation} 
the zeros and poles of $\pp_m$, $\pq_m$, $\pu_m$, and $\pv_m$ are contained 
within and densely fill out a fixed (independent of $m$) domain $T$ 
(the \emph{elliptic region}) as $m\to\infty$ (see Figure 
\ref{u13-u14-zeros-poles}).  The explicit 
analytical definition of $T$ is given in 
\cite[Section 3.2]{Buckingham-Miller-rational-noncrit}, and further details can be found below in \S\ref{sec:edges-and-corners}.  These results 
are consistent with numerical studies of the related Yablonskii-Vorob'ev 
polynomials (the functions $\pu_m$ are ratios of successive 
Yablonskii-Vorob'ev polynomials) by Clarkson and Mansfield 
\cite{Clarkson:2003}, as well as Kapaev's results \cite{Kapaev:1997} on the 
large-$m$ asymptotic behavior of general solutions to \eqref{PII}.  
Furthermore, the leading order large-$m$ asymptotic behaviors of $\pp_m$, 
$\pq_m$, $\pu_m$, and $\pv_m$ were computed, with error term of order 
$m^{-1}$, assuming that $x$ is not close to $\partial T$.  Subsequently some of these asymptotic results were reproduced by other authors \cite{BertolaB14}.  In this work we complete the analysis by filling in the missing 
parts of the complex $x$-plane near the smooth arcs and corner points of $\partial T$, supplying asymptotic formulae for the rational Painlev\'e-II functions in these remaining regions.
\begin{figure}[h]
\includegraphics[width=2in]{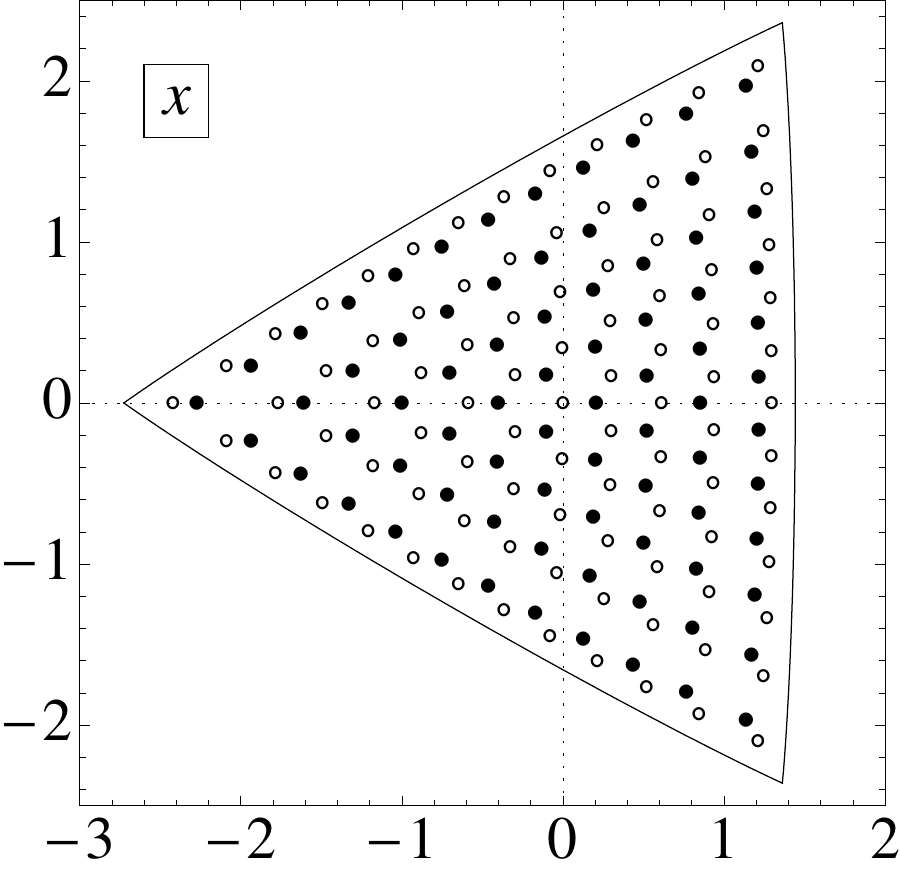}
\includegraphics[width=2in]{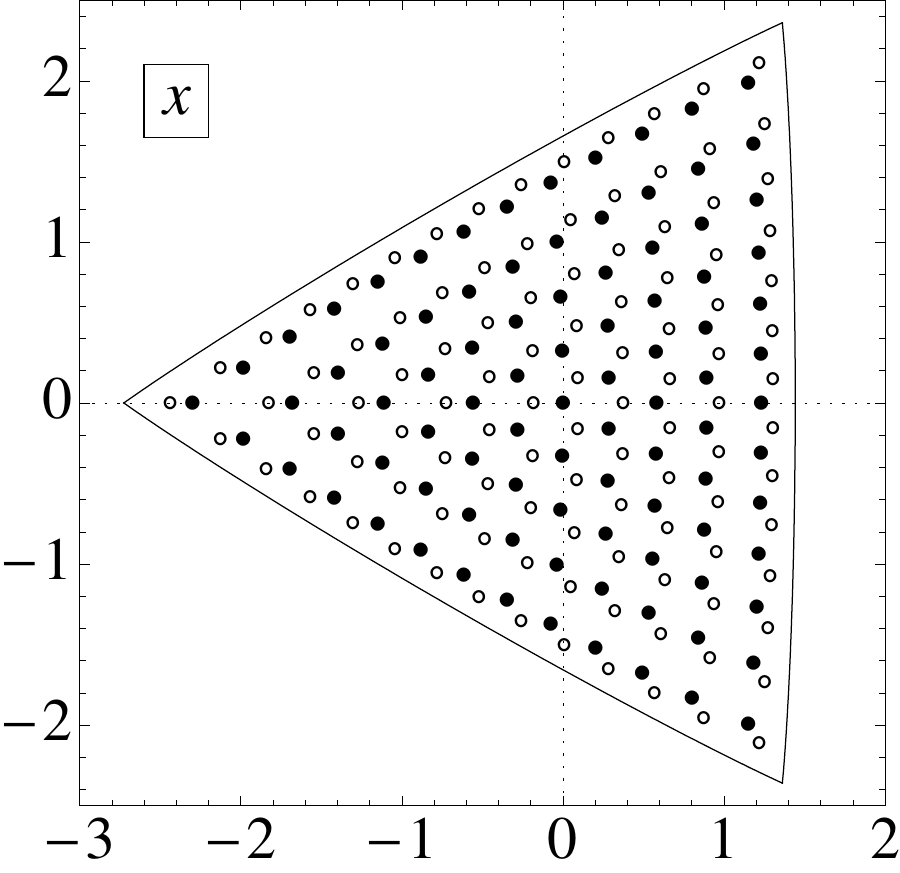}
\caption{The zeros (circles) and poles (dots) of $\pu_m(y)$ for $m=13$ (left) 
and $m=14$ (right) in the complex $x$-plane, where $x=(m-\tfrac{1}{2})^{-2/3}y$, 
along with the $m$-independent boundary of the region $T$.  
Note that each pole of $\pu_{14}(y)$ is also a zero of $\pu_{13}(y)$ (these points 
are slightly shifted between the two plots since the map $y\mapsto x$ depends on $m$).  Every zero of $\pu_m(y)$ is a pole of $\pp_m(y)$ with residue $+1$, 
while every pole of $\pu_m(y)$ is a pole of $\pp_m(y)$ with residue $-1$.}
\label{u13-u14-zeros-poles}
\end{figure}
%\begin{figure}[h]
%\includegraphics[width=2in]{u14-zeros-poles.pdf}
%\includegraphics[width=2in]{p14-zeros-poles.pdf}
%\caption{The zeros (circles) and poles (dots) of $\pu_{14}(y)$ (left) 
%and $\pp_{14}(y)$ (right) in the complex $x$-plane, where 
%$x=(14-\tfrac{1}{2})^{-2/3}y$, 
%along with the $m$-independent boundary of the elliptic region $T$.  
%}
%\label{u14-p14-zeros-poles}
%\end{figure}

\subsection{Results} 
\label{section:results-summary}
The Painlev\'e functions have the two following 
discrete symmetries (see, for example, 
\cite[Section 2]{Buckingham-Miller-rational-noncrit}):
\begin{equation}
\pu_m(y^*)=\pu_m(y)^*, \quad \pv_m(y^*)=\pv_m(y)^*, \quad \pp_m(y^*)=\pp_m(y)^*, \quad \pq_m(y^*)=\pq_m(y)^*
\end{equation}
and
\begin{equation}
\label{pp-pq-pu-pv-symmetries}
\begin{split}
\pp_m(e^{-2\pi i/3}y)=e^{2\pi i/3}\pp_m(y), \quad &\pq_m(e^{-2\pi i/3}y)=e^{2\pi i/3}\pq_m(y), \\
\pu_m(e^{-2\pi i/3}y)=e^{-2m\pi i/3}\pu_m(y),\quad &\pv_m(e^{-2\pi i/3}y)=e^{2(m-1)\pi i/3}\pv_m(y).
\end{split}
\end{equation}
The curve $\partial T$ turns out to consist of three smooth arcs (``edges'') that terminate in pairs at certain points (``corners'') lying along the rays $\arg(-x)=0$ and $\arg(x)=\pm\pi/3$ (complete details and precise definitions can be found in \S\ref{sec:edges-and-corners} below, but these features are completely obvious from the plots in Figure~\ref{u13-u14-zeros-poles}).
Together with the exact symmetry \eqref{pp-pq-pu-pv-symmetries}, this shows it is sufficient to 
analyze the behavior of the rational Painlev\'e-II functions near one edge and 
one corner of $\partial T$ (we pick those that intersect the real axis).  

Our analysis of the rational Painlev\'e-II functions for $x$ near the edge of $\partial T$ subtending the sector $|\arg(x)|<\pi/3$ is presented in \S\ref{section:edge}, and the main results are formulated in
Theorem~\ref{theorem:edge} (which provides asymptotic formulae for $\pu_m$ and $\pv_m$) and Theorem~\ref{theorem:edge-p-q} (which provides asymptotic formulae for $\pp_m$ and $\pq_m$).
These results show that, when viewed in terms of a conformal coordinate (independent of $m$) that maps the edge to a vertical segment, the rational functions can be represented for large $m$ as infinite series of trigonometric terms periodic in the direction parallel to the (straightened) edge, up to an absolute error term proportional to $m^{-1}$.  The period of each term is proportional to $m^{-1}$, and subsequent terms in the series are displaced in the direction perpendicular to the (straightened) edge by a shift proportional to $m^{-1}\log(m)$.  Details of the location of poles of the approximating series are given in \S\ref{section:edge-singularities}, including 
information regarding how the pole lattices shift smoothly as $x$ moves along the edge, how they jump when $m$ is incremented, and how the lattices for subsequent terms in the series are related.  The 
infinite-series formulae show remarkable agreement with the actual rational Painlev\'e-II functions, even when $m$ is not very large and even when $x$ is not so close to $\partial T$.  This agreement is illustrated qualitatively in Figures~\ref{fig:UEdge1}--\ref{fig:UEdge2} (for $\pu_m$) and in Figure~\ref{fig:PEdge} (for $\pp_m$).  It has to be noted, however, that the error terms fail to be controllable if $x$ is allowed to move so far along the edge as to approach one of the two corner points; a different type of analysis is required in this situation.

In \S\ref{section:cusp} we resolve this issue by analyzing the rational Painlev\'e-II functions near the corner point $x=x_c$ of $\partial T$ that lies on the negative real axis.   The main results are formulated in Theorem~\ref{theorem:corner}, which shows that for $x-x_c$ of order $\mathcal{O}(m^{4/5})$ the rational Painlev\'e-II functions can be represented in the limit $m\to\infty$ in terms of a certain solution of the Painlev\'e-I equation $Y''(t)=6Y(t)^2+t$ (a \emph{tritronqu\'ee} solution; for details see \S\ref{section:tritronquee-parametrix}).  The error terms are small in the limit $m\to\infty$ but are large compared with $m^{-1}$; consequently it is necessary to consider $m$ quite large to see good agreement of the exact rational Painlev\'e-II functions with their approximations.  The poles of $\pp_m$ and the approximating tritronqu\'ee 
function are compared in Figure \ref{tritronquee-vs-p10-p40};  note that 
two simple poles of $\pp_m$ of opposite residues converge to each double 
tritronqu\'ee pole.  The functions $\pu_m$ and $\pp_m$ are compared with their 
large-$m$ approximations via Painlev\'e-I functions for real $t$ in 
Figures~\ref{fig:Um-corner} and \ref{fig:Pm-corner}.  
\begin{figure}[h]
\includegraphics[width=2.5in]{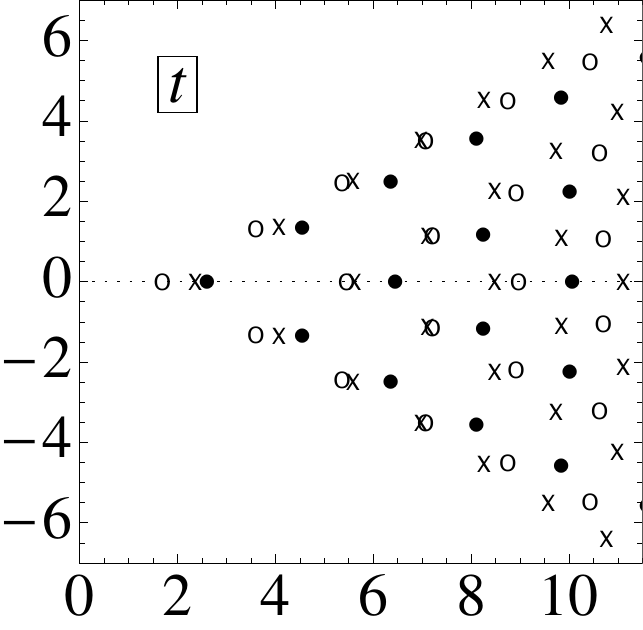}
\includegraphics[width=2.5in]{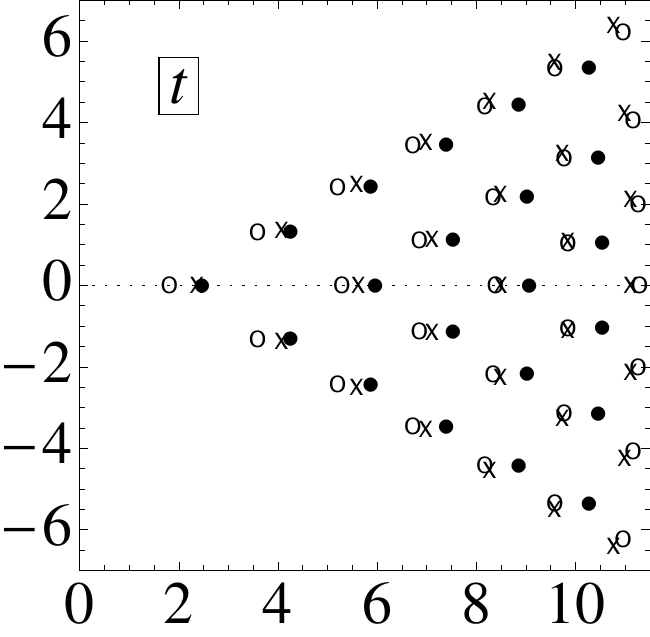}
\caption{Poles of the Painlev\'e-I tritronqu\'ee solution $Y(t)$ (denoted by 
the symbol X) and $\pp_m(y)$ (circles correspond to poles with residue $+1$ and 
dots correspond to poles with residue $-1$) for $m=10$ (left) and $m=40$ 
(right) in the complex $t$-plane, where 
$t=2^{1/15}3^{-1/3}m^{4/5}((m-\tfrac{1}{2})^{-2/3}y-x_c)$.  In the $t$-plane 
the poles of $Y(t)$ are independent of $m$, and $t=0$ corresponds to $x=x_c$.  
In the large-$m$ limit, one (simple) pole of $\pp_m$ with residue $+1$ and one 
with residue $-1$ converge to each (double) pole 
of $Y$.  The tritronqu\'ee poles were numerically computed using the pole 
field solver code provided by Fornberg and Weideman \cite{FornbergW11}.
}
\label{tritronquee-vs-p10-p40}
\end{figure}

The fact that the Painlev\'e-I equation appears exactly when $x$ is near $x_c$ and the other two corner points of $\partial T$ can be appreciated at a formal level by a simple scaling argument (see also \cite{Kapaev:1997}).  Indeed, consider making simultaneous affine coordinate changes in the independent and dependent variables of the inhomogeneous Painlev\'e-II equation \eqref{PII}:
\begin{equation}
y=a+bt\quad\text{and}\quad p=\alpha+\beta Y,
\end{equation}
where $a$, $b$, $\alpha$, and $\beta$ are constants, while $t$ and $Y$ are the new independent and dependent variables, respectively.  Making these substitutions in \eqref{PII} yields
\begin{equation}
\frac{\beta}{b^2}Y_{tt}=2\alpha^3+6\alpha^2\beta Y+6\alpha\beta^2 Y^2+2\beta^3Y^3+\frac{2}{3}a\alpha
+\frac{2}{3}\alpha b t + \frac{2}{3}a\beta Y +\frac{2}{3}b\beta t Y-\frac{2}{3}m.
\label{eq:PII-rescale}
\end{equation}
Balancing the terms proportional to $Y_{tt}$, $Y^2$, and $t$ (the terms appearing in the Painlev\'e-I equation) by taking
\begin{equation}
\alpha\beta b^2=1\quad\text{and}\quad\frac{b}{\beta^2}=\frac{3}{2}\quad\implies\quad
\alpha=\frac{4}{9}\frac{1}{\beta^5}\quad\text{and}\quad b=\frac{3}{2}\beta^2,
\label{eq:PI-identity}
\end{equation}
the transformed equation \eqref{eq:PII-rescale} becomes
\begin{equation}
Y_{tt}-6Y^2-t=2\left(\frac{2}{3}\right)^4\beta^{-12}+6\left(\frac{2}{3}\right)^2\beta^{-6}Y +
2\left(\frac{3}{2}\right)^2\beta^6Y^3+\frac{2}{3}a\beta^{-2}+\frac{3}{2}a\beta^4Y+\left(\frac{3}{2}\right)^2\beta^6tY-\frac{3}{2}m\beta^3.
\label{eq:PII-rescale-again}
\end{equation}
The right-hand side can be made formally small in the limit $m\to\infty$ by choosing $a$ and $\beta$ appropriately.  Indeed, one may first observe that there is a unique term proportional to the product $tY$, and for this term to be negligible it is necessary that $\beta\ll 1$.  Also, there are three constant terms, one of which is proportional to $\beta^{-12}$, which is large; this situation can be resolved by assuming that the constant terms sum exactly to zero:
\begin{equation}
2\left(\frac{2}{3}\right)^4\beta^{-12} + \frac{2}{3}a\beta^{-2} -\frac{3}{2}m\beta^3=0.
\label{eq:constant-terms}
\end{equation}
Similarly, there are two terms proportional to $Y$ (without a factor of $t$), one of which is proportional to $\beta^{-6}$, which is large; we therefore remove these terms by supposing that
\begin{equation}
6\left(\frac{2}{3}\right)^2\beta^{-6}+\frac{3}{2}a\beta^4=0.
\label{eq:linear-terms}
\end{equation}
Solving \eqref{eq:constant-terms} and \eqref{eq:linear-terms} for $a$ and $\beta$ in terms of $m$
and using \eqref{eq:PI-identity} we arrive at
\begin{equation}
\begin{split}
a&=-\left(\frac{9}{2}\right)^{2/3}m^{2/3}e^{2\pi ik/3},\\
b&=\frac{3^{1/3}}{2^{1/15}}m^{-2/15}e^{4\pi i k/15},\\
\alpha&=-6^{-1/3}m^{1/3}e^{-2\pi ik/3},\\
\beta&=-4^{1/15}\left(\frac{2}{3}\right)^{1/3}m^{-1/15}e^{2\pi ik/15},
\end{split}
\end{equation}
where $k\in\mathbb{Z}$, which puts \eqref{eq:PII-rescale-again} in the form of a perturbation of the Painlev\'e-I equation:
\begin{equation}
Y_{tt}=6Y^2+t + \mathcal{O}(m^{-2/5}Y^3) +\mathcal{O}(m^{-2/5}tY),\quad m\to\infty.
\end{equation}
Based on this formal argument, one is led to expect solutions of \eqref{PII} to be asymptotically expressible in terms of solutions of Painlev\'e-I near three distinguished points in the $y$-plane given by the three values of $a$; upon rescaling by \eqref{x} these points correspond to the three corners of $\partial T$.  Of course this formal argument does not suggest which solutions of Painlev\'e-I should appear in the large-$m$ behavior of any given family of solutions of \eqref{PII}, nor does it yield a direct method of proof.

\subsection{Notation}
We define for later use the Pauli spin matrices
\eq
\sigma_1:=\bbm 0 & 1 \\ 1 & 0 \ebm, \quad \sigma_2:=\bbm 0 & -i \\ i & 0 \ebm, \quad \sigma_3:=\bbm 1 & 0 \\ 0 & -1 \ebm, \quad
\sigma_+:=\bbm 0 & 1\\ 0 & 0\ebm,\quad \sigma_-:=\bbm 0 & 0\\ 1 & 0\ebm.
\label{eq:Pauli-matrices}
\endeq
We will frequently refer to certain sectors of the complex plane, for which we define special notation:
\begin{equation}
\mathcal{S}_\sigma:=\left\{x\in\mathbb{C}:  |\arg(x)|<\frac{\pi}{3}-\sigma\right\}, \quad\sigma\ge 0.
\label{eq:edge-sector-define}
\end{equation}
Given an oriented contour arc and a function defined locally in the complement of the arc, we use the subscript ``$+$'' (respectively, ``$-$'') to denote the nontangential boundary value taken by the function on the arc 
from the left (respectively, right).  
Finally, it will be convenient to introduce the notation
\begin{equation}
\label{epsilon}
\epsilon:=(m-\tfrac{1}{2})^{-1}.
\end{equation}

\subsection{Acknowledgements}  
In making Figures \ref{tritronquee-vs-p10-p40}, \ref{fig:Um-corner}, and 
\ref{fig:Pm-corner} we used the 
``pole field solver''  of Bengt Fornberg and Andr\'e Weideman 
\cite{FornbergW11} to compute the tritronqu\'ee functions $H$ and $Y$, and we 
are grateful to them for providing us with a code for this purpose.  
We also thank Marco Bertola for pointing out the approach to the proof of 
Lemma \ref{lemma:conformal} via differential equations in Banach space.  
%We are also 
%thankful for beneficial conversations with many others, including Peter 
%Clarkson, Alexander Its, Andrei Kapaev, Erik Koelink, Andrei 
%Martinez-Finkelshtein, Davide Masoero, and Boris Shapiro.  
R. J. Buckingham was partially supported by the National Science Foundation via grant DMS-1312458, by the Simons Foundation via award 245775, and the Charles Phelps Taft Research Foundation. P. D. Miller was partially supported by the National Science Foundation under grants DMS-0807653 and DMS-1206131, and by a Fellowship in Mathematics from the Simons Foundation.

\section{Edges and Corners}  
\label{sec:edges-and-corners}
The curve $\partial T$ can be detected with the help of
the $g$-function introduced in \cite{Buckingham-Miller-rational-noncrit} to establish asymptotic formulae for
the rational Painlev\'e-II functions under the assumption that $|x|$ is sufficiently large.  In this section, we recall that $g$-function and use it to precisely describe the boundary $\partial T$ of the set $T$
as this will be the focus of the analysis in the rest of the paper.  
\subsection{Definition of $g$ and related functions}
\label{section:g-define}
Set 
\begin{equation}
x_c:=-\left(\frac{9}{2}\right)^{2/3}
\label{eq:xc-define}
\end{equation}
(this turns out to be the leftmost corner point of $T$) and define 
$\Sigma_S$ to be the union of the three straight line segments 
$[x_c,0]$, $[0,x_ce^{-2i\pi/3}]$, and $[0,x_ce^{2i\pi/3}]$ (see Figure~\ref{fig:Sigma-S}).
\begin{figure}[h]
\includegraphics{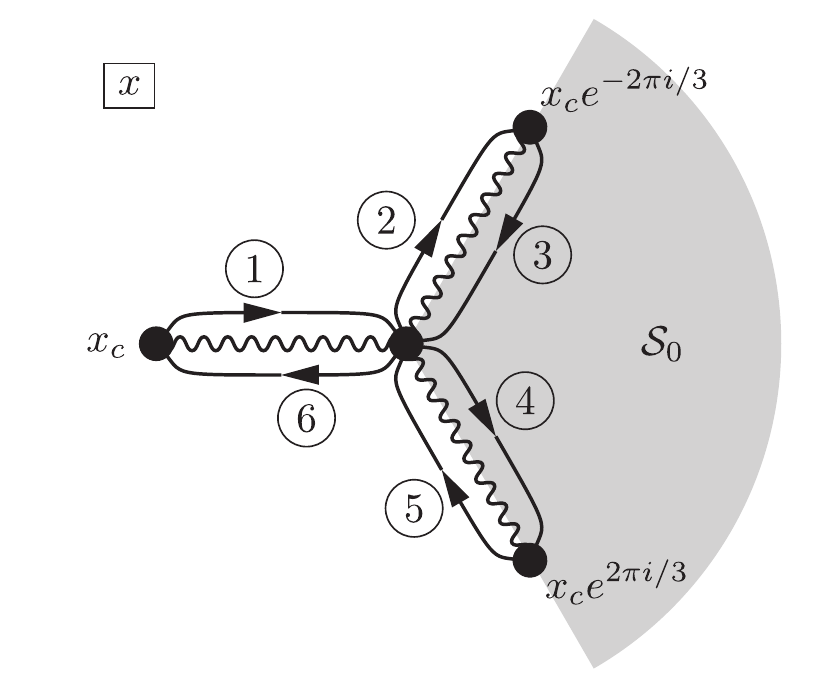}
\caption{The jump contour $\Sigma_S$ for $S=S(x)$ (wavy lines), and the six segments of the positively-oriented boundary of $\mathbb{C}\setminus\Sigma_S$.  The sector $\mathcal{S}_0$ is shaded.  The contours $\circled{1}$ -- $\circled{6}$ correspond to the contours in the $S$-plane with the same labels in Figure \ref{fig:S-plane}.}
\label{fig:Sigma-S}
\end{figure}
Let
$S:\mathbb{C}\backslash\Sigma_S\to\mathbb{C}$ 
denote the unique analytic function satisfying  
\eq
\label{cubic-equation}
3S(x)^3+4xS(x)+8=0
\endeq
with branch cut $\Sigma_S$ and satisfying $S(x)=-2x^{-1}+\mathcal{O}(x^{-4})$ 
as $x\to\infty$.  The three endpoints of $\Sigma_S$ are precisely the values of $x$ at which \eqref{cubic-equation} has a double root.  It is easy to see that $S$ is Schwarz-symmetric and satisfies the rotational symmetry
\begin{equation}
S(xe^{2\pi i/3})=e^{-2\pi i/3}S(x) \text{ for all } x\in \mathbb{C}\setminus\Sigma_S.
\label{eq:S-rotational-symmetry}
\end{equation}
\begin{lemma}
$S:\mathbb{C}\setminus\Sigma_S\to\mathbb{C}$ is univalent, and its conformal image is the domain
bounded by the arc $S=(\tfrac{4}{3})^{1/3}(\sec(\vartheta))^{1/3}e^{i\vartheta}$, $|\vartheta|<\pi/3$, and its rotations about the origin by $\pm 2\pi/3$ radians (see Figure~\ref{fig:S-plane}).
\label{lemma-S-univalent}
\end{lemma}
The proof of Lemma~\ref{lemma-S-univalent} is presented in \S\ref{sec:lemma-S-univalent}.
\begin{figure}[h]
\includegraphics{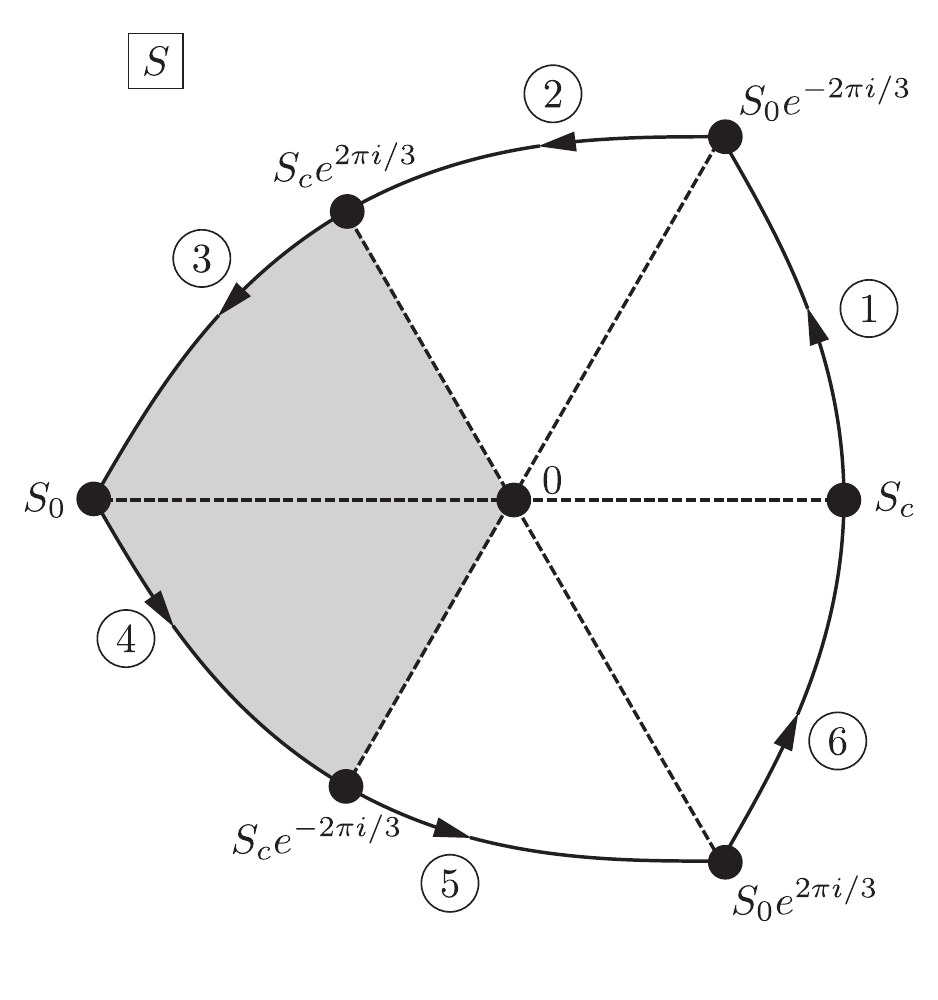}
\caption{The conformal image of $\mathbb{C}\setminus\Sigma_S$ under $S$.  The shaded region is the image of the sector $\mathcal{S}_0$.  Here $S_0:=-(8/3)^{1/3}$ and $S_c:=(4/3)^{1/3}$.  The contours $\circled{1}$ -- $\circled{6}$ correspond to the contours in the $x$-plane with the same labels in Figure \ref{fig:Sigma-S}.  Note that this figure does \emph{not} depict the elliptic region 
$T$ (in fact the interior corner angles here are $2\pi/3$ radians, while those 
of $T$ will be shown in \S\ref{opening-angle-subsection} to be $2\pi/5$ 
radians).}
\label{fig:S-plane}
\end{figure}
Define the semi-infinite ray 
$\mathscr{R}_{-\pi/3}:=(x_ce^{2i\pi/3},\infty e^{-i\pi/3})$.  Then let 
$\Delta:\mathbb{C}\backslash(\Sigma_S\cup\mathscr{R}_{-\pi/3})\to\mathbb{C}$ 
be the unique analytic function satisfying
\eq
\Delta(x)^2 = \frac{16}{3S(x)}
\label{eq:Delta-define}
\endeq
that is positive real for real $x<x_c$.  Now define 
$a,b:\mathbb{C}\backslash(\Sigma_S\cup\mathscr{R}_{-\pi/3})\to\mathbb{C}$ by  
\eq
a(x):=\frac{S(x)-\Delta(x)}{2}, \quad b(x):=\frac{S(x)+\Delta(x)}{2}.
\label{eq:a-b-define}
\endeq
(Observe that $S$ is the \emph{sum} $a+b$ and $\Delta$ is the 
\emph{difference} $b-a$).  
For $x\in\mathbb{C}\backslash(\Sigma_S\cup\mathscr{R}_{-\pi/3})$, 
$a(x)$ and $b(x)$ are distinct points in the complex plane satisfying 
$\text{Re}(a(x))<\text{Re}(b(x))$ for $0<\arg(x)<5\pi/3$ and 
$\text{Im}(a(x))>\text{Im}(b(x))$ for $-\pi/3<\arg(x)<\pi$.  
We call the oriented 
straight line segment $\Sigma:=\overrightarrow{ab}$ the \emph{band}.  
%This 
%is the sole jump contour in the Riemann-Hilbert problem for the outer 
%parametrix when $x$ is outside $T$ (and bounded a fixed distance away from 
%$T$ as $\epsilon\to 0$).
Given values $a$ and $b$, let
$r:\mathbb{C}\backslash\Sigma\to\mathbb{C}$ be the analytic function 
satisfying 
\eq
r(z;a,b)^2=(z-a)(z-b) \quad \text{and} \quad r(z;a,b)=z+\mathcal{O}(1), \quad z\to\infty.
\endeq
Let $L$ (depending on $x$) denote an unbounded arc joining $z=b$ to 
%$z=-S/2$ to  
$z=\infty$ without otherwise touching $\Sigma$, and suppose that $L$ agrees with the positive real $z$-axis for sufficiently large $|z|$.  Then set
\begin{equation}
\zc=\zc(x):=-\frac{1}{2}S(x),
\label{eq:edge-zc}
\end{equation}
and then
\begin{equation}
h(z)=h(z;x):=\frac{3}{2}\int_{a(x)}^z (\zeta-\zc(x))r(\zeta;a(x),b(x))\,d\zeta -\frac{1}{2}\lambda(x),\quad
z\in \mathbb{C}\setminus (\Sigma\cup L),
\label{eq:hdef}
\end{equation}
where $\lambda(x)$ denotes an integration constant (see below) and the path of integration lies in $\mathbb{C}\setminus (\Sigma\cup L)$.  It is a consequence of the definition of $a(x)$ and $b(x)$ that
the associated function $g(z)=g(z;x)$ defined by
\begin{equation}
g(z;x):=\frac{1}{2}\theta(z;x)-h(z;x),
\label{g-formula}
\end{equation}
where $\theta:\mathbb{C}\to\mathbb{C}$ is defined by 
\begin{equation}
\theta(z)=\theta(z;x):=z^3+xz,
\label{eq:theta-define}
\end{equation}
has the asymptotic expansion
\begin{equation}
g(z;x)=\log(-z)+g_0(x)+\mathcal{O}(z^{-1}),\quad z\to\infty,
\end{equation}
where $g_0(x)$ is a constant related to $\lambda(x)$.  The constant $\lambda(x)$ is chosen to ensure that $g_0(x)=0$.
%
%Let $\mathscr{L}(z)$ be given by
%\begin{equation}
%\mathscr{L}(z)=\mathscr{L}(z;x):=\log\left(\frac{S(x)-2z-2r(z;x)}{4}\right), \quad z\notin\Sigma\cup L, \quad x\notin \Sigma_S \cup \mathscr{R}_{-\pi/3}
%\end{equation}
%with the interpretation that $\mathscr{L}$ is single-valued and analytic in $z$ where defined, and
%that $\mathscr{L}(z)=\log(-z)+\mathcal{O}(z^{-1})$ as $|z|\to\infty$.  
%For each $x\in\mathbb{C}\backslash(\Sigma_S\cup\mathscr{R}_{-\pi/3})$, 
%the $g$-function is defined via \textcolor{red}{(is the formula for $g$ needed, same question for $\mathscr{L}$, $h$, and $\lambda$)}
%\begin{multline}
%\label{g-formula}
%g(z) = g(z;x) := \frac{1}{2}\theta(z;x) - \frac{1}{8}\left(4z^2+2S(x)z-2S(x)^2-\Delta(x)^2\right)r(z;x) + \mathscr{L}(z;x) + \frac{1}{8}S(x)^3,  \\
%\quad z\notin\Sigma\cup L, \quad x\notin \Sigma_S \cup \mathscr{R}_{-\pi/3},
%\end{multline}
%where $\theta:\mathbb{C}\to\mathbb{C}$ is defined by 
%\begin{equation}
%\theta(z)=\theta(z;x):=z^3+xz.
%\label{eq:theta-define}
%\end{equation}
%A key role in the analysis is played by the related function 
%\begin{equation}
%h(z)=h(z;x):=\frac{1}{2}\theta(z;x)-g(z;x),\quad z\in\mathbb{C}\setminus(\Sigma\cup L),\quad x\in\mathbb{C}\setminus(\Sigma_S\cup\mathscr{R}_{-\pi/3}),
%\label{eq:hgdef}
%\end{equation}
%which satisfies 
%\begin{equation}
%h'(z)=\frac{3}{2}\left(z+\frac{1}{2}S\right)r(z),\quad z\in\mathbb{C}\setminus\Sigma.
%\label{eq:hprime}
%\end{equation}
%Obviously, $h(z)$ has a single critical point at which it is analytic, namely $z=\zc$ where
%
Further salient properties of $g$ and $h$ are recorded in 
\cite[Proposition 1]{Buckingham-Miller-rational-noncrit}, one of which is 
that 
the identity $h_+(z;x)+h_-(z;x)+\lambda(x)=0$ holds for all $z\in\Sigma$.
%there is a complex number $\lambda=\lambda(x)$ that is the constant 
%value of $-(h_+(z)+h_-(z))$ for $z\in\Sigma$.  In particular, 
%$\lambda(x)$ may be expressed for $x\not\in\Sigma_S\cup\mathscr{R}_{-\pi/3}$ as
%\begin{equation}
%\begin{split}
%\lambda(x)&=\mathscr{L}_+(z;x)+\mathscr{L}_-(z;x)+\frac{1}{4}S(x)^3,\quad z\in\Sigma\\
%%&=\frac{1}{4}S(x)^3 +\log\left(\frac{\Delta(x)^2}{16}\right)\\
%&=\frac{1}{4}S(x)^3 -\log(3S(x))
%\end{split}
%\label{eq:lambdadef}
%\end{equation}
%for an appropriate choice of the logarithm (here not necessarily the principal branch).

\subsection{Characterization of $\partial T$ in terms of $g$}
The analysis of the rational Painlev\'e-II functions carried out in \cite{Buckingham-Miller-rational-noncrit} with the help of the function $g$ defined by \eqref{g-formula} succeeds as long as the contour $L$ can be positioned so that the inequality $\Re(2h(z;x)+\lambda(x))>0$ holds for $z\in L$ and along two other unbounded contours with asymptotic directions $\arg(z)=\pm 2\pi/3$.  This is the case for sufficiently large $|x|$.  The boundary $\partial T$ is then the set of $x\in\mathbb{C}$ for which the region of the inequality $\Re(2h(z;x)+\lambda(x))>0$ undergoes a topological bifurcation, such that once $x\in T$ it is no longer possible to choose $L$ or the other two contours with the relevant inequality holding at every point.
The hallmark of the bifurcation is the appearance of the critical point $z=\zc(x)$ on the zero level curve of the function $\Re(2h(z;x)+\lambda(x))$.  If we set
\begin{equation}
\mathfrak{c}(x):= \frac{3}{2}\int_{a(x)}^{\zc(x)}(\zeta-\zc(x))r(\zeta;a(x),b(x))\,d\zeta,
\label{eq:c-function-define}
\end{equation}
where the path of integration is a straight line, then it is clear from comparison to \eqref{eq:hdef} that 
$\mathfrak{c}(x)=h(\zc(x);x)+\tfrac{1}{2}\lambda(x)\pmod{2\pi i}$ since the residue of $h'(z;x)$ at $z=\infty$ is $-1$.  Therefore, 
the bifurcation corresponds to the equation $\Re(\mathfrak{c}(x))=0$.  We note that $\mathfrak{c}(x)$ is well-defined for $x\in\mathcal{S}_0\cup\mathcal{S}_0e^{2\pi i/3}\cup\mathcal{S}_0e^{-2\pi i/3}$,
%$|\arg(x)|<\pi/3$, $\pi/3<\arg(x)<\pi$, and $-\pi<\arg(x)<-\pi/3$, 
but along the three excluded rays either $\zc(x)$ fails to be well-defined (because $x\in\Sigma_S$) or $\zc(x)\in\Sigma$ so that the integral \eqref{eq:c-function-define} is not well-defined.  The bifurcation phenomenon is illustrated in Figures~\ref{edge-break} and \ref{corner-break}.
\begin{figure}[h]
\includegraphics[width=1.5in]{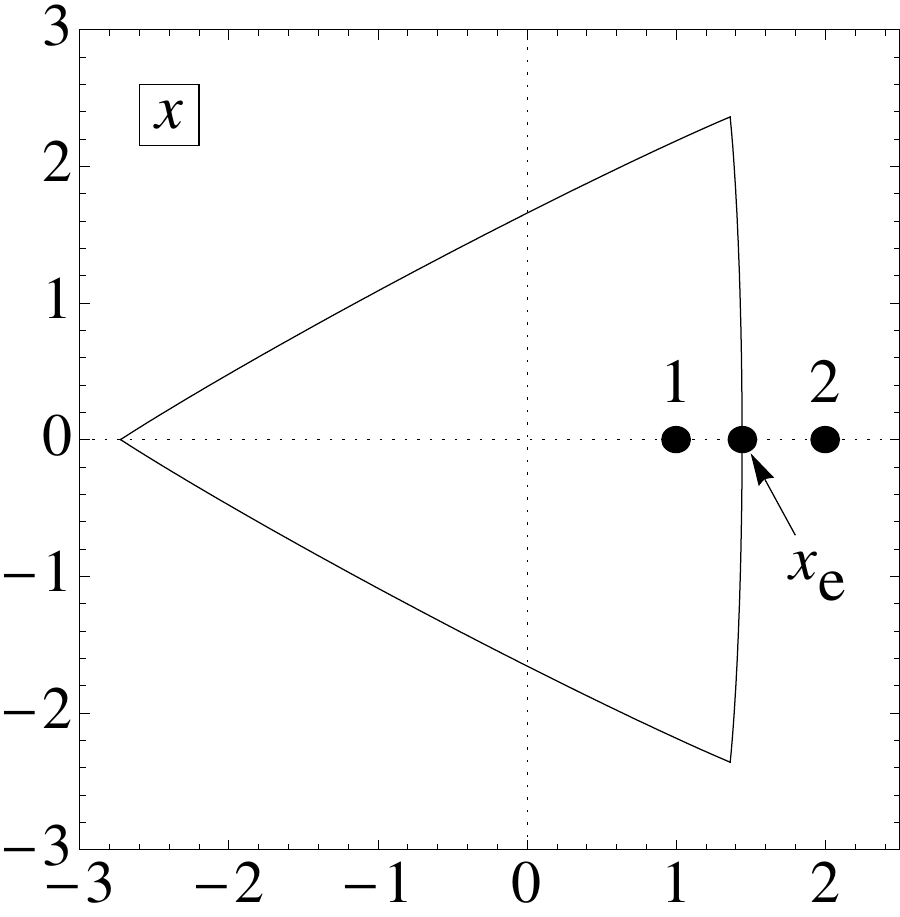}
\includegraphics[width=1.5in]{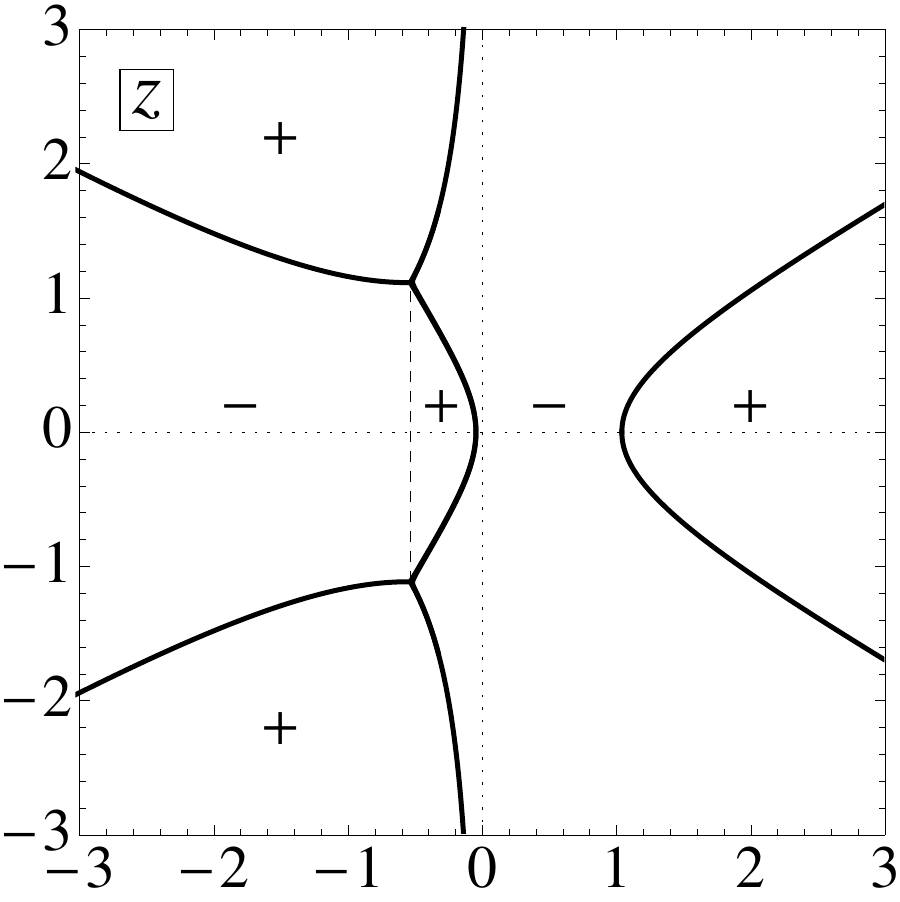}
\includegraphics[width=1.5in]{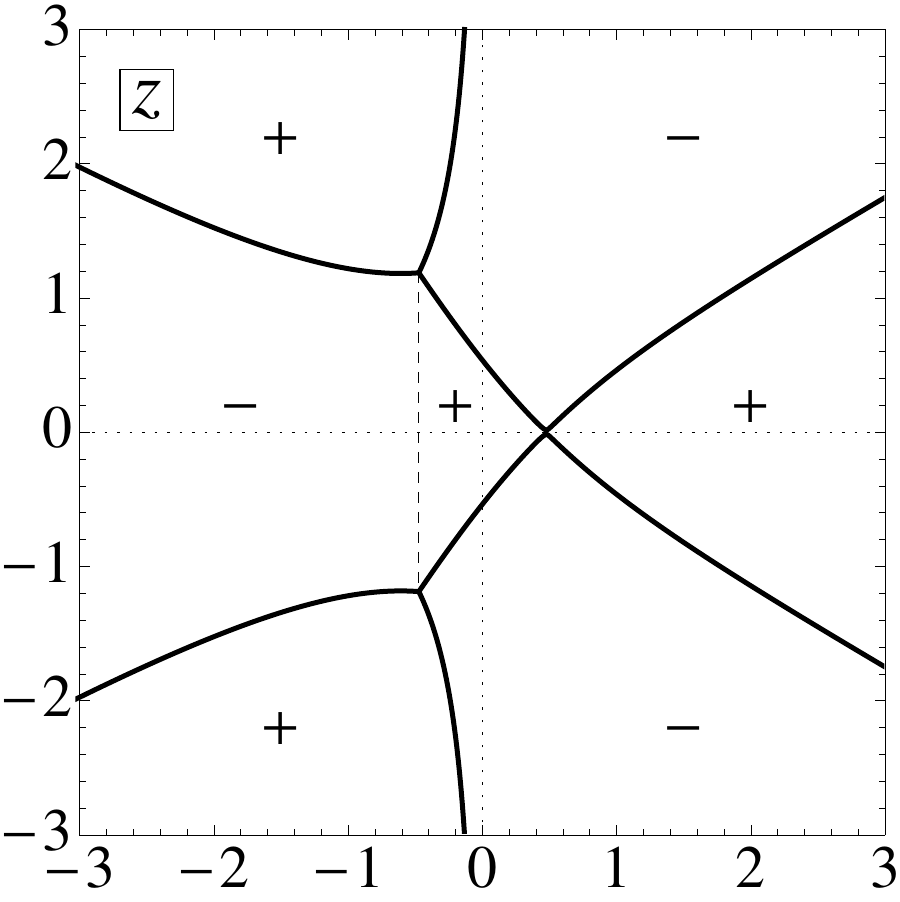}
\includegraphics[width=1.5in]{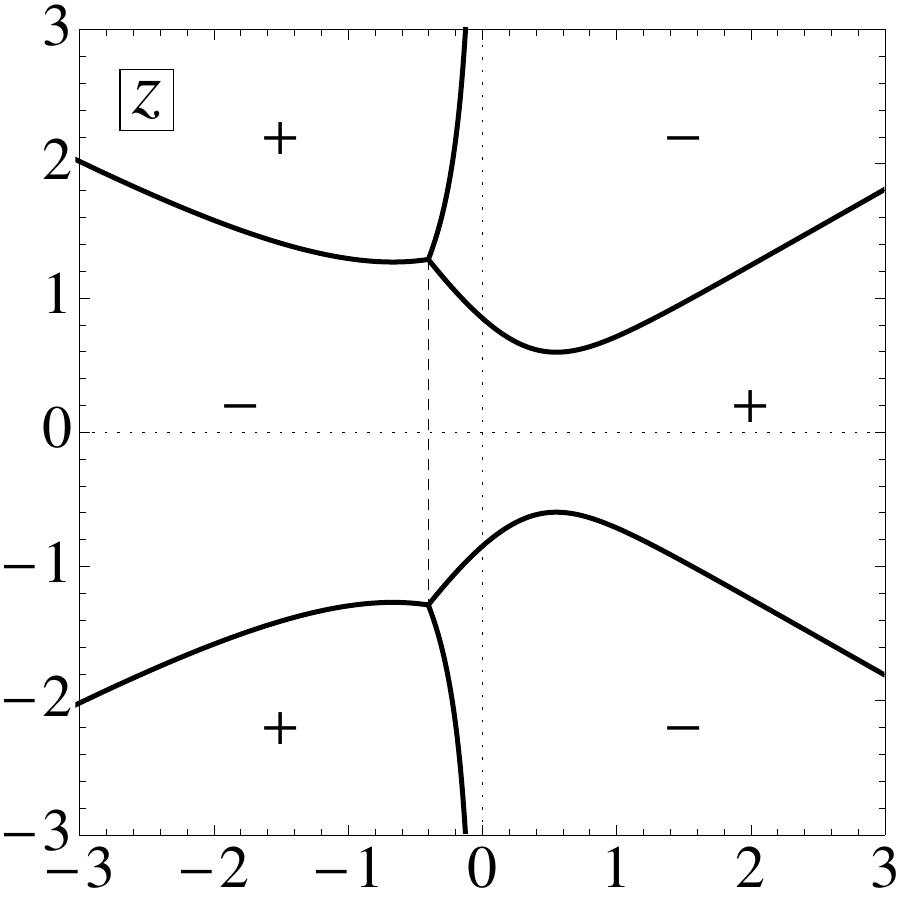}
\caption{
Signature charts for $\Re(2h(z;x)+\lambda(x))$ in the $z$-plane corresponding to $x=1$ (second panel), $x=x_e\approx 1.445$ (third panel), and $x=2$ (fourth panel).  Solid curves are zero level curves, while the dashed line in each case is the segment $\Sigma$ across which $h$ has a jump discontinuity.  
The location of each of these $x$-values in relation to $\partial T$ is illustrated in the first panel.  
%The breaking mechanism along an edge of the root region $T$.  
%First picture:  Relative locations of the chosen $x$-values and $\partial T$.  
%Second picture:  The signature chart of $\Re(2h(z;x)+\lambda(x))$ with 
%$x=1$.  Note that, due to the lack of a path from $a$ or $b$ to large positive 
%real values of $z$ along which $\Re(2h+\lambda)>0$, the function $h$ 
%cannot be used to carry out the asymptotic analysis for $x=1$.  
%Third picture:  The signature chart of 
%$\Re(2h(z;x)+\lambda(x))$ with $x=x_e\approx 1.445$.  
%Fourth picture:  The signature chart of $\Re(2h(z;x)+\lambda(x))$ with 
%$x=2$.  The solid lines are zero level curves, while the dashed lines 
%represent jump discontinuities.}
}
\label{edge-break}
\end{figure}
\begin{figure}[h]
%\hspace{1.5in}
\includegraphics[width=1.5in]{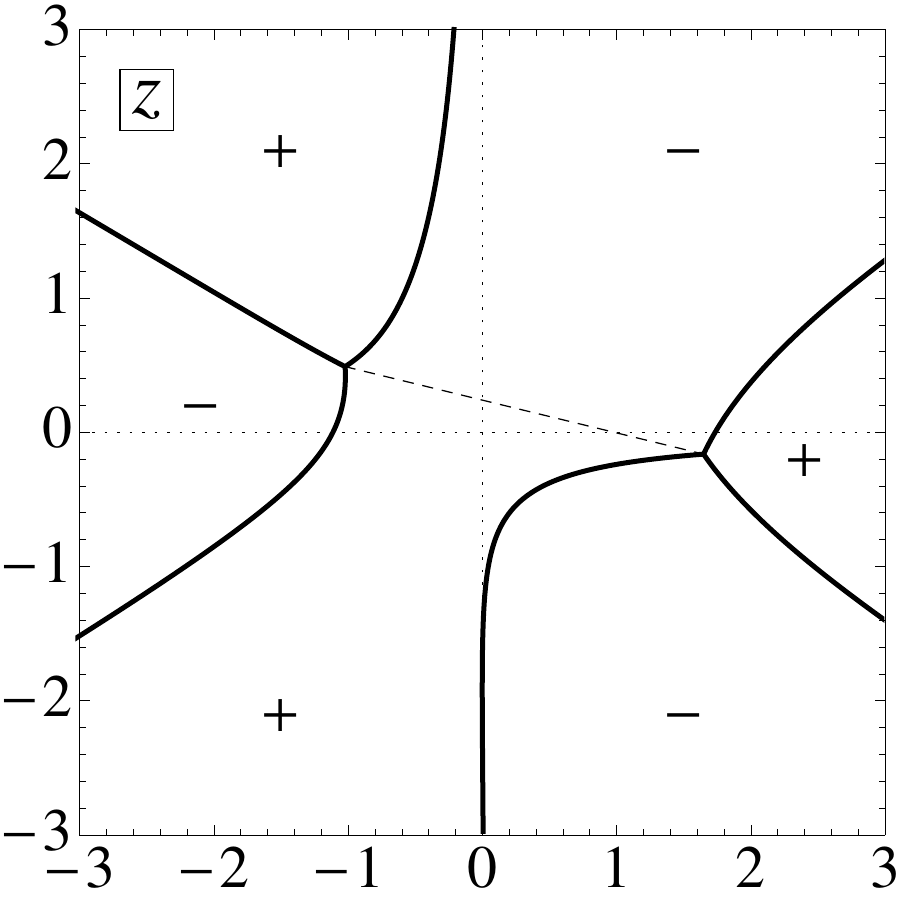}\\
\includegraphics[width=1.5in]{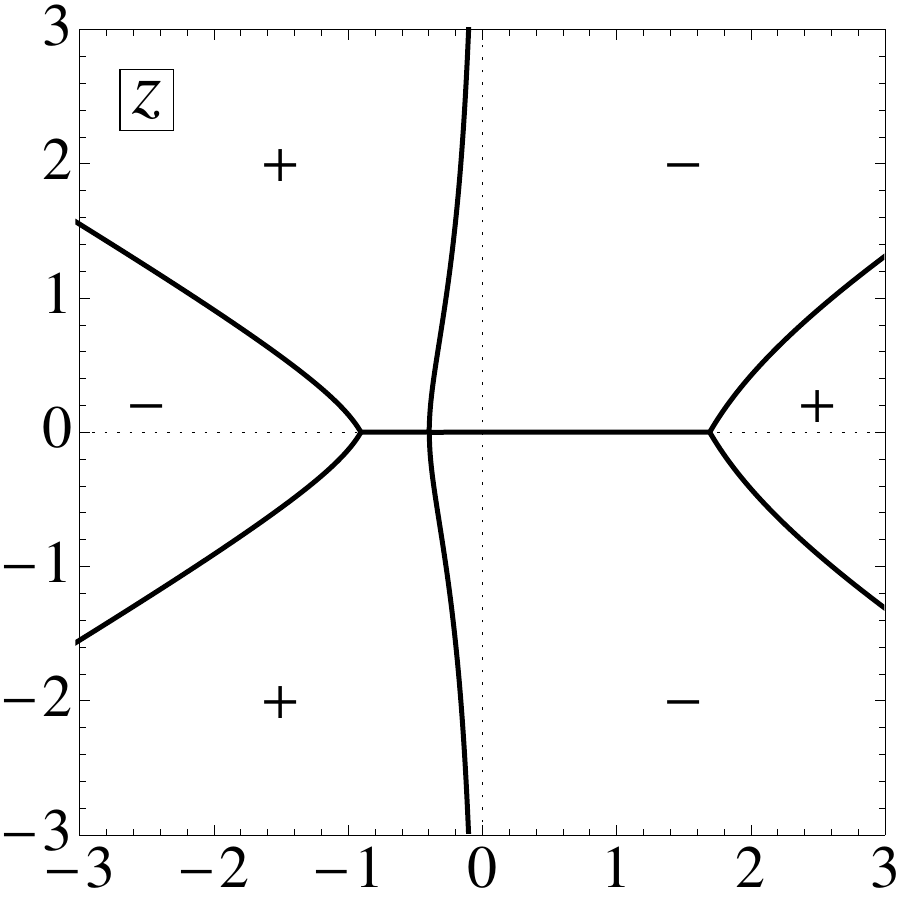}
\includegraphics[width=1.5in]{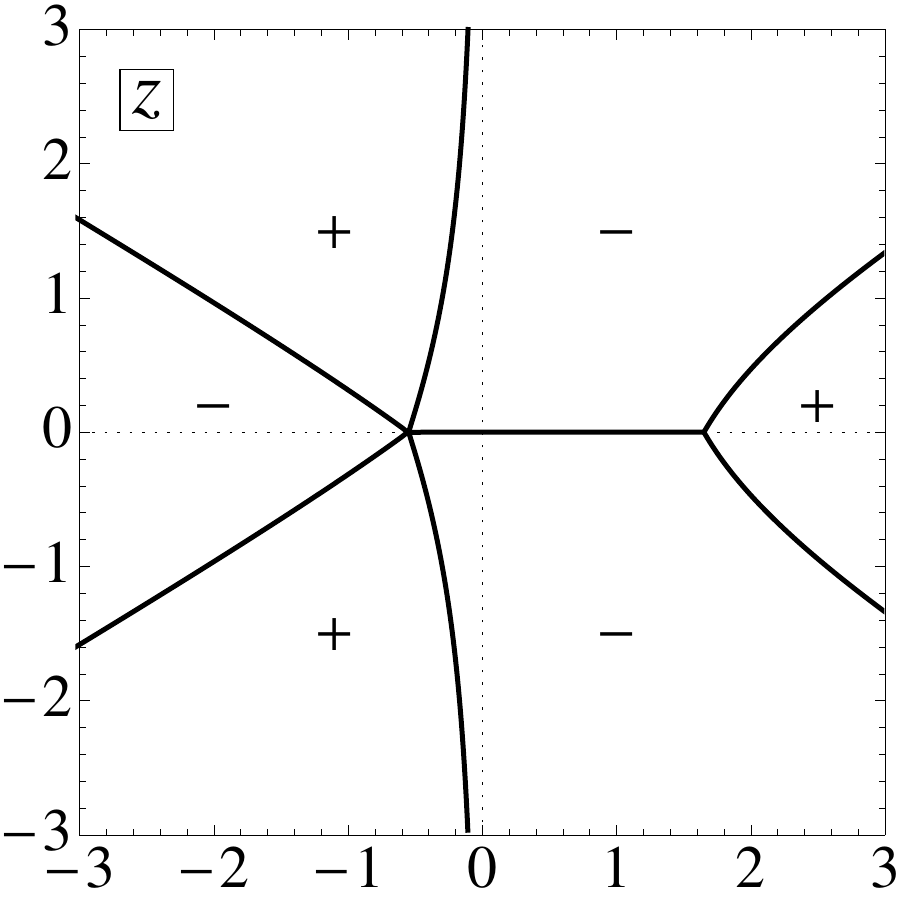}
\includegraphics[width=1.5in]{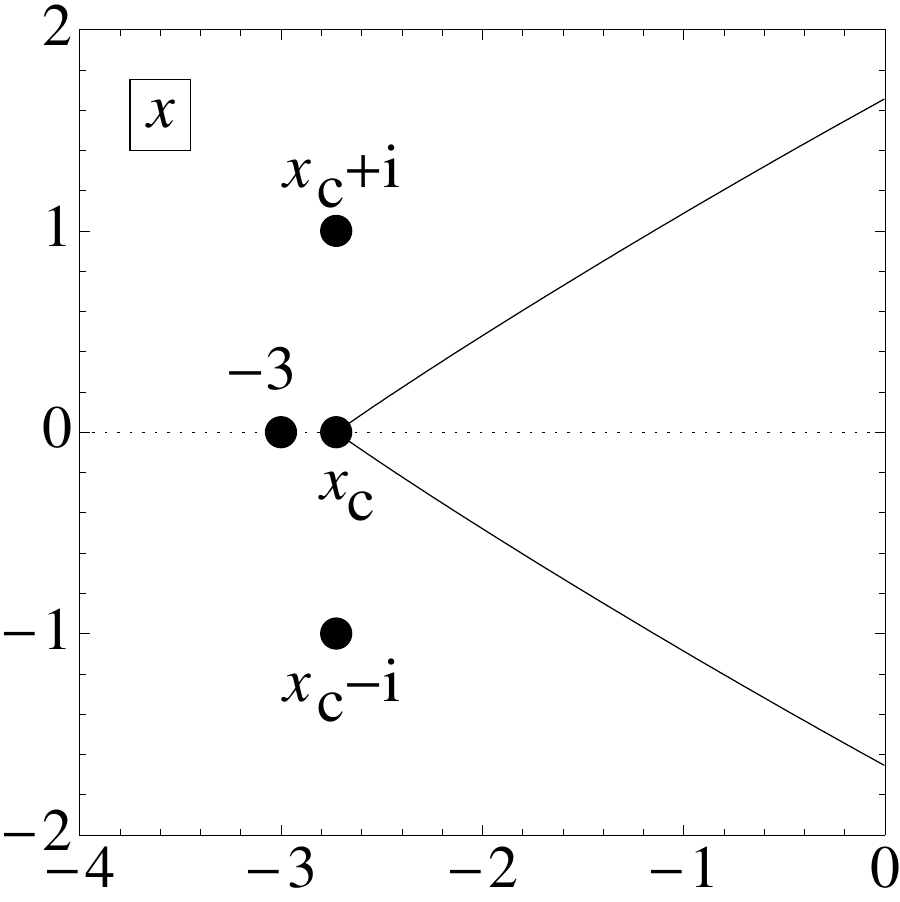}\\
\includegraphics[width=1.5in]{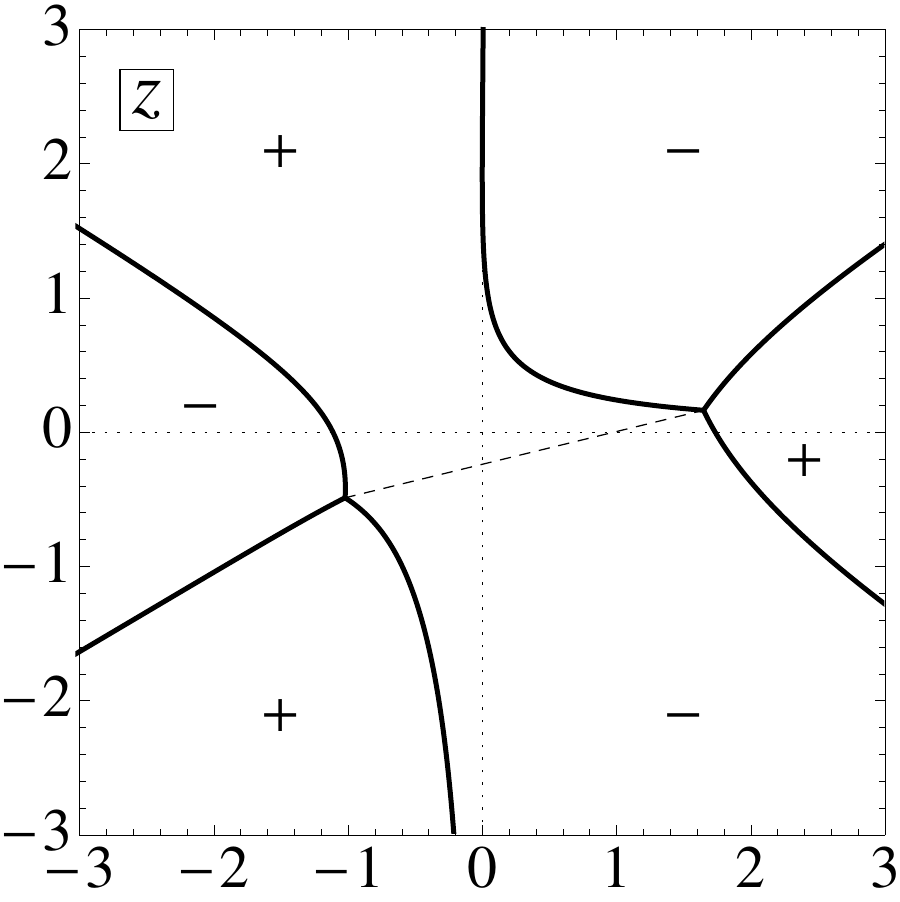}
\caption{
Signature charts for $\Re(2h(z;x)+\lambda(x))$ in the $z$-plane for $x=x_c+i$ (top panel), $x=x_c-i$ (bottom panel),
$x=-3$ (left panel), and $x=x_c$ (center panel).   The selected values of $x$ are plotted in relation to $\partial T$ in the right panel.  
%The breaking mechanism at a corner of the root region $T$.  
%First row:  The signature chart of $\text{Re}(2h(z;x)+\lambda(x))$ with 
%$x=x_c+i$ and a plot of chosen $x$ values.  Second row:  The signature 
%chart of $\text{Re}(2h(z;x)+\lambda(x))$ with $x=-3$, the signature 
%chart of $\text{Re}(2h(z;x)+\lambda(x))=\text{Re}(2H(z;x)+\lambda(x))$ with 
%$x=x_c\approx -2.72568$, and the signature chart of 
%$\text{Re}(2H(z;x)+\Lambda_0(x))$ with $x=-2$.  Third row:  The signature 
%chart of $\text{Re}(2h(z;x)+\lambda(x))$ with $x=x_c-i$.
%The solid lines are zero level curves while dotted lines are jump 
%discontinuities.}
%\textcolor{red}{$H$ is not defined.}
}
\label{corner-break}
\end{figure}

The function $\mathfrak{c}$ has a number of symmetries.  The following formulae arise from noting that upon rotation of $x$ by multiples of $2\pi/3$ radians, the pair $(a,b)$ can undergo permutation, so that comparing the formula for $\mathfrak{c}$ for $x$ in different sectors one has to include an additional contribution amounting to integrating along the branch cut $\Sigma$ from $b$ to $a$.  This contribution can be evaluated by residues taking into account the definitions of $a(x)$ and $b(x)$.  The formulae are:
\begin{equation}
\mathfrak{c}(x)=
\begin{cases}
\mathfrak{c}(xe^{-2\pi i/3})-i\pi,&\quad \displaystyle x\in\mathcal{S}_0e^{2\pi i/3},\\
\mathfrak{c}(xe^{-4\pi i/3}),&\quad\displaystyle x\in\mathcal{S}_0e^{-2\pi i/3}.
\end{cases}
\label{eq:c-rotate}
\end{equation}
Similarly, upon noting that for $x\in\mathcal{S}_0$ one has $S(x^*)=S(x)^*$ but $a(x^*)=b(x)^*$, one can derive the relation
\begin{equation}
\mathfrak{c}(x^*)=\mathfrak{c}(x)^*+i\pi,\quad x\in\mathcal{S}_0.
\label{eq:c-reflection}
\end{equation}
\begin{lemma}
The function $\mathfrak{d}$ defined as
\begin{equation}
\mathfrak{d}(x):=\mathfrak{c}(x)-\frac{i\pi}{2},\quad x\in\mathcal{S}_0
\label{eq:d-define}
\end{equation}
is univalent and Schwarz-symmetric in its domain of definition.  The conformal image of $\mathcal{S}_0$ under $\mathfrak{d}$ intersects the imaginary axis exactly in the segment with endpoints $\pm i\pi/2$ (see Figure~\ref{fig:I-plane}).
\label{lemma:d-univalent}
\end{lemma}
\begin{figure}[h]
\includegraphics{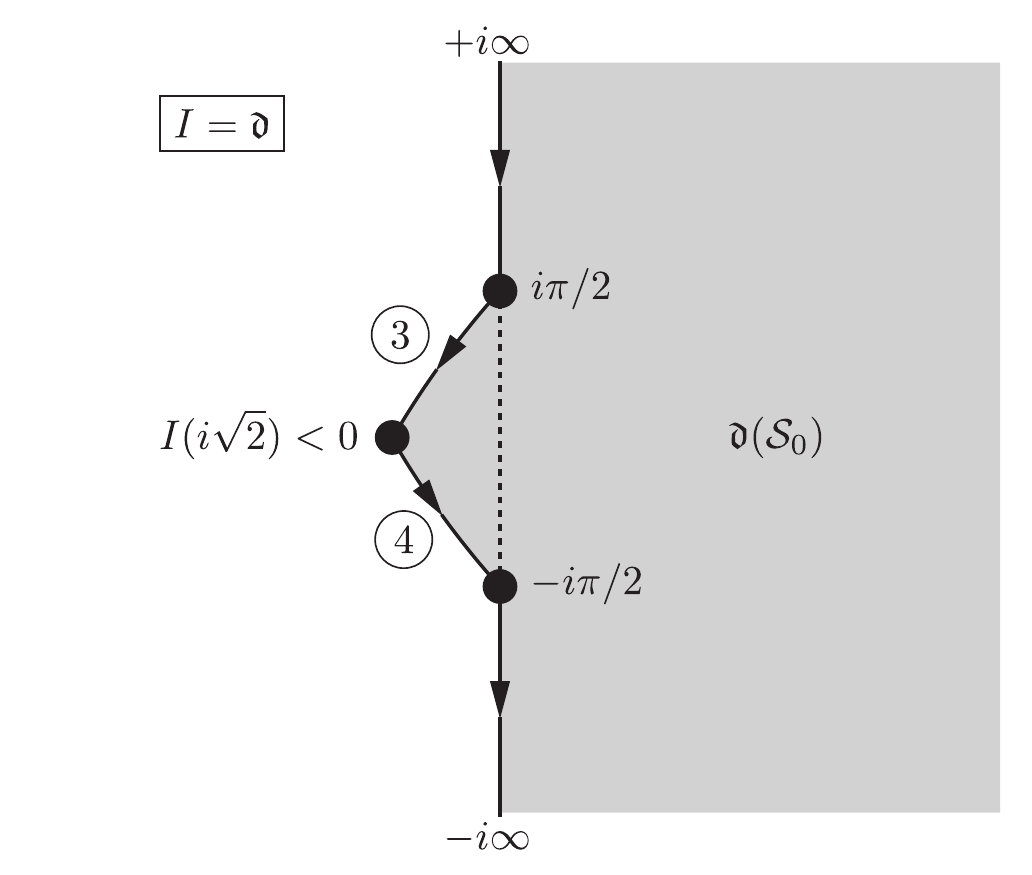}
\caption{The conformal image of $\mathcal{S}_0$ in the $\mathfrak{d}$-plane or, in 
the notation of Appendix \ref{sec:Appendix}, the conformal image of $\tau$ in the 
$I$-plane (see \eqref{proofs-I-def}).}
\label{fig:I-plane}
\end{figure}
The proof of Lemma~\ref{lemma:d-univalent} is given in \S\ref{sec:lemma:d-univalent}.
The symmetries \eqref{eq:c-rotate} and the definition \eqref{eq:d-define} show that the full locus of points $x$ where $\Re(\mathfrak{c}(x))=0$ can be obtained by considering the equation $\Re(\mathfrak{d}(x))=0$ for $x\in\mathcal{S}_0$ and then including rotations of this set by angles $\pm 2\pi/3$.  According to Lemma~\ref{lemma:d-univalent}, the solution of the equation $\Re(\mathfrak{d}(x))=0$ for $x\in\mathcal{S}_0$ is exactly the preimage under $\mathfrak{d}$ of the imaginary segment with endpoints $\pm i\pi/2$.  This is an analytic arc connecting the points $x_ce^{\pm 2\pi i/3}$.  This arc in the $x$-plane and its two rotations are what we call the three \emph{edges} of $T$.  
The edges join in pairs at three \emph{corners}, namely the points $x_c<0$ and $x_ce^{\pm 2\pi i}$.

\subsection{Opening angle of $\partial T$ near its corners}
\label{opening-angle-subsection}
To understand the nature of $\partial T$ near a corner, it suffices to consider $x$ near $x_c<0$, since 
$\partial T$ is invariant under rotations by $2\pi/3$ radians.  First, we use \eqref{cubic-equation}
to analyze $S(x)$ for $x$ near $x_c$.  This local analysis shows that $S(x)=s((x_c-x)^{1/2})$, where $s(\cdot)$ is an analytic function of its (small) argument and where $(x_c-x)^{1/2}$ denotes the principal branch, positive for $x<x_c$ and cut in the interval $x>x_c$.  Moreover, $s(0)=S_c:=(4/3)^{1/3}$ while $s'(0)=-2/3$.  From \eqref{eq:Delta-define}--\eqref{eq:a-b-define}, we then obtain corresponding series expansions for $a(x)$ and $b(x)$ in integer powers of $(x_c-x)^{1/2}$.  In particular, $a(x)=a_c+\mathcal{O}((x_c-x)^{1/2})$ and $b(x)=b_c+\mathcal{O}((x_c-x)^{1/2})$, where
\begin{equation}
a_c:=-\frac{1}{6^{1/3}} \quad\text{and}\quad b_c:=\left(\frac{9}{2}\right)^{1/3}.
\label{eq:ac-bc-define}
\end{equation}
Moreover, $a(x)+\tfrac{1}{2}S(x)=-(x_c-x)^{1/2}+\mathcal{O}(x_c-x)$.
Rescaling the integration variable $\zeta$ in \eqref{eq:c-function-define} by $\zeta=a(x)+(x_c-x)^{1/2}w$ then renders $\mathfrak{c}(x)$ in the form
\begin{equation}
\mathfrak{c}(x)=\frac{3}{2}(x_c-x)\int_0^{W(x)}(w-W(x))r(a(x)+(x_c-x)^{1/2}w;a(x),b(x))\,dw,
\end{equation}
where
\begin{equation}
W(x):=-(x_c-x)^{-1/2}\left(a(x)+\frac{1}{2}S(x)\right).
\end{equation}
From the local analysis of $S$, $a$, and $b$, we find that $W(x)$ is an analytic function of $(x_c-x)^{1/2}$ for $x$ near $x_c$ satisfying $W(x_c)=1$ (i.e., the apparent singularity at $x=x_c$ is removable).  Furthermore, uniformly for bounded $w$ one has in the limit $x\to x_c$:
\begin{equation}
r(a(x)+(x_c-x)^{1/2}w;a(x),b(x))=-i\,\mathrm{sgn}(\Im(x))(b_c-a_c)^{1/2}(x_c-x)^{1/4}w^{1/2}(1+\mathcal{O}((x_c-x)^{1/2}).
\end{equation}
The fact that we have two formulae depending on the sign of $\Im(x)$ is related to the fact that $\zc(x)$ crosses the branch cut $\Sigma$ of $r$ when $x<x_c$.  Therefore, 
\begin{equation}
\begin{split}
\mathfrak{c}(x)&=-\frac{3}{2}i\,\mathrm{sgn}(\Im(x))(b_c-a_c)^{1/2}(x_c-x)^{5/4}\int_0^1 (w-1)w^{1/2}\,dw + \mathcal{O}((x_c-x)^{7/4})\\
& = \frac{2}{5}i\,\mathrm{sgn}(\Im(x))(b_c-a_c)^{1/2}(x_c-x)^{5/4}+
\mathcal{O}((x_c-x)^{7/4}),\quad x\to x_c.
\end{split}
\end{equation}
The condition $\Re(\mathfrak{c}(x))=0$ with $0<|\arg(x_c-x)|<\pi$ then implies that $\arg(x_c-x)\to \pm 4\pi/5$ as $x\to x_c$.  This proves the following.
\begin{proposition}
The three corners of $\partial T$ are located at the points $x_c$, $e^{2\pi i/3}x_c$, and $e^{-2\pi i/3}x_c$ where $x_c$ is given by \eqref{eq:xc-define}.  At each corner, $T$ subtends an angle of $2\pi/5$ radians.
\label{prop:corner}
\end{proposition}
In particular, $\partial T$ is \emph{not} a Euclidean triangle.

\section{Analysis near an edge of the elliptic region $T$}
\label{section:edge}
In this section, we study the rational Painlev\'e-II functions for $x$ in the vicinity of a smooth point of $\partial T$, that is, along an edge of $T$.  By the symmetries \eqref{pp-pq-pu-pv-symmetries} it is sufficient to assume that the edge of interest subtends the sector $\mathcal{S}_0$.
%close to an edge of the 
%root region $T$, but away from a corner.  The case with $x$ near a corner 
%will be studied in Section \ref{section:cusp}.  Throughout this section 
%we will assume that $-\pi/3<\arg(x)<\pi/3$.  Our analysis follows the work of 
%Claeys and Grava on the so-called trailing edge of an oscillatory region in a 
%solution of the Korteweg-de Vries equation \cite{Claeys:2010}.  Using a local 
%parametrix constructed from Hermite polynomials, they showed that near an edge 
%the leading term of the semiclassical solution is a constant plus a train of 
%solitons.  

\subsection{The Positive-$x$ Configuration Riemann-Hilbert problem}
In the course of studying a librational-rotational transition point in the 
semiclassical sine-Gordon equation, the authors derived a Riemann-Hilbert 
problem well suited for studying the large-degree asymptotics of the 
rational Painlev\'e-II functions \cite{BuckinghamMcritical}.  In 
the companion paper 
\cite{Buckingham-Miller-rational-noncrit}, various transformations 
(i.e., contour deformations and introduction of the $g$-function defined by \eqref{g-formula}) 
dependent on $x$ were performed on this Riemann-Hilbert problem to facilitate 
the asymptotic analysis.  To analyze the rational Painlev\'e-II functions for $x$ near the edge in the sector $\mathcal{S}_0$, it is sufficient to 
start with the ``dressed'' Riemann-Hilbert problem referred to in 
\cite{Buckingham-Miller-rational-noncrit} as the Positive-$x$ Configuration 
problem.  In the companion paper this Riemann-Hilbert problem was used for $x\in\mathcal{S}_0\setminus\overline{T}$.
%$\{x:-\pi/3<\arg(x)<\pi/3\text{ and }x\in\mathbb{C}\backslash\overline{T}\}$.  
Here we will modify the analysis to allow $x$ to penetrate $\partial T$ by an distance proportional to $m^{-1}\log(m)$, assuming also that $x$ remains bounded away from the endpoints of the edge (corner points of $T$). 
The analytical issue that must be addressed is that 
uniform decay of the jump matrices to the identity is lost near the point $z=\zc(x)$.  A 
local parametrix will be inserted around this point to allow for the application of small norm theory (after some additional steps), ultimately leading to asymptotic formulae for 
the rational Painlev\'e-II functions for $x$ near the edge of $T$ in $\mathcal{S}_0$.  Significantly, these formulae display qualitatively different behavior than the corresponding formulae derived in \cite{Buckingham-Miller-rational-noncrit} assuming that $x$ lies outside of $T$.

We now define ${\bf O}(z)={\bf O}(z;x,\epsilon)$ as the solution to the 
following Riemann-Hilbert problem:
\begin{rhp}[The Positive-$x$ Configuration]
Fix a real number $x$ and an integer $m$ and recall $\epsilon$ defined in terms of $m$ by \eqref{epsilon}. 
%$\epsilon:=(m-\tfrac{1}{2})^{-1}$.  
Seek a matrix ${\bf O}(z;x,\epsilon)$ 
with the following properties:
\begin{itemize}
\item[]\textbf{Analyticity:}  ${\bf O}(z)$
is analytic in $z$ except along the contour 
$\Sigma^{({\bf O})}$ defined in Figure \ref{fig:O-jumps-edge},
from each region of analyticity it may be continued to
a slightly larger region except at the  points $a=a(x)$ and $b=b(x)$ given by \eqref{eq:a-b-define}, 
and in each region is H\"older continuous 
% with any exponent $\gamma\le 1$ 
up to the boundary in neighborhoods of $a$ and $b$. 
%\textcolor{red}{Check this.}
\item[]\textbf{Jump condition:}  The jump conditions satisfied by 
the matrix function ${\bf O}(z)$ are of the form 
${\bf O}_+(z)=
{\bf O}_-(z)
\mathbf{V}^{({\bf O})}(z)$, 
where the jump matrix $\mathbf{V}^{({\bf O})}(z)$ and contour 
orientations are as shown in Figure~\ref{fig:O-jumps-edge}.
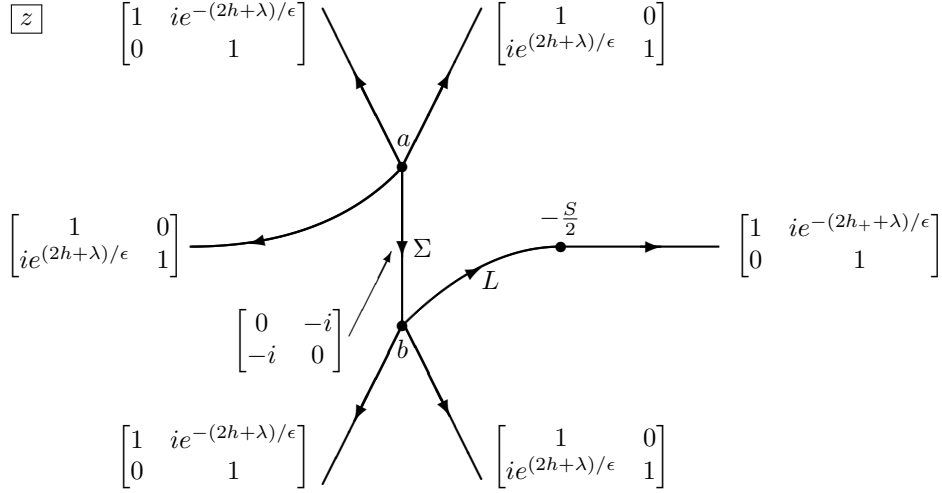
\begin{figure}[h]
\setlength{\unitlength}{2pt}
\begin{center}
\begin{picture}(100,100)(-50,-50)
\put(-84,42){\framebox{$z$}}
\put(-20,-17){\vector(1,2){8}}
\put(-10,15){\circle*{2}}
\put(-11,19){$a$}
\put(-10,-15){\circle*{2}}
\put(-11,-21){$b$}
\thicklines
\qbezier(-10,-15)(5,0)(20,0)
\put(20,0){\line(1,0){30}}
\put(1,-6){\vector(2,1){4}}
\put(35,0){\vector(1,0){4}}
\put(20,0){\circle*{2}}
\put(53,-1){$\bbm 1 & ie^{-(2h_++\lambda)/\epsilon} \\ 0 & 1 \ebm$}
\put(5,-8){$L$}
\put(16,4){$-\frac{S}{2}$}
\put(-10,15){\line(1,2){15}}
\put(-10,15){\vector(1,2){9}}
\put(7,39){$\bbm 1 & 0 \\ ie^{(2h+\lambda)/\epsilon} & 1 \ebm$}
\put(-10,-15){\line(0,1){30}}
\put(-10,15){\vector(0,-1){17}}
\put(-8,-2){$\Sigma$}
\put(-42,-19){$\bbm 0 & -i \\ -i & 0 \ebm$}
\put(-10,15){\line(-1,2){15}}
\put(-10,15){\vector(-1,2){9}}
\put(-64,39){$\bbm 1 & ie^{-(2h+\lambda)/\epsilon} \\ 0 & 1 \ebm$}
\qbezier(-10,15)(-24,0)(-50,0)
\put(-35,1.5){\vector(-4,-1){4}}
\put(-85,-1){$\bbm 1 & 0 \\ ie^{(2h+\lambda)/\epsilon} & 1 \ebm$}
\put(-10,-15){\line(-1,-2){15}}
\put(-10,-15){\vector(-1,-2){9}}
\put(-64,-41){$\bbm 1 & ie^{-(2h+\lambda)/\epsilon} \\ 0 & 1 \ebm$}
\put(-10,-14){\line(1,-2){15}}
\put(-10,-14){\vector(1,-2){9}}
\put(7,-41){$\bbm 1 & 0 \\ ie^{(2h+\lambda)/\epsilon} & 1 \ebm$}
\end{picture}
\end{center}
\caption{The jump matrices $\mathbf{V}^{(\mathbf{O})}(z;x,\epsilon)$ for 
$x$ on the positive real axis near the edge of the elliptic region $T$.  The topology of 
the jump contour $\Sigma^{(\mathbf{O})}$ is the same for any $x$ near $\partial T$ in the 
sector $\mathcal{S}_0$.  Recall that $h(z;x)$ and $\lambda(x)$ were defined in \S\ref{section:g-define}.
}
\label{fig:O-jumps-edge}
\end{figure}
\item[]\textbf{Normalization:}  The matrix 
${\bf O}(z)$ satisfies the condition
\begin{equation}
\lim_{z\to\infty}{\bf O}(z)=\mathbb{I},
\label{eq:Zindnorm}
\end{equation}
with the limit being uniform with respect to direction in each of the
six regions of analyticity.
\end{itemize}
\label{rhp:DSlocalII}
\end{rhp}

The rational Painlev\'e-II functions are recovered from ${\bf O}(z)$ by 
writing 
\eq
{\bf O}(z;x;\epsilon) = \mathbb{I} + \frac{{\bf O}_1(x;\epsilon)}{z} + \frac{{\bf O}_2(x;\epsilon)}{z^2} + \mathcal{O}\left(\frac{1}{z^3}\right),\quad z\to\infty,
\endeq
and using (see \cite[Section 3.6.3]{Buckingham-Miller-rational-noncrit}) 
\eq
\begin{split}
\label{eq:um-vm-pm-qm-ito-O-edge}
\epsilon^{(2+\epsilon)/(3\epsilon)}e^{-\lambda/\epsilon}\pu_m = O_{1,12}, \quad \epsilon^{-(2-\epsilon)/(3\epsilon)}e^{\lambda/\epsilon}\pv_m &= O_{1,21}, \\
\epsilon^{1/3}\pp_m = O_{1,22} - \frac{O_{2,12}}{O_{1,12}}, \quad \epsilon^{1/3}\pq_m  = -O_{1,11} +& \frac{O_{2,21}}{O_{1,21}}.
\end{split}
\endeq

\subsection{The outer parametrix for $\mathbf{O}(z)$}
\label{section:outer-parametrix-edge}
%To analyze the rational Painlev\'e II solutions near the edge of the elliptic region, 
%we start with the function 
%\begin{equation}
%\label{O-def-gen0}
%\mathbf{O}(z;x,\epsilon):=e^{-\lambda\sigma_3/2\epsilon}\mathbf{N}(z;x,\epsilon)e^{-g(z)\sigma_3/\epsilon}e^{\lambda\sigma_3/2\epsilon}
%\end{equation}
%defined in the companion paper \cite{Buckingham-Miller-rational-noncrit}; specifically, we use the \emph{positive-$x$ configuration} described in \S3.4 of that paper.
%The edge $\partial T$ is characterized by the pinch-off at the critical point $z=\zc$ given by \eqref{eq:edge-zc}
%of the domain containing the contour $L$ in which $\Re(2h(z;x)+\lambda(x))>0$ 
%holds.  
%This breaking mechanism is illustrated in Figure \ref{edge-break}.  
Recall the critical point $\zc(x)$ of $h$ defined by \eqref{eq:edge-zc} and the function $r(\cdot;a,b)$.
We introduce the related notation
\begin{equation}
\rc=\rc(x):=r(\zc(x);a(x),b(x)),
\label{eq:edge-rc}
\end{equation}
which defines $\rc$ as an analytic function of $x$ in the sector $\mathcal{S}_0$, and record the useful formula 
\eq
\rc(x)^2 = S(x)^2-\frac{1}{4}\Delta(x)^2.
\endeq
We assume that the contour $L$ passes through the critical point $\zc$
(see Figure \ref{fig:O-jumps-edge}), and that with the exception of the arc $\Sigma$, the jump contour $\Sigma^{(\mathbf{O})}$ locally agrees exactly with the steepest-descent directions for $\Re(2h(z;x)+\lambda(x))$ near the points $z=a(x)$, $z=b(x)$, and $z=\zc(x)$.  
Assuming that $x$ does not penetrate the elliptic region $T$ too far
%is not too far from the exterior region $\mathbb{C}\setminus T$ 
(this will be made precise eventually), the asymptotic behavior of the jump matrix $\mathbf{V}^{(\mathbf{O})}(z)$ in the limit $\epsilon\downarrow 0$ is clear from the signature charts shown in Figure~\ref{edge-break} with the exception of the neighborhood of the critical point $z=\zc$.
Allowing for a certain type of singular behavior at this point (which will be repaired soon with the installation of an appropriate  local parametrix) as well as at the two endpoints $z=a$ and $z=b$ of the straight-line contour $\Sigma$ (which will be repaired with the installation of standard Airy parametrices), we are led to the following model problem:
%During the construction of the Hermite polynomial parametrix, we will see 
%that we must allow an algebraic growth in the outer parametrix near $z=-S(x)/2$.  
%We pose the following outer model problem:
\begin{rhp}[Outer parametrix for $x$ near an edge] 
Given $K\in\mathbb{Z}$, find a $2\times 2$ matrix-valued function 
$\dot{\mathbf{O}}^{(\mathrm{out},K)}(z;x)$ satisfying the following conditions:
\begin{itemize}
\item[]\textbf{Analyticity:}  $\dot{\mathbf{O}}^{(\mathrm{out},K)}(z;x)$ is analytic for 
$z\notin\Sigma\cup\{\zc(x)\}$ and takes H\"older-continuous boundary values on $\Sigma$ with 
the exception of its endpoints $a(x)$ and $b(x)$.
\item[]\textbf{Jump condition:}  $\dot{\mathbf{O}}_+^{(\mathrm{out},K)}(z;x) = \dot{\mathbf{O}}_-^{(\mathrm{out},K)}(z;x)\bbm 0 & -i \\ -i & 0 \ebm$ for $z\in\Sigma$.
\item[]\textbf{Singularities:}  $\dot{\mathbf{O}}^{(\mathrm{out},K)}(z)=\mathcal{O}((z-a(x))^{-1/4}(z-b(x))^{-1/4})$ for $z$ in a neighborhood of $\Sigma$, and 
$\dot{\mathbf{O}}^{(\mathrm{out},K)}(z;x)(z-\zc(x))^{-K\sigma_3}$ is analytic at $z=\zc(x)$.
\item[]\textbf{Normalization:}  $\dot{\mathbf{O}}^{(\mathrm{out},K)}(z;x) = \mathbb{I} + \mathcal{O}(z^{-1})$ as $z\to\infty$. 
\end{itemize}
\label{rhp:outer-model-edge}
\end{rhp}
It is easy to check that the solution of this problem is unique if it exists, and it necessarily has unit determinant.
For $K=0$, the solution of this problem is obtained as follows.  Let $\beta(z)=\beta(z;x)$ be the function analytic for $z\in\mathbb{C}\setminus\Sigma$ that satisfies the conditions
\begin{equation}
\beta(z;x)^4=\frac{z-a(x)}{z-b(x)}\quad\text{and}\quad\beta(\infty;x)=1.
\label{eq:edge-beta-define}
\end{equation}
For future reference, we note that $\beta$ satisfies the identities
\begin{equation}
\beta(z)^2+\beta(z)^{-2}=\frac{2z-S}{r(z)}\quad\text{and}\quad
\beta(z)^2-\beta(z)^{-2}=\frac{\Delta}{r(z)},
\label{eq:edge-beta-identities}
\end{equation}
and we introduce the related notation
\begin{equation}
\bc=\bc(x):=\beta(\zc(x);x).
\label{eq:edge-betac}
\end{equation}

Using the fact that $\beta_+(z)=-i\beta_-(z)$ holds for $z\in\Sigma$, the formula
\eq
\label{Odot-gen0}
\Odot^{(\mathrm{out},0)}(z;x) = \mathbf{M}\beta(z;x)^{\sigma_3}\mathbf{M}^{-1}=\frac{1}{2}\bbm \beta(z;x)+\beta(z;x)^{-1} & \beta(z;x)-\beta(z;x)^{-1} \\ \beta(z;x)-\beta(z;x)^{-1} & \beta(z;x)+\beta(z;x)^{-1} \ebm,\quad \mathbf{M}:=\frac{1}{\sqrt{2}}\bbm 1 & 1\\ 1 & -1\ebm,
\endeq
yields the (unique) solution of Riemann-Hilbert Problem~\ref{rhp:outer-model-edge} for $K=0$, as is easily checked.  To solve Riemann-Hilbert Problem~\ref{rhp:outer-model-edge} for $K\neq 0$, we modify the formula for $K=0$ as follows (see \cite[Section 3.2]{Bertola:2009}).  Let $t:\mathbb{C}\setminus\Sigma\to\mathbb{C}$ be the Joukowski map defined by
\begin{equation}
t(z)=t(z;x):=\frac{2}{i\Delta(x)}\left(z-\frac{S(x)}{2}+r(z;a(x),b(x))\right),\quad z\in\mathbb{C}\setminus\Sigma.
\label{eq:edge-Joukowski}
\end{equation}
We introduce the related notation
\begin{equation}
\tc=\tc(x):=t(\zc(x);x).
\label{eq:edge-tc}
\end{equation}
Note that $\tc>0$ holds for $x>0$.
The function $t=t(z)$ satisfies the quadratic equation
\begin{equation}
t-t^{-1}=\frac{4}{i\Delta}\left(z-\frac{S}{2}\right),
\end{equation}
the other solution of which is $-1/t(z)$.  The function $t$ is a univalent ($1$-to-$1$) conformal map of the slit domain $\mathbb{C}\setminus\Sigma$ onto the exterior of the unit circle, and its boundary values satisfy $t_+(z)t_-(z)=-1$ for $z\in\Sigma$.  It follows that the function $\mathcal{T}:\mathbb{C}\setminus\Sigma\to\mathbb{C}$ defined by
\begin{equation}
\mathcal{T}(z)=\mathcal{T}(z;x):=i\frac{\tc(x)-t(z;x)}{1+\tc(x)t(z;x)}
\label{eq:edge-correction-factor}
\end{equation}
is analytic in its domain of definition and has only one simple zero at the critical point $z=\zc(x)$,
and its boundary values satisfy the jump condition $\mathcal{T}_+(z)\mathcal{T}_-(z)=1$ for $z\in\Sigma$.  Also, as $z\to\infty$, $\mathcal{T}(z)\to\mathcal{T}(\infty;x)= -i\tc(x)^{-1}$.  The solution of Riemann-Hilbert Problem~\ref{rhp:outer-model-edge} for general $K\in\mathbb{Z}$ is then given by
\begin{equation}
\dot{\mathbf{O}}^{(\mathrm{out},K)}(z;x)=\mathcal{T}(\infty;x)^{-K\sigma_3}\dot{\mathbf{O}}^{(\mathrm{out},0)}(z;x)\mathcal{T}(z;x)^{K\sigma_3}.
\label{eq:edge-Outer-parametrix}
\end{equation}
It is easy to check that, as a consequence of the jump condition satisfied by $\mathcal{T}$, the jump condition for $\dot{\mathbf{O}}^{(\mathrm{out},K)}(z)$ across $\Sigma$ holds for general $K\in\mathbb{Z}$ just as it does for $K=0$.  The additional factors present for $K\neq 0$ are bounded and have bounded inverses near $z=a,b$, so the nature of the admitted singularities near these points is unchanged.  Finally, in a neighborhood of $z=\zc$ one can write $\mathcal{T}(z;x)=(z-\zc(x))\widetilde{\mathcal{T}}(z;x)$, where $\widetilde{\mathcal{T}}$ is analytic and non-vanishing at the critical point $z=\zc$, which confirms the desired singular behavior of $\dot{\mathbf{O}}^{(\mathrm{out},K)}(z)$ near this point.  In fact, 
\begin{equation}
\widetilde{\mathcal{T}}(\zc(x);x)=\mathcal{T}'(\zc(x),x)=\frac{\Delta(x)}{4\rc(x)^2} \neq 0.
\label{eq:edge-calT-prime}
\end{equation}

For later use, we record here two useful identities involving the functions $\tc(x)$, $\rc(x)$, $S(x)$, and $\Delta(x)$, which are easily proved by direct calculation.
\begin{lemma}
The following identities hold:
\begin{equation}
\begin{split}
-\tc(x)^2\left[\frac{i\Delta(x)}{4}-\frac{\rc(x)}{i\Delta(x)}(\rc(x)+S(x))\right]&=\frac{i\Delta(x)}{4}-\frac{\rc(x)}{i\Delta(x)}(\rc(x)-S(x)),\\
\tc(x)^2(\rc(x)+S(x))&=\rc(x)-S(x).
\end{split}
\end{equation}
\label{lemma-identities}
\end{lemma}

\subsection{The inner (Airy) parametrices near $z=a(x)$ and $z=b(x)$}  The construction of inner parametrices near the points $z=a$ and $z=b$ in terms of Airy functions
is quite standard.  The full details can be found in \cite[Sections 3.6.1--2]{Buckingham-Miller-rational-noncrit} (following the case of the ``Positive-$x$ Configuration''), with the only change being that the outer parametrix $\dot{\mathbf{O}}^{(\mathrm{out})}(z)$ has to be generalized from the special case of $K=0$ to $\dot{\mathbf{O}}^{(\mathrm{out},K)}(z)$ as defined in \S\ref{section:outer-parametrix-edge}.  This implies that the corresponding parametrices constructed for $z$ near $a$ and $b$ will also depend on $K$, and we reflect this dependence by denoting the parametrices as $\dot{\mathbf{O}}^{(a,K)}(z)$ and $\dot{\mathbf{O}}^{(b,K)}(z)$, respectively.

The parametrices so-defined have the following properties.  Let $\mathbb{D}_a$ and $\mathbb{D}_b$ denote $\epsilon$-independent open disks containing the points $z=a$ and $z=b$ respectively.  Then $\dot{\mathbf{O}}^{(a,K)}(z)$ (respectively, $\dot{\mathbf{O}}^{(b,K)}(z)$) is defined and analytic for $z\in\mathbb{D}_a\setminus\Sigma^{(\mathbf{O})}$ (respectively, for
$z\in\mathbb{D}_b\setminus\Sigma^{(\mathbf{O})}$), and it satisfies exactly the same jump conditions along the three arcs of $\Sigma^{(\mathbf{O})}$ within its disk of definition as does $\mathbf{O}(z)$.  A further crucial property is that $\dot{\mathbf{O}}^{(a,K)}(z)$ and $\dot{\mathbf{O}}^{(b,K)}(z)$ both achieve a good match with the outer parametrix $\dot{\mathbf{O}}^{(\mathrm{out},K)}(z)$ on the boundaries of their respective disks in the sense that
\begin{equation}
\dot{\mathbf{O}}^{(a,K)}(z)\dot{\mathbf{O}}^{(\mathrm{out},K)}(z)^{-1}=\mathbb{I}+\mathcal{O}(\epsilon),\quad z\in\partial \mathbb{D}_a
\label{eq:edge-Airy-a-match}
\end{equation}
and
\begin{equation}
\dot{\mathbf{O}}^{(b,K)}(z)\dot{\mathbf{O}}^{(\mathrm{out},K)}(z)^{-1}=\mathbb{I}+\mathcal{O}(\epsilon),\quad z\in\partial \mathbb{D}_b,
\label{eq:edge-Airy-b-match}
\end{equation}
with the estimates holding uniformly with respect to $z$ on the corresponding circle, provided that the points $a$, $b$, and $\zc$ remain bounded away from one another.  This latter condition holds if $x$ is bounded away from the corner points of $T$; we will therefore assume for the duration of \S\ref{section:edge} that for some $\sigma>0$, $x\in\mathcal{S}_\sigma$ (see \eqref{eq:edge-sector-define}).

\subsection{The inner (Hermite) parametrix near $z=\zc(x)$}  Following the same basic methodology as in the construction of the Airy parametrices, we will here define a matrix $\dot{\mathbf{O}}^{(\zc,K)}(z)$ in a neighborhood of the critical point $z=\zc(x)$ that (i) locally solves the jump conditions for $\mathbf{O}(z)$ exactly and (ii) matches as well as possible the outer parametrix $\dot{\mathbf{O}}^{(\mathrm{out},K)}(z)$ for an appropriate choice of $K\in\mathbb{Z}$.  

The choice of $K\in\mathbb{Z}$ will depend on $x$ near $\partial T$ in a way that will be made clear shortly.  Recall the function $\mathfrak{c}$ defined by \eqref{eq:c-function-define}.
Since the upper limit of integration lies
on the branch cut $L$ of $h$, the Fundamental Theorem of Calculus can be applied if the proper boundary value of $h$ is used, yielding
\begin{equation}
\mathfrak{c}(x)=h_+(\zc(x);x)-h(a(x);x)=h_+(\zc(x);x)+\tfrac{1}{2}\lambda(x).
\end{equation}
Therefore, $2\mathfrak{c}(x)$ is the value of the function $2h_+(z;x)+\lambda(x)$ at the critical point
$z=\zc(x)$.  In \cite{Buckingham-Miller-rational-noncrit}, it is shown that if $\Re(2\mathfrak{c}(x))>\log(m)/m$, then no parametrix near the critical point is required, and the outer parametrix may be taken in the case $K=0$; this results in an error estimate proportional to $\epsilon$.  
%\textcolor{red}{Probably omit from here\dots} It is easy to generalize the error analysis in \cite[Section 3.6]{Buckingham-Miller-rational-noncrit} to allow the same asymptotic formula to hold under the weaker condition $\Re(2\mathfrak{c}(x))\ge p\log(m)/m$ for $0<p<1$ at the cost of a larger error estimate proportional to $\epsilon^p$ (or $m^{-p}$).  To be near the edge therefore means that $\Re(2\mathfrak{c}(x))< p\log(m)/m$ for some $p\in (0,1]$.  \textcolor{red}{To here\dots} 
While it is therefore harmless to allow $\Re(2\mathfrak{c}(x))$ to become large and positive, the use of different values of the integer $K$ will be needed to allow it to become negative, and with this device we will be able to admit $x$ to penetrate $T$ such that $\Re(2\mathfrak{c}(x))<0$ is of size $\mathcal{O}(\log(m)/m)$.  
%We will also require that $\Re(2\mathfrak{c}(x))=\mathcal{O}(\log(m)/m)$.  \textcolor{red}{Really?  Don't we allow it to become unbounded in the positive direction with $K=0$?}  
The value of $K\ge 0$ will be gauged from the magnitude of $\Re(2\mathfrak{c}(x))$ compared with $\log(m)/m$.  This type of situation, in which a parametrix involving an integer-valued parameter tied to various auxiliary continuous parameters (here, $x$) plays an important role in the asymptotic analysis, has appeared at least three times in the literature on integrable nonlinear waves:  the analysis of Claeys and Grava \cite{Claeys:2010} on the ``solitonic'' edge of the Whitham oscillation zone for the Korteweg-de Vries equation in the small-dispersion limit, the analysis of Bertola and Tovbis \cite{BertolaT-edge} of semiclassical solutions to the focusing nonlinear Schr\"odinger equation near an analogous transition point, and our own work \cite{BuckinghamMcritical} on a certain universal feature of semiclassical solutions of the sine-Gordon equation.  In the former two cases \cite{BertolaT-edge,Claeys:2010}, the parametrix involved is very similar to the one that appears here.
%, while in \cite{BuckinghamMcritical}, the parametrix is different as one needs the doubly-infinite sequence of rational Painlev\'e-II functions to construct it rather than the singly-infinite sequence of Hermite polynomials.  (The reader may find it interesting that the analysis of the rational Painlev\'e-II functions near the edge of the elliptic region has some key elements in common with a quite different asymptotic problem in which the same functions arise in a similar situation.  \textcolor{red}{Keep this strange remark?  Or omit it --- and perhaps also the reference to the sine-Gordon paper which doesn't fit as well as the others?})  
Parametrices involving integer-valued parameters have also appeared in the theory of orthogonal polynomials and random matrices; see \cite{Bertola:2009}.

In order to construct the local parametrix $\dot{\mathbf{O}}^{(\zc,K)}(z)$, it is convenient to introduce %another 
a conformal local coordinate $W$ taking a neighborhood of the critical point in the $z$-plane to a target domain near the origin in the $W$-plane.  The conformal map $W=W(z)=W(z;x)$ is defined for $z$ near $\zc(x)$ by the relation
\begin{equation}
2h_+(z;x)+\lambda(x)=2\mathfrak{c}(x)+W(z;x)^2
\label{eq:edge-W-definition}
\end{equation}
(here $h_+(z;x)$ is analytically continued through $L$ to a neighborhood of $z=\zc(x)$ by the relation $h_+-h_-=2\pi i$) and the condition that $W$ is real and strictly increasing along $L$.  That these conditions serve to define $W$ as an analytic and univalent function follows because $2h_+(z;x)+\lambda(x)-2\mathfrak{c}(x)$ vanishes to second order at $z=\zc(x)$.  Note that $W(\zc(x);x)=0$.  

In terms of $W$, the jump condition along $L$ for $\mathbf{O}(z)$ reads
\begin{equation}
\mathbf{O}_+(z)=\mathbf{O}_-(z)\begin{bmatrix}1 & ie^{-2\mathfrak{c}(x)/\epsilon}e^{-W^2/\epsilon}\\
0 & 1\end{bmatrix},\quad W(z)\in\mathbb{R},
\end{equation}
or, written differently,
\begin{equation}
\mathbf{O}_+(z)e^{i\pi\sigma_3/4}e^{-\mathfrak{c}(x)\sigma_3/\epsilon}=
\mathbf{O}_-(z)e^{i\pi\sigma_3/4}e^{-\mathfrak{c}(x)\sigma_3/\epsilon}\begin{bmatrix}1 & e^{-W^2/\epsilon}\\0 & 1\end{bmatrix}.
\label{eq:O-crit-jump}
\end{equation}
At the same time, the outer parametrix $\dot{\mathbf{O}}^{(\mathrm{out},K)}(z)$ can be represented near $z=\zc(x)$ in the form
\begin{equation}
\dot{\mathbf{O}}^{(\mathrm{out},K)}(z)=\mathbf{F}^{(\zc,K)}(z)W(z;x)^{K\sigma_3},
\label{eq:Outer-crit}
\end{equation}
which serves to define $\mathbf{F}^{(\zc,K)}(z)$ as an analytic matrix function of $z$ in a neighborhood of  $z=\zc(x)$ that has unit determinant and is bounded independently of $\epsilon>0$.  The analogue of the matrix $\mathbf{A}(\zeta)$ relevant to the construction of Airy parametrices and defined in \cite[Appendix A]{Buckingham-Miller-rational-noncrit} is in this case the matrix function $\mathbf{H}^{(K)}(\zeta)$ defined by the following conditions:
\begin{rhp}[Basic Hermite parametrix]
Given $K=0,1,2,3,\dots$, find a $2\times 2$ matrix function $\mathbf{H}^{(K)}(\zeta)$ with the following properties:
\begin{itemize}
\item[]\textbf{Analyticity:}  $\mathbf{H}^{(K)}(\zeta)$ is analytic for $\zeta\in\mathbb{C}\setminus\mathbb{R}$, and $\mathbf{H}^{(K)}$ takes H\"older-continuous boundary values $\mathbf{H}^{(K)}_\pm(\zeta)$ on the real axis from the half-planes $\mathbb{C}_\pm$.
\item[]\textbf{Jump condition:} The boundary values are related by the jump condition
\begin{equation}
\mathbf{H}^{(K)}_+(\zeta)=\mathbf{H}^{(K)}_-(\zeta)\begin{bmatrix}1 & e^{-\zeta^2}\\0 & 1\end{bmatrix},\quad \zeta\in\mathbb{R}.
\label{eq:Hermite-jump}
\end{equation}
\item[]\textbf{Normalization:}  As $\zeta\to\infty$ in any direction (including tangentially to the real axis), 
\begin{equation}
\mathbf{H}^{(K)}(\zeta)\zeta^{-K\sigma_3}=\mathbb{I}+\mathcal{O}(\zeta^{-1}).
\label{eq:Hermite-norm}
\end{equation}
The four constants implicit in the bound of the matrix-valued term $\mathcal{O}(\zeta^{-1})$ in \eqref{eq:Hermite-norm} may depend on $K$.
\end{itemize}
\label{rhp-Hermite}
\end{rhp}
This Riemann-Hilbert problem can be traced back to the work of Fokas, Its, and Kitaev \cite{FIK}, and it is solved (uniquely) in terms of the Hermite polynomials $\{\mathfrak{H}_n(\zeta)\}_{n=0}^\infty$ defined by the orthonormality conditions
\begin{equation}
\int_{\mathbb{R}}\mathfrak{H}_j(\zeta)\mathfrak{H}_k(\zeta)e^{-\zeta^2}\,d\zeta=\delta_{jk}
\end{equation}
and the assertion that $\mathfrak{H}_n$ is a polynomial of degree exactly $n$:  $\mathfrak{H}_n(\zeta)=h_n\zeta^n + \text{lower order terms,}$ with $h_n>0$.  It turns out that the leading coefficients are given explicitly by
\begin{equation}
h_n:=\frac{2^{n/2}}{\pi^{1/4}\sqrt{n!}}.
\end{equation}
The solution of Riemann-Hilbert Problem~\ref{rhp-Hermite} is, explicitly,
\begin{equation}
\mathbf{H}^{(0)}(\zeta)=\begin{bmatrix}1 & \displaystyle\frac{1}{2\pi i}\int_\mathbb{R}\frac{e^{-s^2}\,ds}{s-\zeta}\\0 & 1\end{bmatrix},\;
\mathbf{H}^{(K)}(\zeta)=\begin{bmatrix}
\displaystyle \frac{1}{h_K}\mathfrak{H}_K(\zeta) & \displaystyle\frac{1}{2\pi ih_K}\int_\mathbb{R}\frac{\mathfrak{H}_K(s)e^{-s^2}\,ds}{s-\zeta}\\
-2\pi ih_{K-1}\mathfrak{H}_{K-1}(\zeta) & \displaystyle -h_{K-1}\int_\mathbb{R}
\frac{\mathfrak{H}_{K-1}(s)e^{-s^2}\,ds}{s-\zeta}\end{bmatrix}, \; K\ge 1.
\label{eq:HermiteH}
\end{equation}
Let $\mathbb{D}_{\zc}$ be a small $\epsilon$-independent disk containing the critical point $z=\zc(x)$.
The local parametrix $\dot{\mathbf{O}}^{(\zc,K)}(z)$ is then defined by the formula
\begin{equation}
\dot{\mathbf{O}}^{(\zc,K)}(z):=\mathbf{F}^{(\zc,K)}(z)\epsilon^{K\sigma_3/2}e^{-\mathfrak{c}(x)\sigma_3/\epsilon}e^{i\pi\sigma_3/4}\mathbf{H}^{(K)}(\epsilon^{-1/2}W(z;x))e^{-i\pi\sigma_3/4}
e^{\mathfrak{c}(x)\sigma_3/\epsilon},\quad z\in\mathbb{D}_{\zc}.
\end{equation}
Comparing \eqref{eq:O-crit-jump} and \eqref{eq:Hermite-jump} shows that $\dot{\mathbf{O}}^{(\zc,K)}(z)$ satisfies exactly the same jump condition along $L$ near $z=\zc(x)$ as does $\mathbf{O}(z)$ itself.  When $z\in\partial\mathbb{D}_{\zc}$, $W(z;x)$ is bounded away from zero, and hence $\epsilon^{-1/2}W(z;x)$ is large.  From \eqref{eq:Outer-crit} and \eqref{eq:Hermite-norm} it then follows that
\begin{multline}
\dot{\mathbf{O}}^{(\zc,K)}(z)\dot{\mathbf{O}}^{(\mathrm{out},K)}(z)^{-1}=\\
\mathbf{F}^{(\zc,K)}(z)\epsilon^{K\sigma_3/2}e^{-\mathfrak{c}(x)\sigma_3/\epsilon}\left(\mathbb{I}+\mathcal{O}(\epsilon^{1/2})\right)e^{\mathfrak{c}(x)\sigma_3/\epsilon}\epsilon^{-K\sigma_3/2}
\mathbf{F}^{(\zc,K)}(z)^{-1},\quad z\in\partial \mathbb{D}_{\zc}.
\label{eq:mismatch-critical}
\end{multline}

Now we can explain how the integer $K\ge 0$ is chosen.  To obtain the best possible match between $\dot{\mathbf{O}}^{(\zc,K)}(z)$ and $\dot{\mathbf{O}}^{(\mathrm{out},K)}(z)$ on $\partial\mathbb{D}_{\zc}$, it is clear from \eqref{eq:mismatch-critical} that, given $x$, $K$ should be chosen so that $|\epsilon^{K/2}e^{-\mathfrak{c}(x)/\epsilon}|$ is as close to $1$ as possible.  Indeed, if $|\epsilon^{K/2}e^{-\mathfrak{c}(x)/\epsilon}|$ is asymptotically large or small, then one or the other of the off-diagonal elements of the error term $\mathcal{O}(\epsilon^{1/2})$ in \eqref{eq:mismatch-critical} will become amplified.  We therefore choose
\begin{equation}
K=K(x;\epsilon):=\left\llbracket -\frac{\Re(2\mathfrak{c}(x))}{\epsilon\log(\epsilon^{-1})}\right\rrbracket\ge 0,\quad\text{whenever}\quad \Re(2\mathfrak{c}(x))\le\frac{1}{2}\epsilon\log(\epsilon^{-1}),
\label{eq:K-define}
\end{equation}
where $\llbracket z\rrbracket$ denotes the nearest integer to $z\in\mathbb{R}$ defined as $z+\tfrac{1}{2}$ in the ambiguous case that $z$ is an odd half-integer.  For reasons that will be clear shortly, as a special case, we want to define also  
\begin{equation}
K(x;\epsilon):=0\quad\text{whenever}\quad\Re(2\mathfrak{c}(x))>\frac{1}{2}\epsilon\log(\epsilon^{-1}).
\label{eq:K-define-zero}
\end{equation}
\emph{For the duration of \S\ref{section:edge}, the symbol $K$ will always stand for the integer-valued function of $x$ and $\epsilon$ defined in \eqref{eq:K-define}--\eqref{eq:K-define-zero}.}
%\textcolor{red}{Probably omit from here\dots} The condition $K\ge 0$ then corresponds to the inequality
%\begin{equation}
%\Re(2\mathfrak{c}(x))\le\frac{1}{2}\epsilon\log(\epsilon^{-1})=\frac{1}{2}\frac{\log(m)}{m}\left(1+\mathcal{O}(m^{-1})\right).
%\label{eq:edge-inequality}
%\end{equation}
%This inequality defines a domain near the boundary of the elliptic region $T$ that is complementary to the region analyzed in \cite[Section 3]{Buckingham-Miller-rational-noncrit}, provided that the error estimate of $\mathcal{O}(m^{-1})$ recorded in \cite[Theorem 1]{Buckingham-Miller-rational-noncrit} is replaced by $\mathcal{O}(m^{-1/2})$.  In particular, the inequality \eqref{eq:edge-inequality} includes the boundary arc of $\partial T$ defined by $\Re(\mathfrak{c}(x))=0$.  \textcolor{red}{To here\dots}

\subsection{The global parametrix for $\mathbf{O}(z)$ and analysis of the preliminary error}
Define the global parametrix for $\mathbf{O}(z)$ as follows:
\begin{equation}
\dot{\mathbf{O}}(z):=\begin{cases}
\dot{\mathbf{O}}^{(\mathrm{out},K)}(z),&\quad z\in\mathbb{C}\setminus(\Sigma\cup\overline{\mathbb{D}}_a\cup\overline{\mathbb{D}}_b\cup\overline{\mathbb{D}}_{\zc}),\\
\dot{\mathbf{O}}^{(a,K)}(z),&\quad z\in\mathbb{D}_a,\\
\dot{\mathbf{O}}^{(b,K)}(z),&\quad z\in\mathbb{D}_b,\\
\dot{\mathbf{O}}^{(\zc,K)}(z),&\quad z\in\mathbb{D}_{\zc},
\end{cases}
\end{equation}
where the integer $K=K(x;\epsilon)\ge 0$ is defined in \eqref{eq:K-define}--\eqref{eq:K-define-zero}.
To gauge the accuracy of approximating $\mathbf{O}(z)$ with $\dot{\mathbf{O}}(z)$, we consider the (preliminary) error defined by
\begin{equation}
\mathbf{P}(z):=\mathbf{O}(z)\dot{\mathbf{O}}(z)^{-1}
\end{equation}
wherever both factors are defined.  This domain of definition is $\mathbb{C}\setminus\Sigma^{(\mathbf{P})}$, where $\Sigma^{(\mathbf{P})}$ is a contour consisting of:
\begin{itemize}
\item the disk boundaries $\partial\mathbb{D}_a$, $\partial\mathbb{D}_b$, and $\partial\mathbb{D}_{\zc}$, all of which are taken to be oriented in the negative (clockwise) direction, and
\item the arcs of $\Sigma^{(\mathbf{O})}$ lying outside of these disks, with the exception of the straight line segment $\Sigma$ connecting $z=a$ and $z=b$.  These arcs are taken with the same orientation as when they are considered as part of $\Sigma^{(\mathbf{O})}$.
\end{itemize}
The arc $\Sigma$ and the arcs of $\Sigma^{(\mathbf{O})}$ within the disks are not part of $\Sigma^{(\mathbf{P})}$ because $\mathbf{O}(z)$ and $\dot{\mathbf{O}}(z)$ share exactly the same jump conditions across each of these arcs.  On the other hand, the disk boundaries are included in $\Sigma^{(\mathbf{P})}$ because the local parametrices do not match the outer parametrix exactly.

Assume that $x\in\mathcal{S}_\sigma$ and that, for some fixed $M>0$, $K(x;\epsilon)\le M$.  The preliminary error matrix $\mathbf{P}(z)$ satisfies the conditions of a Riemann-Hilbert problem
relative to the jump contour $\Sigma^{(\mathbf{P})}$, with a jump condition of the form $\mathbf{P}_+(z)=\mathbf{P}_-(z)\mathbf{V}^{(\mathbf{P})}(z)$ on each oriented arc of $\Sigma^{(\mathbf{P})}$, and with the normalization condition $\mathbf{P}(z)\to\mathbb{I}$ as $z\to\infty$.  The jump matrix $\mathbf{V}^{(\mathbf{P})}(z)$ is defined on the arcs of $\Sigma^{(\mathbf{P})}$ as follows:
\begin{itemize}
\item For $z\in\Sigma^{(\mathbf{P})}\setminus(\partial\mathbb{D}_a\cup\partial\mathbb{D}_b\cup\partial\mathbb{D}_{\zc})$, we have $\mathbf{V}^{(\mathbf{P})}(z)=\dot{\mathbf{O}}^{(\mathrm{out},K)}(z)\mathbf{V}^{(\mathbf{O})}(z)\dot{\mathbf{O}}^{(\mathrm{out},K)}(z)^{-1}$.
The outer parametrix and its inverse are uniformly bounded independently of $\epsilon>0$
for $z$ outside of the three disks, and $\mathbf{V}^{(\mathbf{O})}(z)-\mathbb{I}$ is exponentially small in the limit $\epsilon\downarrow 0$ (in both the $L^\infty$ and weighted $L^1$ sense, with weight $z^2$, as can be easily shown).  Therefore $\mathbf{V}^{(\mathbf{P})}(z)-\mathbb{I}$ obeys a similar exponential estimate for such $z$.  These estimates are uniform for $x\in\mathcal{S}_\sigma$ with $K(x;\epsilon)\le M$.
\item For $z\in\partial\mathbb{D}_a$ (respectively, $z\in\partial\mathbb{D}_b$), we have $\mathbf{V}^{(\mathbf{P})}(z)=\dot{\mathbf{O}}^{(a,K)}(z)\dot{\mathbf{O}}^{(\mathrm{out},K)}(z)^{-1}$ (respectively, $\mathbf{V}^{(\mathbf{P})}(z)=\dot{\mathbf{O}}^{(b,K)}(z)\dot{\mathbf{O}}^{(\mathrm{out},K)}(z)^{-1}$), and therefore from \eqref{eq:edge-Airy-a-match} and \eqref{eq:edge-Airy-b-match} we have $\mathbf{V}^{(\mathbf{P})}(z)-\mathbb{I}=\mathcal{O}(\epsilon)$ in the $L^\infty$ sense on both disk boundaries, uniformly for $x\in\mathcal{S}_\sigma$ with $K(x;\epsilon)\le M$.
\item For $z\in\partial\mathbb{D}_{\zc}$, we have $\mathbf{V}^{(\mathbf{P})}(z)=\dot{\mathbf{O}}^{(\zc,K)}(z)\dot{\mathbf{O}}^{(\mathrm{out},K)}(z)^{-1}$, which is given by
\eqref{eq:mismatch-critical}.  Let $f(x;\epsilon)$ be defined by
\begin{equation}
f=f(x;\epsilon):=\epsilon^{K(x;\epsilon)/2}e^{-\mathfrak{c}(x)/\epsilon}.
\label{eq:edge-f-define}
\end{equation}
Because $|f(x;\epsilon)|$ can occupy the full range of values $\epsilon^{1/4}\le |f(x;\epsilon)|<\epsilon^{-1/4}$, $\mathbf{V}^{(\mathbf{P})}(z)-\mathbb{I}$ is not generally small for $z\in\partial\mathbb{D}_{\zc}$ although it is bounded\footnote{The given range of values of $|f(x;\epsilon)|$ holds for $K\ge 1$.  When $K=0$ we also allow $\Re(\mathfrak{c}(x))$ to become arbitrarily positive, and this means that while we still have the upper bound $|f(x;\epsilon)|<\epsilon^{-1/4}$, the lower bound can be exponentially small.  However, when $K=0$, the factor $\mathbb{I}+\mathcal{O}(\epsilon^{1/2})$ in \eqref{eq:mismatch-critical} is upper-triangular, and therefore the lower bound is immaterial.  We still conclude that for $z\in\partial\mathbb{D}_{\zc}$, $\mathbf{V}^{(\mathbf{P})}-\mathbb{I}$ is bounded although not generally small.}. 
\end{itemize}
The boundary values of $\mathbf{P}$ are required to be taken in the classical (H\"older-continuous) sense.

\subsubsection{Formulation of a parametrix for $\mathbf{P}(z)$}
Due to the contribution to the jump matrix $\mathbf{V}^{(\mathbf{P})}(z)$ from points $z\in\partial\mathbb{D}_{\zc}$, the preliminary error $\mathbf{P}(z)$ does not generally satisfy a Riemann-Hilbert problem of small-norm type (see \cite[Appendix B]{Buckingham-Miller-rational-noncrit} for a self-contained description of such problems and their solution).  Following the approach of \cite{BuckinghamMcritical}, we proceed by building a \emph{parametrix for the error}.  The term $\mathcal{O}(\epsilon^{1/2})$ in \eqref{eq:mismatch-critical} is shorthand for the matrix
\begin{equation}
\mathbf{G}(z)=\mathbf{G}(z;x,\epsilon):=e^{i\pi\sigma_3/4}\mathbf{H}^{(K(x;\epsilon))}(\epsilon^{-1/2}W(z;x))(\epsilon^{-1/2}W(z;x))^{-K(x;\epsilon)\sigma_3}e^{-i\pi\sigma_3/4}-\mathbb{I},
\end{equation}
and hence we have
\begin{equation}
\mathbf{V}^{(\mathbf{P})}(z)=\mathbf{F}^{(\zc,K)}(z)\left[\mathbb{I}+f^{\sigma_3}\mathbf{G}(z)f^{-\sigma_3}\right]\mathbf{F}^{(\zc,K)}(z)^{-1},\quad z\in\partial\mathbb{D}_{\zc}.
\end{equation}
By expanding the formulae \eqref{eq:HermiteH} for large $\zeta$, one obtains that
\begin{equation}
\mathbf{G}(z)=\begin{bmatrix}0 & \displaystyle -\frac{1}{2\pi h_{K}^2}\\
-2\pi h_{K-1}^2 & 0\end{bmatrix}\frac{\epsilon^{1/2}}{W} + 
\begin{bmatrix}\tfrac{1}{4}K(K-1) & 0\\0 & -\tfrac{1}{4}K(K+1)\end{bmatrix}\frac{\epsilon}{W^2} + \mathcal{O}\left(\frac{\epsilon^{3/2}}{W^3}\right),
\end{equation}
where $W=W(z;x)$,  and by convention $h_{-1}:=0$.  Since $W$ is bounded away from zero on $\partial\mathbb{D}_{\zc}$, the inequalities $\epsilon^{1/4}\le | f|<\epsilon^{-1/4}$ valid for $K\ge 1$ imply that 
\begin{equation}
\mathbf{V}^{(\mathbf{P})}(z)=\mathbf{F}^{(\zc,K)}(z)\left(\mathbb{I}+\begin{bmatrix}
0 & \displaystyle -\frac{f^2}{2\pi h_K^2}\\\displaystyle -\frac{2\pi h_{K-1}^2}{f^2} & 0\end{bmatrix}
\frac{\epsilon^{1/2}}{W} + \mathcal{O}(\epsilon)\right)\mathbf{F}^{(\zc,K)}(z)^{-1},\quad z\in\partial\mathbb{D}_{\zc},\quad K\ge 1,
\label{eq:edge-VP-Kge1}
\end{equation}
with $W=W(z;x)$, $f=f(x;\epsilon)$, and $K=K(x;\epsilon)$.  For $K=0$ we have only the upper bound $|f|<\epsilon^{-1/4}$ but $h_{K-1}=0$, and therefore
\begin{equation}
\mathbf{V}^{(\mathbf{P})}(z)=\mathbf{F}^{(\zc,0)}(z)\left(\mathbb{I}+\begin{bmatrix}
0 &\displaystyle-\frac{f^2}{2\sqrt{\pi}}\\0 & 0\end{bmatrix} \frac{\epsilon^{1/2}}{W}+\mathcal{O}(\epsilon^{3/2})\right)\mathbf{F}^{(\zc,0)}(z)^{-1},\quad z\in\partial\mathbb{D}_{\zc},\quad
K=0.
\label{eq:edge-VP-K0}
\end{equation}
Furthermore, the expression \eqref{eq:edge-VP-Kge1} can be simplified in two different ways depending on whether $|f|\le 1$ or $|f|\ge 1$:
\begin{equation}
\mathbf{V}^{(\mathbf{P})}(z)=\mathbf{F}^{(\zc,K)}(z)\left(\begin{bmatrix} 1 & 0\\
\displaystyle-\frac{2\pi \epsilon^{1/2}h_{K-1}^2}{f^2W} & 1\end{bmatrix} + \mathcal{O}(\epsilon^{1/2})\right)\mathbf{F}^{(\zc,K)}(z)^{-1},\; z\in\partial \mathbb{D}_{\zc},\;
 \epsilon^{1/2}\le |f|^2\le 1,\;K\ge 1,
\label{eq:VP-lower-exact}
\end{equation}
and
\begin{equation}
\mathbf{V}^{(\mathbf{P})}(z)=\mathbf{F}^{(\zc,K)}(z)\left(\begin{bmatrix} 1 & \displaystyle
-\frac{\epsilon^{1/2}f^2}{2\pi h_K^2W}\\
0 & 1\end{bmatrix} + \mathcal{O}(\epsilon^{1/2})\right)\mathbf{F}^{(\zc,K)}(z)^{-1},\; z\in\partial\mathbb{D}_{\zc},\; 1\le |f|^2<\epsilon^{-1/2},\;K\ge 1.
\label{eq:VP-upper-exact}
\end{equation}
By keeping only the leading terms in \eqref{eq:edge-VP-K0}--\eqref{eq:VP-upper-exact}, 
and ignoring all other jump discontinuities of $\mathbf{P}(z)$, we arrive at an explicitly solvable model for $\mathbf{P}(z)$:
\begin{rhp}[Parametrix for the error]
Let $s=\pm$ be a sign determined by $x$ and $\epsilon$ as follows:
\begin{equation}
s=s(x;\epsilon):=\begin{cases}\mathrm{sgn}(\log(|f(x;\epsilon)|)),&\quad K(x;\epsilon)\ge 1\\
+,&\quad K(x;\epsilon)=0.
\end{cases}
\label{eq:edge-s-define}
\end{equation}
Seek a $2\times 2$ matrix function $\dot{\mathbf{P}}(z)=\dot{\mathbf{P}}(z;x,\epsilon)$ with the following properties:
\begin{itemize}
\item[]\textbf{Analyticity:}  $\dot{\mathbf{P}}(z)$ is analytic for $z\in\mathbb{C}\setminus\partial\mathbb{D}_{\zc}$, and takes H\"older-continuous boundary values $\dot{\mathbf{P}}_+(z)$ (respectively, $\dot{\mathbf{P}}_-(z)$) on $\partial\mathbb{D}_{\zc}$ from outside (respectively, inside).
\item[]\textbf{Jump condition:}  The boundary values are related for $z\in\partial\mathbb{D}_{\zc}$ by $\dot{\mathbf{P}}_+(z)=\dot{\mathbf{P}}_-(z)\dot{\mathbf{V}}^{(\mathbf{P})}(z)$, where
\begin{equation}
\dot{\mathbf{V}}^{(\mathbf{P})}(z):=
\mathbf{F}^{(\zc,K)}(z)\left(\mathbb{I} -\frac{Q_s}{W(z)}\sigma_s\right)\mathbf{F}^{(\zc,K)}(z)^{-1}.
\end{equation}
Here, the Pauli matrices $\sigma_\pm$ are defined by \eqref{eq:Pauli-matrices} 
%\begin{equation}
%\sigma_-:=\begin{bmatrix}0 & 0\\1 & 0\end{bmatrix}\quad\text{and}\quad
%\sigma_+:=\begin{bmatrix}0 & 1\\0 & 0\end{bmatrix},
%\label{eq:edge-sigma-plus-minus}
%\end{equation}
and the constants (i.e., independent of $z$) $Q_\pm=Q_\pm(x;\epsilon)$ are defined by
\begin{equation}
Q_-:=\frac{2\pi\epsilon^{1/2}h_{K-1}^2}{f^2}\quad\text{and}\quad
Q_+:=\frac{\epsilon^{1/2}f^2}{2\pi h_K^2}.
\label{eq:edge-Q-define}
\end{equation}
\item[]\textbf{Normalization:}  $\dot{\mathbf{P}}(z)\to\mathbb{I}$ as $z\to\infty$.
\end{itemize}
%The top (``$+$'') sign is taken for $K=0$, while if $K\ge 1$ we take $\pm:=\mathrm{sgn}(\log(|f|))$.  (Recall that the non-negative integer $K$ is defined in terms of  $x$ and $\epsilon$ by \eqref{eq:K-define}--\eqref{eq:K-define-zero}.)
\label{rhp:edge-error-parametrix}
\end{rhp}
Note that $Q_s$ is bounded as $\epsilon\to 0$ as a consequence of the definition \eqref{eq:K-define}--\eqref{eq:K-define-zero} of $K(x;\epsilon)$, the definition \eqref{eq:edge-f-define} of $f(x;\epsilon)$, and the definition \eqref{eq:edge-s-define} of the sign $s$.  We also have the important identity
\begin{equation}
Q_+Q_-=\frac{1}{2}\epsilon K.
\label{eq:edge-Qplus-Qminus-product}
\end{equation}
For later use, we let $\overline{s}=\overline{s}(x;\epsilon)$ denote the sign opposite to $s$.  While $Q_s$ is merely bounded, $Q_{\overline{s}}=\mathcal{O}(\epsilon^{1/2})$ (and in the case that $K(x;\epsilon)=0$, $Q_{\overline{s}}=0$).
\subsubsection{Solution of Riemann-Hilbert Problem~\ref{rhp:edge-error-parametrix}}
%\textcolor{red}{Probably omit from here\dots} We observe immediately a special case:  if $K=0$ but $\epsilon^{1/2}\le |f|^2\le 1$, then 
%$Q_-=0$ and the solution is simply $\dot{\mathbf{P}}(z)=\mathbb{I}$. This situation corresponds to the inequalities $0\le\Re(2\mathfrak{c}(x))\le\tfrac{1}{2}\epsilon\log(\epsilon^{-1})$, i.e., the thin strip exterior to the elliptic region $T$ not covered by the results in our earlier paper \cite{Buckingham-Miller-rational-noncrit} with an error estimate of size $\mathcal{O}(\epsilon^{1/2})$.  \textcolor{red}{To here\dots}
Observe that the jump matrix $\dot{\mathbf{V}}^{(\mathbf{P})}(z)$ has a meromorphic continuation to the interior of the disk $\mathbb{D}_{\zc}$, the only singularity of which is a simple pole at $z=\zc(x)$ coming from the simple zero of $W$.  Based on this observation, we introduce the auxiliary unknown $\ddot{\mathbf{P}}(z)$ by setting
\begin{equation}
\ddot{\mathbf{P}}(z):=\begin{cases}\dot{\mathbf{P}}(z),&\quad z\in\mathbb{C}\setminus\overline{\mathbb{D}}_{\zc},\\
\dot{\mathbf{P}}(z)\dot{\mathbf{V}}^{(\mathbf{P})}(z),&\quad z\in\mathbb{D}_{\zc}\setminus\{\zc(x)\}.
\end{cases}
\label{eq:dotP-ddotP}
\end{equation}
It follows from the jump condition satisfied by $\dot{\mathbf{P}}(z)$ that $\ddot{\mathbf{P}}(z)$ admits a common analytic continuation to $\partial\mathbb{D}_{\zc}$ from either side, and hence may be regarded as a meromorphic function on the whole complex plane $\mathbb{C}$ with at worst a simple pole at the point $z=\zc$.  The normalization condition on $\dot{\mathbf{P}}(z)$ implies that $\ddot{\mathbf{P}}(\infty)=\mathbb{I}$, and therefore $\ddot{\mathbf{P}}(z)$ necessarily has the form
\begin{equation}
\ddot{\mathbf{P}}(z)=\mathbb{I} + (z-\zc(x))^{-1}\mathbf{B}
\label{eq:ddotP-formula}
\end{equation}
for some constant matrix $\mathbf{B}=\mathbf{B}(x;\epsilon)$ to be determined.  To find $\mathbf{B}$, we note that, since $\dot{\mathbf{P}}(z)$ is analytic at $z=\zc$, 
\begin{equation}
\left(\mathbb{I}+(z-\zc)^{-1}\mathbf{B}\right)\dot{\mathbf{V}}^{(\mathbf{P})}(z)^{-1}=
\ddot{\mathbf{P}}(z)\dot{\mathbf{V}}^{(\mathbf{P})}(z)^{-1}=\dot{\mathbf{P}}(z)=\mathcal{O}(1),\quad z\to\zc.
\label{eq:edge-analytic-condition}
\end{equation}
But, by direct calculation, the left-hand side is
\begin{equation}
\begin{split}
\left(\mathbb{I}+(z-\zc)^{-1}\mathbf{B}\right)\dot{\mathbf{V}}^{(\mathbf{P})}(z)^{-1}&=
\left(\mathbb{I}+(z-\zc)^{-1}\mathbf{B}\right)\mathbf{F}^{(\zc,K)}(z)\left(\mathbb{I}+\frac{Q_s}{W(z)}\sigma_s\right)\mathbf{F}^{(\zc,K)}(z)^{-1}\\ &=
(z-\zc)^{-2}\left[\rho_0 Q_s\mathbf{B}\mathbf{F}_0\sigma_s\mathbf{F}_0^{-1}\right] \\
&{}\quad +
(z-\zc)^{-1}\Big[\rho_0 Q_s\mathbf{F}_0\sigma_s\mathbf{F}_0^{-1}\\
&{}\quad\quad+\mathbf{B}\left(\mathbb{I}+\rho_1 Q_s \mathbf{F}_0\sigma_s\mathbf{F}_0^{-1}+\rho_0 Q_s\mathbf{F}_1\sigma_s\mathbf{F}_0^{-1}-\rho_0 Q_s\mathbf{F}_0\sigma_s\mathbf{F}_0^{-1}\mathbf{F}_1\mathbf{F}_0^{-1}\right)
\Big] \\
&\quad {}+ \mathcal{O}(1)
\end{split}
\label{eq:edge-LHS-expand}
\end{equation}
in the limit $z\to \zc$, where we recall that $s=\pm$ is the sign defined by \eqref{eq:edge-s-define} and where $\rho_0$ and $\rho_1$ (respectively, $\mathbf{F}_0$ and $\mathbf{F}_1$) are the first two coefficients in the convergent power series expansion of $1/W(z)$ (respectively, $\mathbf{F}^{(\zc,K)}(z)$) about $z=\zc$:
\begin{equation}
\frac{1}{W(z)}=\rho_0(z-\zc)^{-1}+\rho_1+\mathcal{O}(z-\zc)\quad\text{and}\quad
\mathbf{F}^{(\zc,K)}(z)=\mathbf{F}_0 + (z-\zc)\mathbf{F}_1 + \mathcal{O}((z-\zc)^2).
\end{equation}
Therefore, comparing \eqref{eq:edge-analytic-condition} and \eqref{eq:edge-LHS-expand}, the matrix $\mathbf{B}$ is required to satisfy the equations (assuming $Q_\pm\neq 0$)
\begin{equation}
\mathbf{B}\mathbf{F}_0\sigma_\pm=\mathbf{0}\quad\text{and}\quad
\mathbf{B}\mathbf{F}_0\left(\mathbb{I}+\rho_0 Q_\pm\mathbf{F}_0^{-1}\mathbf{F}_1\sigma_\pm\right)+
\rho_0 Q_\pm\mathbf{F}_0\sigma_\pm=\mathbf{0}.
\end{equation}
These matrix equations are easily solved by separating the columns and taking into account that $\sigma_\pm$ has only one nonzero column.  For example, in the case that $s=-$, the first equation amounts to the condition that $\mathbf{B}\mathbf{F}_0=(\mathbf{v},\mathbf{0})$, that is, the second column of $\mathbf{B}\mathbf{F}_0$ vanishes.  The second column of the second equation is then an exact identity, while the first column yields a $2\times 2$ linear system on the two elements of $\mathbf{v}$.  The result of applying this procedure (also in the other case that $s=+$) is the explicit formula
\begin{equation}
\mathbf{B}=-\frac{\rho_0 Q_s\phi_s}{1+\rho_0 Q_s \phi_s}\cdot\frac{1}{\phi_s}\mathbf{F}_0\sigma_s\mathbf{F}_0^{-1},
\label{eq:A-formula}
\end{equation}
where 
\begin{equation}
\phi_+:=(\mathbf{F}_0^{-1}\mathbf{F}_1)_{21}\quad\text{and}\quad
\phi_-:=(\mathbf{F}_0^{-1}\mathbf{F}_1)_{12}.
\end{equation}
This completes the construction of $\ddot{\mathbf{P}}(z)$ via the formula \eqref{eq:ddotP-formula}, and hence of $\dot{\mathbf{P}}(z)$ solving Riemann-Hilbert Problem~\ref{rhp:edge-error-parametrix}
by means of \eqref{eq:dotP-ddotP}.  Note that the solution exists as long as the  denominator in \eqref{eq:A-formula} is nonzero.

Now $\rho_0=W'(\zc)^{-1}$, and to calculate the derivative we differentiate \eqref{eq:edge-W-definition} twice with respect to $z$ and set $z=\zc$.  Since $h_+'(z;x)=\tfrac{3}{2}(z-\zc)r(z)$,
it follows that $W'(\zc)^2=\tfrac{3}{2}\rc$.  The correct sign of the square root to take to obtain $W'(\zc)$ and hence $\rho_0$ can be determined by analyzing the case that $x$ is a positive real number near $\partial T$.  In this case $\rc$ is positive real, and since $W'(\zc)$ should also be positive real so that $W$ is real and increasing along $L$ (which may be taken to coincide with the real axis locally near the critical point), we see that the correct choice is the positive square root for $x>0$ analytically continued to the sector $\mathcal{S}_0$.  Thus, 
\begin{equation}
W'(\zc)=\sqrt{\frac{3}{2}}\rc^{1/2}\quad\text{and hence}\quad
\rho_0=\sqrt{\frac{2}{3}}\rc^{-1/2}.
\label{eq:edge-W-prime}
\end{equation}
These formulae hold true also for complex $x$, with the interpretation that the power functions are principal branches.  
Next, we calculate $\mathbf{F}_0$ and the off-diagonal elements of $\mathbf{F}_0^{-1}\mathbf{F}_1$ to obtain $\phi_\pm$.  According to \eqref{Odot-gen0}, \eqref{eq:edge-Outer-parametrix}, and \eqref{eq:Outer-crit}, we have
\begin{equation}
\mathbf{F}^{(\zc,K)}(z):=\mathcal{T}(\infty)^{-K\sigma_3}\frac{1}{2}\begin{bmatrix}
\beta(z)+\beta(z)^{-1} & \beta(z)-\beta(z)^{-1}\\\beta(z)-\beta(z)^{-1} & \beta(z)+\beta(z)^{-1}\end{bmatrix}(-i\eta(z))^{K\sigma_3},
\label{eq:edge-FcK}
\end{equation}
where
\begin{equation}
\eta(z)=\eta(z;x):=i\frac{\mathcal{T}(z;x)}{W(z;x)}.
\label{eq:edge-eta-define}
\end{equation}
Note that since $\mathcal{T}$ and $W$ both vanish to precisely first order at $z=\zc$, the function $\eta(z)$ is analytic and nonzero at $z=\zc$. We introduce the related notation
\begin{equation}
\ec=\ec(x):=\eta(\zc(x);x).
\label{eq:edge-ec}
\end{equation}
Of course, we obtain $\mathbf{F}_0$ simply by setting $z=\zc$ in \eqref{eq:edge-FcK}:  $\mathbf{F}_0=\mathbf{F}^{(\zc,K)}(\zc)$.  
Using \eqref{eq:Pauli-matrices} therefore yields
\begin{equation}
\begin{split}
\mathbf{F}_0\sigma_\pm\mathbf{F}_0^{-1}&=\frac{1}{4}(-1)^K\ec^{\pm 2K}\begin{bmatrix}
\mp(\bc^2-\bc^{-2}) & \pm(-1)^K(\bc\pm\bc^{-1})^2\tc^{2K}\\
\mp(-1)^K(\bc\mp\bc^{-1})^2\tc^{-2K} & \pm(\bc^2-\bc^{-2})\end{bmatrix}\\
&=(-1)^K\frac{\ec^{\pm 2K}}{4\rc}\begin{bmatrix}\mp\Delta & 2(-1)^K(\rc\mp S)\tc^{2K}\\
2(-1)^K(\rc\pm S)\tc^{-2K} & \pm\Delta\end{bmatrix}.
\end{split}
\label{eq:edge-F0-sigmapm-F0inverse}
\end{equation}
(In the second step we have used the identities \eqref{eq:edge-beta-identities} and the definition 
\eqref{eq:edge-zc} of $\zc$.)  
To obtain $\mathbf{F}_1$, we differentiate with respect to $z$ and then evaluate at $z=\zc$.  This yields
\eq
\begin{split}
\mathbf{F}_1 = & \mathcal{T}(\infty)^{-K\sigma_3}\frac{\beta'(\zc)}{2\bc}\begin{bmatrix} \bc-\bc^{-1} & \bc+\bc^{-1}\\ \bc+\bc^{-1} & \bc-\bc^{-1}\end{bmatrix}(-i\ec)^{K\sigma_3} \\ 
   & -i\mathcal{T}(\infty)^{-K\sigma_3}\frac{1}{2}\begin{bmatrix}\bc +\bc^{-1} & \bc-\bc^{-1}\\\bc-\bc^{-1} & \bc+\bc^{-1}\end{bmatrix}K\eta'(\zc)\sigma_3(-i\ec)^{K\sigma_3-\mathbb{I}}.
\end{split}
\endeq
Using the fact that $\det(\mathbf{F}_0)=1$, it then follows that
\begin{equation}
\mathbf{F}_0^{-1}\mathbf{F}_1=(-1)^K\frac{\beta'(\zc)}{\bc}\begin{bmatrix}0 & \ec^{-2K}\\
\ec^{2K} & 0\end{bmatrix} +\frac{K\eta'(\zc)}{\ec}\sigma_3.
\end{equation}
Differentiation of \eqref{eq:edge-beta-define} then yields
\begin{equation}
\frac{\beta'(\zc)}{\bc}=\frac{1}{4}\left[\frac{1}{\zc-a}-\frac{1}{\zc-b}\right]=i\frac{i\Delta}{4\rc^2}.
\end{equation}
Therefore,
\begin{equation}
\phi_\pm = i(-1)^K\frac{\ec^{\pm 2K}i\Delta}{4\rc^2},
\label{eq:edge-phi-define}
\end{equation}
and combining this with \eqref{eq:edge-F0-sigmapm-F0inverse} yields
\begin{equation}
\frac{1}{\phi_\pm}\mathbf{F}_0\sigma_\pm\mathbf{F}_0^{-1}=
-\frac{\rc}{\Delta}\begin{bmatrix}\mp\Delta & 2(-1)^K(\rc\mp S)\tc^{2K}\\2(-1)^K(\rc\pm S)\tc^{-2K} & \pm\Delta\end{bmatrix}.
\label{eq:edge-F0-sigmapm-F0inverse-fraction}
\end{equation}
Note also that by combining \eqref{eq:edge-calT-prime} and \eqref{eq:edge-W-prime} we have
\begin{equation}
\ec(x)=i\lim_{z\to \zc(x)}\frac{\mathcal{T}(z;x)}{W(z;x)}=i\frac{\mathcal{T}'(\zc(x);x)}{W'(\zc(x);x)}=
\sqrt{\frac{2}{3}}\frac{i\Delta(x)}{4\rc(x)^{5/2}}.
\label{eq:edge-ec-formula}
\end{equation}
When $x>0$, $\ec(x)$ is a positive real number, and $\ec$ is an analytic function of $x\in\mathcal{S}_0$.

\subsubsection{Singularities of $\dot{\mathbf{P}}(z)$}
\label{section:edge-singularities}
The parametrix $\dot{\mathbf{P}}(z)$ will exist as long as $x$ is such that the matrix $\mathbf{B}$ exists, i.e., the denominator $1+\rho_0Q_s\phi_s$ is nonzero.  Moreover, since $Q_s$
is bounded,
%when the choice of $\pm$ sign is correlated with $\mathrm{sgn}(\log(|f|))$ (or if we consider only $Q_+$ in the region where $K(x;\epsilon)=0$), 
$\dot{\mathbf{P}}(z)$ will be uniformly bounded in $z$, $x$, and $\epsilon$ as long as $1+\rho_0Q_s\phi_s$ is bounded away from zero.  Since $\det(\dot{\mathbf{P}}(z))\equiv 1$, as is easily checked, the same holds for the inverse matrix.  It is the goal of the subsequent discussion to clarify the conditions on $x$ necessary to guarantee the uniform boundedness of $\dot{\mathbf{P}}(z)$.

Combining \eqref{eq:edge-f-define}, \eqref{eq:edge-Q-define}, \eqref{eq:edge-W-prime}, \eqref{eq:edge-phi-define}, and \eqref{eq:edge-ec-formula}, and using the definition \eqref{epsilon} of $\epsilon$ in terms of $m$, we obtain
\begin{equation}
1+\rho_0Q_+\phi_+=1-(-1)^{m-K}e^{-2X_K/\epsilon}
%\exp\left(-\frac{2}{\epsilon}\left[\mathfrak{d}(x)+\frac{1}{2}(K+\tfrac{1}{2})\epsilon\log(\epsilon^{-1})-
%(K+\tfrac{1}{2})\epsilon\ell(x)+\epsilon\log(\sqrt{2\pi}h_K)\right]\right)
%%\left(\sqrt{\frac{2}{3}}\frac{i\Delta}{4r(\zc)^{5/2}}\right)^{2K-1}
%%\epsilon^{K+1/2}e^{-2\mathfrak{d}(x)/\epsilon}(2\pi h_K^2)^{-1}
\label{eq:edge-pert-plus}
\end{equation}
and
\begin{equation}
1+\rho_0Q_-\phi_-=
1-\left[(-1)^{m-J}e^{-2X_{J}/\epsilon}
\right]^{-1},\quad J:=K-1,
%\exp\left(-\frac{2}{\epsilon}\left[\mathfrak{d}(x)+\frac{1}{2}(J+\tfrac{1}{2})\epsilon\log(\epsilon^{-1})-
%(J+\tfrac{1}{2})\epsilon\ell(x)+\epsilon\log(\sqrt{2\pi}h_J)\right]\right)\right]^{-1},
%%
%%\left(\sqrt{\frac{2}{3}}\frac{i\Delta}{4r(\zc)^{5/2}}\right)^{1-2K}
%%\epsilon^{1/2-K}e^{2\mathfrak{d}(x)/\epsilon}2\pi h_{K-1}^2,
\label{eq:edge-pert-minus}
\end{equation}
where %$J:=K-1$,
\begin{equation}
X_n=X_n(x;\epsilon):=\mathfrak{d}(x)+\frac{1}{2}\left(n+\frac{1}{2}\right)\epsilon\log(\epsilon^{-1}) -
\left(n+\frac{1}{2}\right)\epsilon\ell(x)+\epsilon\log(\sqrt{2\pi}h_n),\quad n=0,1,2,3,\dots,
\label{eq:edge-Xn-define}
\end{equation}
$\mathfrak{d}(\cdot)$ is defined in terms of $\mathfrak{c}(\cdot)$ by \eqref{eq:d-define},
%\begin{equation}
%\mathfrak{d}(x):=\mathfrak{c}(x)-\frac{i\pi}{2}
%\end{equation}
and
\begin{equation}
\ell(x):=\log\left(\sqrt{\frac{2}{3}}\frac{i\Delta(x)}{4\rc(x)^{5/2}}\right).
\end{equation}
Here, the logarithm and power function in its argument are defined by analytic continuation from the ray $\arg(x)=0$ to the full sector $\mathcal{S}_0$.  Indeed, for $x>0$ we have $\rc(x)>0$ and $i\Delta(x)>0$.  The value of $\ell(x)$ for $\arg(x)=0$ is then defined to be real for $x>0$.  It is easy to see that this definition makes $\ell(x)$ a Schwarz-symmetric analytic function, i.e., $\ell(x^*)=\ell(x)^*$ holds for all $x\in\mathcal{S}_0$.  

According to Lemma~\ref{lemma:d-univalent}, 
%as a translate of $\mathfrak{c}(x)$, 
the function $\mathfrak{d}(x)$ is analytic and univalent for $x\in\mathcal{S}_0$, and $\mathfrak{d}(x)$ is real for real positive $x$. 
%and univalent on a neighborhood of every closed sub-arc $A$ of the edge of $\partial T$ in this open sector.  Moreover, $\mathfrak{d}(x)$ is real for real positive $x$.  To see the reality for $x>0$, note that in this situation we have $b(x)=a(x)^*$, which makes $h'(z)$ a Schwarz-symmetric function of $z$, i.e., $h'(z^*)=h'(z)^*$.  If we deform the contour $L$ into a literal ``L''-shape coinciding with the right edge of the branch cut $\Sigma$ between $z=b$ and $z=\Re(b)=\Re(a)$ and the real half-line $(\Re(b),+\infty)$, then $h(z)$ itself becomes Schwarz-symmetric.  Since $h_+(z)-h_-(z)=2\pi i$ across the horizontal ray $(\Re(b),+\infty)$, Schwarz symmetry implies that $\Im(h_\pm(z))=\pm\pi$ for $z>\Re(b)$, and in particular $\Im(h_+(\zc))=\pi$.  Also by Schwarz symmetry and the analyticity of $h$ for $z<\Re(b)$, we have $\Im(h(z))=0$ for $z<\Re(b)$.  Now the sum of the boundary values of $h(z)$ taken on the part of the branch cut $\Sigma$ that is not coincident with the vertical leg of $L$ is the constant $-\lambda(x)$.  By evaluating this constant sum at a point $z\in\Sigma$ taken in the limit $z\to\Re(b)$ along the segment of $\Sigma$ in the upper half-plane, we obtain 
%$\Im(\lambda(x))=-\pi$.  Therefore,
%\begin{equation}
%\Im(\mathfrak{d}(x))=\Im(\mathfrak{c}(x))-\frac{\pi}{2}=\Im\left(h_+(\zc(x);x)+\frac{1}{2}\lambda(x)\right)-\frac{\pi}{2}=\pi-\frac{\pi}{2}-\frac{\pi}{2}=0,\quad x>0.
%\end{equation}
It follows that the exponent $X_n$ defined by \eqref{eq:edge-Xn-define} is real-valued for all $x>0$.  
Of course the edge of the elliptic region $T$ in the sector $\mathcal{S}_0$ is given by the equation 
$\Re(\mathfrak{d}(x))=0$.  Along the edge from $x_ce^{2\pi i/3}$ to $x_ce^{-2\pi i/3}$ (bottom to top), $\Im(\mathfrak{d}(x))$ varies monotonically from $-\pi/2$ to $\pi/2$, also according to Lemma~\ref{lemma:d-univalent}.  
%By Lemma~\ref{lemma:c-univalent}, the monotonicity follows from the univalence of $\mathfrak{c}$ and hence of $\mathfrak{d}$ near the edge.  To obtain the extreme value at $x=x_ce^{-2\pi i/3}$, simply note that at this point $\zc(x)$ coalesces with $b(x)$, and hence the integral formula for $\mathfrak{c}(x)$ can be evaluated by residues; thus one obtains $\mathfrak{c}(x)=-3\pi i\Delta(x)^3/32$.  It therefore remains to evaluate $\Delta(x)$ at the corner point $x=x_ce^{-2\pi i/3}$; as a branch point for $S(x)$ it is easy to obtain the corresponding values of $S$ and hence $\Delta$.  It follows that $\mathfrak{c}(x_ce^{-2\pi i/3})=i\pi$, and hence $\mathfrak{d}(x_ce^{-2\pi i/3})=i\pi/2$.  The extreme value at $x=x_ce^{2\pi i/3}$ is obtained by
%Schwarz symmetry of $\mathfrak{d}$.   
%
%\textcolor{red}{Formulate the following long paragraph as a lemma, whose proof can be moved to an appendix?}  
A similar result involving the function $\ell(\cdot)$ is the following, the proof of which can be found in \S\ref{sec:lemma:ell}.
\begin{lemma}
As $x$ varies along the edge $\partial T\cap\mathcal{S}_0$ from $x_ce^{2\pi i/3}$ to $x_ce^{-2\pi i/3}$ (bottom to top), $\Im(\ell(x))$ varies continuously from $-\pi/4$ to $\pi/4$ with $\Im(\ell(x_e))=0$, where $x_e$ denotes the real point of $\partial T\cap\mathcal{S}_0$.
\label{lemma:ell}
\end{lemma}

Recall that if $s(x;\epsilon)=+$ (respectively, if $s(x;\epsilon)=-$) and also if the expression written in \eqref{eq:edge-pert-plus} (respectively, the expression written in \eqref{eq:edge-pert-minus})
is bounded away from zero, then $\dot{\mathbf{P}}(z)$ will be under control.  
It is clear from \eqref{eq:edge-pert-plus}--\eqref{eq:edge-pert-minus} that the condition that $1+\rho_0Q_-\phi_-$ vanishes is exactly the same as the condition that $1+\rho_0Q_+\phi_+$ vanishes, provided that $K$ is replaced by $K-1$ in the former.  This shows that the existence of a singularity of $\mathbf{B}$ is insensitive to the jump discontinuities in the integer-valued function $K(x;\epsilon)$ defined by \eqref{eq:K-define}--\eqref{eq:K-define-zero}.  Taking $n=K$ (for $s=+$) or $n=K-1$ (for $s=-$), the condition $1+\rho_0Q_s\phi_s=0$ is equivalent to 
\begin{equation}
X_n(x;\epsilon)
%\mathfrak{d}(x)+\frac{1}{2}\left(K+\frac{1}{2}\right)\epsilon\log(\epsilon^{-1})-\left(K+\frac{1}{2}\right)\epsilon\ell(x)+\epsilon\log(\sqrt{2\pi}h_K)
\in\begin{cases}i\pi\epsilon\mathbb{Z}, &\quad m\equiv n\pmod{2},\\
i\pi\epsilon(\mathbb{Z}+\tfrac{1}{2}),&\quad m\not\equiv n \pmod{2}.
\end{cases}
\label{eq:edge-pole-condition}
\end{equation}
To find the singularities of $\mathbf{B}$, and hence of $\dot{\mathbf{P}}(z)$, we temporarily ignore any relation between the integer $n=0,1,2,3,\dots$ and $x$.  Consider solving \eqref{eq:edge-pole-condition} perturbatively in the limit $\epsilon\to 0$, or equivalently, $m\to\infty$.  Let $\alpha$ be a fixed real number in the interval $(-\tfrac{1}{2},\tfrac{1}{2})$, let $n\ge 0$ be a fixed integer, and suppose (first) that $m$ is tending to infinity through a sequence of integers of the same parity as $n$.  We then use \eqref{eq:edge-Xn-define} to write \eqref{eq:edge-pole-condition} in the form
\begin{equation}
\mathfrak{d}(x)+\frac{1}{2}\left(n+\frac{1}{2}\right)\epsilon\log(\epsilon^{-1})-\left(n+\frac{1}{2}\right)\epsilon\ell(x)+\epsilon\log(\sqrt{2\pi}h_n)=i\pi\epsilon (N_0(\alpha;\epsilon)+N_1),
\end{equation}
where $N_0(\alpha;\epsilon):=\llbracket\epsilon^{-1}\alpha\rrbracket$ and where $N_1\in\mathbb{Z}$ is assumed to be bounded.  Given $\alpha$, $n$, and $N_1$, there is a unique solution $x$ of this equation with an asymptotic expansion as $\epsilon\to 0$ of the form
\begin{equation}
x = x_0+x_1\epsilon\log(\epsilon^{-1}) +x_2\epsilon + o(\epsilon),\quad\epsilon\to 0.
\end{equation}
Indeed, since $\epsilon N_0(\alpha;\epsilon)\to \alpha$ as $\epsilon\to 0$, we obtain $x_0=\mathfrak{d}^{-1}(i\pi\alpha)$, where the inverse function to $\mathfrak{d}$ is guaranteed to exist and be analytic on the imaginary axis between $-i\pi/2$ and $i\pi/2$ by 
Lemma~\ref{lemma:d-univalent}.
%univalence of $\mathfrak{c}$ and hence of $\mathfrak{d}$ guaranteed by Lemma~\ref{lemma:c-univalent}.  
Note that $x_0$ lies on the boundary arc $\partial T$ in the range $|\arg(x_0)|<\pi/3$.  At subsequent orders in perturbation theory we obtain 
\begin{equation}
x_1=-\frac{n+\tfrac{1}{2}}{2\mathfrak{d}'(x_0)}\quad\text{and}\quad
x_2=\frac{i\pi(\llbracket\epsilon^{-1}\alpha\rrbracket-\epsilon^{-1}\alpha + N_1)+(n+\tfrac{1}{2})\ell(x_0)-\log(\sqrt{2\pi}h_n)}{\mathfrak{d}'(x_0)}.
\end{equation}
Note that $-\tfrac{1}{2}<\llbracket\epsilon^{-1}\alpha\rrbracket-\epsilon^{-1}\alpha\le\tfrac{1}{2}$ by definition of the nearest integer function $\llbracket\cdot\rrbracket$.  Alternately, we could write $\ell(x)=\widetilde{\ell}(\mathfrak{d})$ using univalence of $\mathfrak{d}$ and obtain a simpler expansion of $\mathfrak{d}$:
\begin{equation}
\mathfrak{d}=i\pi\alpha -\frac{1}{2}(n+\tfrac{1}{2})\epsilon\log(\epsilon^{-1}) +
\left(i\pi(\llbracket\epsilon^{-1}\alpha\rrbracket-\epsilon^{-1}\alpha+N_1)+(n+\tfrac{1}{2})\widetilde{\ell}(i\pi\alpha)
-\log(\sqrt{2\pi}h_n)\right)\epsilon+o(\epsilon).
\label{eq:edge-d-sequence}
\end{equation}
In the case that $m$ tends to infinity through a sequence of integers of opposite parity to $n$, 
the above formulae hold true with $N_1$ replaced by $N_1+\tfrac{1}{2}\in\mathbb{Z}+\tfrac{1}{2}$.

These calculations show that the singularities of $\mathbf{B}$ lie $\epsilon$-close to curves in the $x$-plane that are mapped into vertical straight lines $\Re(\mathfrak{d})=-\tfrac{1}{2}(n+\tfrac{1}{2})\epsilon\log(\epsilon^{-1})$ in the $\mathfrak{d}$-plane, $n=0,1,2,\dots$.  Given a value of $\alpha\in (-\tfrac{1}{2},\tfrac{1}{2})$, the singularities of $\mathbf{B}$ near the line indexed by $n$ with $|\Im(\mathfrak{d})-\pi\alpha|=\mathcal{O}(\epsilon)$ form an approximate vertical lattice in the $\mathfrak{d}$-plane indexed by $N_1\in\mathbb{Z}$ with spacing $i\pi\epsilon$.  The lattice is offset from the point $\mathfrak{d}=i\pi\alpha-\tfrac{1}{2}(n+\tfrac{1}{2})\epsilon\log(\epsilon^{-1})$ by a complex shift of size $\mathcal{O}(\epsilon)$ that depends on both $m$ (or $\epsilon$) and $n$ (as well as $\alpha$).  Of particular interest is the way that the imaginary part of this offset depends on the integers $m$ and $n$.  Holding $m$ fixed, the difference between the imaginary part of the offset for line $n$ and $n-1$ is given by $\epsilon\pi/2+\epsilon \Im(\widetilde{\ell}(i\pi\alpha))$ (the term $\epsilon\pi/2$ comes from the replacement of $N_1$ by $N_1+\tfrac{1}{2}$ in \eqref{eq:edge-d-sequence}).  This implies a vertical (in the $\mathfrak{d}$-plane) ``staggering'' effect of the lattices corresponding to neighboring values of $n$ when examined near a common fixed value $\pi\alpha$ of $\Im(\mathfrak{d})$.  The amount of staggering as a fraction of the lattice spacing varies with $\alpha$ (as one moves along the edge).  For example, when $\alpha=0$ (and hence we are examining the singularities of $\mathbf{B}$ near the real axis in the $x$-plane) we have $\Im(\widetilde{\ell}(0))=0$ and therefore the vertical displacement of the lattices corresponding to neighboring values of $n$ is half of the lattice spacing.  As $\alpha$ increases to the extreme value of $\alpha=\tfrac{1}{2}$, $\Im(\widetilde{\ell}(i\pi\alpha))$ increases to $\pi/4$, and therefore near the upper corner $x=x_ce^{-2\pi i/3}$ the vertical displacement of the lattices corresponding to incrementing the value of $n$ is $3/4$ (or equivalently, $-1/4$) of the lattice spacing.  (This effect is clearly visible, with the correct staggering fractions, 
in Figure \ref{fig:u-poles-24-25}.)  On the other hand,
if $n$ is held fixed, the vertical shift of the lattice in the $\mathfrak{d}$-plane associated with replacing $m$ with $m+1$ is, to leading order, $\epsilon\pi/2 -\epsilon \pi\alpha$, where again the first term comes from replacing $N_1$ by $N_1+\tfrac{1}{2}$ in \eqref{eq:edge-d-sequence}.
Therefore when $\alpha=0$ one again has a vertical staggering by half of the lattice spacing, while as $\alpha$ increases to $\alpha=\tfrac{1}{2}$, the vertical staggering associated with replacing $m$ by $m+1$ becomes an integer multiple of the spacing itself, and hence the singularities of $\mathbf{B}$ near the corners of $T$ do not move much as $m$ changes.
(This effect is also visible in Figure \ref{fig:u-poles-24-25}.)
\begin{figure}[h]
\begin{center}
\includegraphics[height=3 in]{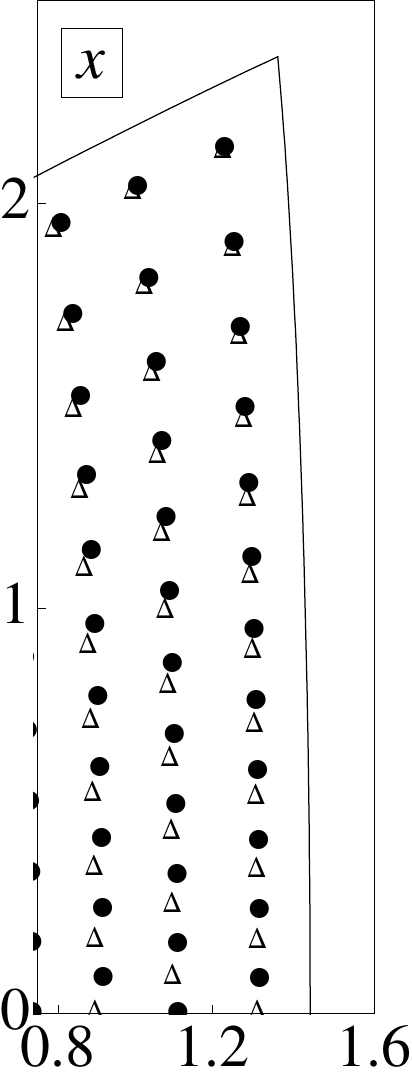}%
\end{center}
\caption{The poles of $\pu_{24}((24-\tfrac{1}{2})^{2/3}x)$ (triangles) and the 
poles of $\pu_{25}((25-\tfrac{1}{2})^{2/3}x)$ (dots) along an edge of $\partial T$ in 
the (upper half) complex $x$-plane.}
\label{fig:u-poles-24-25}
\end{figure}

We define a subset of the full sector $\mathcal{S}_0$ as follows:
\begin{equation}
\mathscr{S}_{m}^\infty(\delta):=\left\{y\in\mathcal{S}_0:\;|y-x|>\delta\epsilon \;\;\text{for all $x$ solving \eqref{eq:edge-pole-condition} for any $n=0,1,2,3,\dots$}\right\}.
\label{eq:edge-pole-cheese}
\end{equation}
The set $\mathscr{S}_m^\infty(\delta)$ omits from the sector $\mathcal{S}_0$ all ``holes'' of radius $\delta\epsilon$ centered at singularities of $\mathbf{B}$.  We have the following result.
\begin{lemma}
Let $\delta>0$ and $M>0$ be given.  Then $\dot{\mathbf{P}}(z;x,(m-\tfrac{1}{2})^{-1})$ and $\dot{\mathbf{P}}(z;x,(m-\tfrac{1}{2})^{-1})^{-1}$ are uniformly bounded for $m=0,1,2,3,\dots$,
for $x\in\mathscr{S}_m^\infty(\delta)$ with $K(x;(m-\tfrac{1}{2})^{-1})\le M$, and for all $z\in\mathbb{C}\setminus\partial\mathbb{D}_{\zc}$.  Under the same conditions, $\mathbf{B}(x;(m-\tfrac{1}{2})^{-1})=\mathcal{O}(Q_s)$, where the sign $s$ is determined as a function of $x$ and $\epsilon$ by \eqref{eq:edge-s-define}.
\label{lemma:edge-dotP-bounded-in-cheese}
\end{lemma}
\begin{proof}
The condition $x\in\mathscr{S}_m^\infty(\delta)$ bounds the denominator $1+\rho_0Q_s\phi_s$ away from zero.
\end{proof}

\subsection{Error analysis by small-norm theory}
Let positive numbers $\sigma$, $\delta$, and $M$ be given, and suppose that $x\in\mathcal{S}_\sigma\cap\mathscr{S}_m^\infty(\delta)$ and that $K(x;\epsilon)\le M$.  Consider the matrix (the \emph{error} in approximating $\mathbf{P}(z)$ by its parametrix $\dot{\mathbf{P}}(z)$) defined by
\begin{equation}
\mathbf{E}(z):=\mathbf{P}(z)\dot{\mathbf{P}}(z)^{-1}
\end{equation}
for all $z\in\mathbb{C}$ for which both factors are defined.  Thus, $\mathbf{E}(z)$ is analytic for $z\in\mathbb{C}\setminus\Sigma^{(\mathbf{E})}$, where $\Sigma^{(\mathbf{E})}=\Sigma^{(\mathbf{P})}$ because the jump contour for $\dot{\mathbf{P}}(z)$ is contained within that of $\mathbf{P}(z)$.  Let $\mathbf{V}^{(\mathbf{E})}(z)$ denote the jump matrix for $\mathbf{E}(z)$, i.e., $\mathbf{E}_+(z)=\mathbf{E}_-(z)\mathbf{V}^{(\mathbf{E})}(z)$ holds at each regular point of $\Sigma^{(\mathbf{E})}$.  For $z\in\Sigma^{(\mathbf{E})}\setminus\partial\mathbb{D}_{\zc}$, we have
\begin{equation}
\mathbf{V}^{(\mathbf{E})}(z)=\dot{\mathbf{P}}(z)\mathbf{V}^{(\mathbf{P})}(z)\dot{\mathbf{P}}(z)^{-1} =\mathbb{I}+\mathcal{O}(\epsilon),
\quad z\in\Sigma^{(\mathbf{E})}\setminus\partial\mathbb{D}_{\zc},
\label{eq:edge-error-jump-outside}
\end{equation}
because $\dot{\mathbf{P}}(z)$ and its inverse are uniformly bounded in the complex plane according to Lemma~\ref{lemma:edge-dotP-bounded-in-cheese}, and $\mathbf{V}^{(\mathbf{P})}(z)-\mathbb{I}$ is dominated by its uniformly $\mathcal{O}(\epsilon)$ behavior for $z\in\partial\mathbb{D}_a\cup\partial\mathbb{D}_b$.  The estimate \eqref{eq:edge-error-jump-outside} holds both in $L^\infty(\Sigma^{(\mathbf{E})})$ and in the weighted $L^1$-space $L^1(\Sigma^{(\mathbf{E})},z^2)$.
On the other hand, for $z\in\partial\mathbb{D}_{\zc}\subset\Sigma^{(\mathbf{E})}$, the jump matrix for $\mathbf{E}(z)$ is generally a larger (but still small) perturbation of the identity:
\begin{equation}
\begin{split}
\mathbf{V}^{(\mathbf{E})}(z)&=\dot{\mathbf{P}}_-(z)\mathbf{V}^{(\mathbf{P})}(z)\dot{\mathbf{V}}^{(\mathbf{P})}(z)^{-1}\dot{\mathbf{P}}_-(z)^{-1}\\ &=
\dot{\mathbf{P}}_-(z)\mathbf{F}^{(\zc,K)}(z)
\left(\mathbb{I}-\frac{Q_{\overline{s}}}{W(z)}\sigma_{\overline{s}} +\mathcal{O}(\epsilon)\right)\mathbf{F}^{(\zc,K)}(z)^{-1}\dot{\mathbf{P}}_-(z)^{-1},\quad z\in\partial\mathbb{D}_{\zc},
\end{split}
\label{eq:edge-error-jump-circle}
\end{equation}
where we recall that $\overline{s}$ is the sign opposite to $s$.  
%$\mp = -\mathrm{sgn}(\log(|f|))$ (for $K\ge 1$; for $K=0$ we always have the ``$-$'' sign and obtain $Q_-=0$ so that the right-hand side is simply $\mathbb{I}+\mathcal{O}(\epsilon)$).  
%In this situation, the coefficient $Q_\mp$ always satisfies $Q_\mp=\mathcal{O}(\epsilon^{1/2})$, and consequently the jump matrix 
Note that since $Q_{\overline{s}}=\mathcal{O}(\epsilon^{1/2})$, 
\eqref{eq:edge-error-jump-circle} could be written with less precision simply as $\mathbf{V}^{(\mathbf{E})}(z)=\mathbb{I}+\mathcal{O}(\epsilon^{1/2})$ with the estimate holding uniformly for $z\in\partial\mathbb{D}_{\zc}$.  Combining these estimates 
%\textcolor{red}{(say something about uniformity of the Cauchy operator norms)} 
and applying the theory of small-norm Riemann-Hilbert problems as described in \cite[Appendix B]{Buckingham-Miller-rational-noncrit},
we see that $\mathbf{E}(z)$ is determined uniquely by its jump matrix \eqref{eq:edge-error-jump-outside}--\eqref{eq:edge-error-jump-circle} together with the normalization condition $\mathbf{E}(z)\to\mathbb{I}$ as $z\to\infty$, and it has an expansion for large $z$ of the form
\begin{equation}
\mathbf{E}(z)=\mathbb{I}+\mathbf{E}_1z^{-1}+\mathbf{E}_2z^{-2}+o(z^{-2}),\quad z\to\infty,
\end{equation}
where the overall $\mathcal{O}(\epsilon^{1/2})$ estimate of $\mathbf{V}^{(\mathbf{E})}(z)-\mathbb{I}$ and the more precise estimates \eqref{eq:edge-error-jump-outside}--\eqref{eq:edge-error-jump-circle} imply that (see equation (B-18) of \cite{Buckingham-Miller-rational-noncrit})
\begin{equation}
\begin{split}
\mathbf{E}_1&=\frac{Q_{\overline{s}}}{2\pi i}\oint_{\partial\mathbb{D}_{\zc}}\dot{\mathbf{P}}_-(z)\mathbf{F}^{(\zc,K)}(z)\sigma_{\overline{s}}\mathbf{F}^{(\zc,K)}(z)^{-1}\dot{\mathbf{P}}_-(z)^{-1}\frac{dz}{W(z)} + \mathcal{O}(\epsilon),\\
\mathbf{E}_2&=\frac{Q_{\overline{s}}}{2\pi i}\oint_{\partial\mathbb{D}_{\zc}}\dot{\mathbf{P}}_-(z)\mathbf{F}^{(\zc,K)}(z)\sigma_{\overline{s}}\mathbf{F}^{(\zc,K)}(z)^{-1}\dot{\mathbf{P}}_-(z)^{-1}\frac{z\,dz}{W(z)}+ \mathcal{O}(\epsilon).
\end{split}
\label{eq:edge-E1-E2}
\end{equation}
The error terms are uniform with respect to $x$ in the specified domain because a finite number of different contours $\Sigma^{(\mathbf{E})}$ suffice even though $x$ can take an infinite number of values (for details about how these contours are chosen, see \cite[\S3.6.3]{Buckingham-Miller-rational-noncrit}).
The explicit terms in \eqref{eq:edge-E1-E2} may be computed by residues (recall that $\partial\mathbb{D}_{\zc}$ is oriented negatively):
\begin{equation}
\begin{split}
\mathbf{E}_1&=-\rho_0Q_{\overline{s}}\dot{\mathbf{P}}(\zc)\mathbf{F}_0\sigma_{\overline{s}}\mathbf{F}_0^{-1}\dot{\mathbf{P}}(\zc)^{-1} + \mathcal{O}(\epsilon),\\
\mathbf{E}_2&=-\zc\rho_0Q_{\overline{s}}\dot{\mathbf{P}}(\zc)\mathbf{F}_0\sigma_{\overline{s}}\mathbf{F}_0^{-1}\dot{\mathbf{P}}(\zc)^{-1} + \mathcal{O}(\epsilon).
\end{split}
\end{equation}
Of course $\mathbf{E}_1=\mathcal{O}(\epsilon)$ and $\mathbf{E}_2=\mathcal{O}(\epsilon)$ 
for $K=0$ because $Q_-=0$ in this case.  To further simplify these matrices for $K\ge 1$, 
we first calculate $\dot{\mathbf{P}}(\zc)$ by continuing the expansion \eqref{eq:edge-LHS-expand}
to the next order in $z-\zc$.  Using the fact that $\mathbf{B}=\mathcal{O}(Q_s)$ according to 
Lemma~\ref{lemma:edge-dotP-bounded-in-cheese}, we obtain
\begin{equation}
\begin{split}
\dot{\mathbf{P}}(\zc)&=\mathbf{F}_0(\mathbb{I}+\rho_1Q_s\sigma_s)\mathbf{F}_0^{-1}+\rho_0Q_s\mathbf{F}_1\sigma_s\mathbf{F}_0^{-1}-\rho_0Q_s\mathbf{F}_0\sigma_s\mathbf{F}_0^{-1}\mathbf{F}_1\mathbf{F}_0^{-1}+\mathcal{O}(\mathbf{B})\\ & = \mathbb{I}+\mathcal{O}(Q_s).
\end{split}
\end{equation}
Similarly, by expanding instead $\dot{\mathbf{V}}^{(\mathbf{P})}(z)(\mathbb{I}+(z-\zc)^{-1}\mathbf{B})^{-1}$, which equals $\dot{\mathbf{V}}^{(\mathbf{P})}(z)(\mathbb{I}-(z-\zc)^{-1}\mathbf{B})$ since $\mathbf{B}^2=\mathbf{0}$, we obtain
\begin{equation}
\begin{split}
\dot{\mathbf{P}}(\zc)^{-1}&=
\mathbf{F}_0(\mathbb{I}-\rho_1Q_s\sigma_s)\mathbf{F}_0^{-1}-\rho_0Q_s\mathbf{F}_1\sigma_s\mathbf{F}_0^{-1}+\rho_0Q_s\mathbf{F}_0\sigma_s\mathbf{F}_0^{-1}\mathbf{F}_1\mathbf{F}_0^{-1}+\mathcal{O}(\mathbf{B})\\ &=\mathbb{I}+\mathcal{O}(Q_s).
\end{split}
\end{equation}
Therefore, using \eqref{eq:edge-Qplus-Qminus-product}, 
\begin{equation}
\begin{split}
\mathbf{E}_1&=-\rho_0Q_{\overline{s}}\mathbf{F}_0\sigma_{\overline{s}}\mathbf{F}_0^{-1}+\mathcal{O}(\epsilon),\\
\mathbf{E}_2&=-\zc\rho_0Q_{\overline{s}}\mathbf{F}_0\sigma_{\overline{s}}\mathbf{F}_0^{-1}+\mathcal{O}(\epsilon)
\end{split}
\end{equation}
both hold for $1\le K\le M$ with $x\in\mathscr{S}_m^\infty(\delta)$.  Finally, since $Q_{\overline{s}}^2=\mathcal{O}(\epsilon)$ (a fact which also bounds $1+\rho_0Q_{\overline{s}}\phi_{\overline{s}}$ away from zero), we choose to write these formulae in the equivalent form
\begin{equation}
\begin{split}
\mathbf{E}_1&=-\frac{\rho_0Q_{\overline{s}}\phi_{\overline{s}}}{1+\rho_0Q_{\overline{s}}\phi_{\overline{s}}}\cdot\frac{1}{\phi_{\overline{s}}}
\mathbf{F}_0\sigma_{\overline{s}}\mathbf{F}_0^{-1}+\mathcal{O}(\epsilon),\\
\mathbf{E}_2&=-\zc\cdot\frac{\rho_0Q_{\overline{s}}\phi_{\overline{s}}}{1+\rho_0Q_{\overline{s}}\phi_{\overline{s}}}\cdot\frac{1}{\phi_{\overline{s}}}
\mathbf{F}_0\sigma_{\overline{s}}\mathbf{F}_0^{-1}+\mathcal{O}(\epsilon).
\end{split}
\label{eq:edge-E1-E2-sym}
\end{equation}
The form \eqref{eq:edge-E1-E2-sym} yields the most symmetric expressions for the asymptotic behavior of the rational Painlev\'e-II functions, allowing us to write formulae that do not involve the sign $s$ defined by \eqref{eq:edge-s-define}, as we will now see.

\begin{remark}
%\textcolor{red}{Formulate this remark more clearly\dots} 
Without computing the leading terms of $\mathbf{E}_1$ and $\mathbf{E}_1$ by residues, one may obtain directly from the crude estimate $\mathbf{V}^{(\mathbf{E})}(z)-\mathbb{I}=\mathcal{O}(\epsilon^{1/2})$ that $\mathbf{E}_1=\mathcal{O}(\epsilon^{1/2})$ and $\mathbf{E}_2=\mathcal{O}(\epsilon^{1/2})$.  This is enough to obtain asymptotic formulae for the rational Painlev\'e-II functions $\pu_m$, $\pv_m$, $\pp_m$, and $\pq_m$ that are accurate also to $\mathcal{O}(\epsilon^{1/2})$.  
Of course, one may freely modify the leading terms in such formulae by adding in quantities of size $\mathcal{O}(\epsilon^{1/2})$, and the estimate of the error will be unchanged.  It turns out that a natural way to so modify the leading terms is to add in precisely those $\mathcal{O}(\epsilon^{1/2})$ quantities that correspond to the leading terms of $\mathbf{E}_1$ and $\mathbf{E}_2$.  This choice leads to the simplest asymptotic formulae for $\pu_m$, $\pv_m$, $\pp_m$, and $\pq_m$.  

A key point is therefore that the same corrections that one would like to add to arrive at the simplest possible formulae for $\pu_m$, $\pv_m$, $\pp_m$, and $\pq_m$ (those corresponding to the leading terms in $\mathbf{E}_1$ and $\mathbf{E}_2$ computed above by residues) are also the correct terms to add to reduce the error in size from $\mathcal{O}(\epsilon^{1/2})$ to $\mathcal{O}(\epsilon)$.
%The idea of computing the leading terms in the basic approximations $\mathbf{E}_1=\mathcal{O}(\epsilon^{1/2})$ and $\mathbf{E}_2=\mathcal{O}(\epsilon^{1/2})$ to improve the order of accuracy of the error term is evidently due to Claeys and Grava \cite{Claeys:2010}.  
We only wish to emphasize that the calculation of residues is not necessary if the cruder error estimate of $\mathcal{O}(\epsilon^{1/2})$ suffices for applications.
%; using just the cruder estimates $\mathbf{E}_1=\mathcal{O}(\epsilon^{1/2})$ and $\mathbf{E}_2=\mathcal{O}(\epsilon^{1/2})$ yields the same asymptotic formulae at leading order, but one can only say that the error of these approximations is $\mathcal{O}(\epsilon^{1/2})$.  By including the corrections to $\mathbf{E}_1$ and $\mathbf{E}_2$, one does not generate any new terms, but one is able to prove that the leading terms are accurate to $\mathcal{O}(\epsilon)$.
\myendrmk
\end{remark}

\subsection{Asymptotic formulae for the rational Painlev\'e-II functions}
In terms of $\mathbf{E}(z)$, the matrix $\mathbf{O}(z)$ can be represented explicitly as
\begin{equation}
\mathbf{O}(z)=\mathbf{P}(z)\dot{\mathbf{O}}(z)=\mathbf{E}(z)\dot{\mathbf{P}}(z)\dot{\mathbf{O}}(z).
\end{equation}
The product $\mathbf{E}(z)\dot{\mathbf{P}}(z)$ therefore plays a similar role in the analysis of the Painlev\'e-II rational functions when $x$ is near the edge $\partial T$ that the matrix $\mathbf{E}(z)$ alone plays for $x$ outside of the elliptic region.  Adapting the exact formulae \eqref{eq:um-vm-pm-qm-ito-O-edge} for the functions $\pu_m$, $\pv_m$, $\pp_m$, and $\pq_m$ 
with this modification, we obtain
\begin{equation}
%\begin{split}
\epsilon^{(2+\epsilon)/(3\epsilon)}e^{-\lambda/\epsilon}\pu_m=\dot{O}_{1,12}+(\dot{\mathbf{P}}_1+\mathbf{E}_1)_{12},
\label{eq:edge-Um-exact}
\end{equation}
\begin{equation}
\epsilon^{-(2-\epsilon)/(3\epsilon)}e^{\lambda/\epsilon}\pv_m=\dot{O}_{1,21}+(\dot{\mathbf{P}}_1+\mathbf{E}_1)_{21},
%\end{split}
\label{eq:edge-Vm-exact}
\end{equation}
\begin{multline}
\epsilon^{1/3}\pp_m=\dot{O}_{1,22}+(\dot{\mathbf{P}}_1+\mathbf{E}_1)_{22}\\{}-
\frac{\dot{O}_{2,12}+(\dot{\mathbf{P}}_2+\mathbf{E}_2)_{12}+(\dot{\mathbf{P}}_1+\mathbf{E}_1)_{11}\dot{O}_{1,12}+(\dot{\mathbf{P}}_1+\mathbf{E}_1)_{12}\dot{O}_{1,22}+(\mathbf{E}_1\dot{\mathbf{P}}_1)_{12}}{\dot{O}_{1,12}+(\dot{\mathbf{P}}_1+\mathbf{E}_1)_{12}},
\label{eq:edge-pp-formula1}
\end{multline}
and
\begin{multline}
\epsilon^{1/3}\pq_m=-\dot{O}_{1,11}-(\dot{\mathbf{P}}_1+\mathbf{E}_1)_{11}\\{}+
\frac{\dot{O}_{2,21}+(\dot{\mathbf{P}}_2+\mathbf{E}_2)_{21}+(\dot{\mathbf{P}}_1+\mathbf{E}_1)_{21}\dot{O}_{1,11}+(\dot{\mathbf{P}}_1+\mathbf{E}_1)_{22}\dot{O}_{1,21}+(\mathbf{E}_1\dot{\mathbf{P}}_1)_{21}}{\dot{O}_{1,21}+(\dot{\mathbf{P}}_1+\mathbf{E}_1)_{21}}.
\label{eq:edge-pq-formula1}
\end{multline}
Here, $\dot{\mathbf{P}}_1$ and $\dot{\mathbf{P}}_2$ (respectively, 
$\dot{\mathbf{O}}_1$ and $\dot{\mathbf{O}}_2$) are matrix coefficients in the 
expansion of $\dot{\mathbf{P}}(z)$ (respectively, $\dot{\mathbf{O}}(z)$) for 
large $z$.  Since $\dot{\mathbf{P}}(z)=\ddot{\mathbf{P}}(z)$ for sufficiently large $|z|$, from \eqref{eq:ddotP-formula} we have
\begin{equation}
\dot{\mathbf{P}}(z)=\mathbb{I}+\dot{\mathbf{P}}_1z^{-1}+\dot{\mathbf{P}}_2z^{-2}+\mathcal{O}(z^{-3}),\quad z\to\infty,
\end{equation}
where, also using \eqref{eq:A-formula}, we have
\begin{equation}
\dot{\mathbf{P}}_1=\mathbf{B}=-\frac{\rho_0Q_s\phi_s}{1+\rho_0Q_s\phi_s}\cdot\frac{1}{\phi_s}\mathbf{F}_0\sigma_s\mathbf{F}_0^{-1}\quad\text{and}\quad
\dot{\mathbf{P}}_2=\zc\mathbf{B}=-\zc\cdot\frac{\rho_0Q_s\phi_s}{1+\rho_0Q_s\phi_s}\cdot\frac{1}{\phi_s}\mathbf{F}_0\sigma_s\mathbf{F}_0^{-1}.
\end{equation}
We also need the first few coefficients in the large-$z$ expansion of the matrix $\dot{\mathbf{O}}(z)$.  Since $\dot{\mathbf{O}}(z)=\dot{\mathbf{O}}^{(\mathrm{out},K)}(z)$ for sufficiently large $|z|$, we have 
\begin{equation}
\dot{\mathbf{O}}(z)=\mathbb{I}+\dot{\mathbf{O}}_1z^{-1}+\dot{\mathbf{O}}_2z^{-2}+\mathcal{O}(z^{-3}),\quad z\to\infty,
\end{equation}
where the coefficients $\dot{\mathbf{O}}_1$ and $\dot{\mathbf{O}}_2$ may be computed by combining \eqref{eq:edge-beta-define}, \eqref{Odot-gen0}, \eqref{eq:edge-Joukowski}, and \eqref{eq:edge-correction-factor} with \eqref{eq:edge-Outer-parametrix}.  The result is that
\begin{equation}
\dot{\mathbf{O}}_1 = \begin{bmatrix}\displaystyle\frac{K\Delta}{4}\left(\frac{2}{i\tc}+\frac{4S}{\Delta}\right) & \displaystyle(-1)^K\frac{\Delta \tc^{2K}}{4}\\
\displaystyle (-1)^K\frac{\Delta}{4\tc^{2K}} & \displaystyle -\frac{K\Delta}{4}\left(\frac{2}{i\tc}+\frac{4S}{\Delta}\right)\end{bmatrix}
\label{eq:edge-dot-O1}
\end{equation}
and
\begin{equation}
\dot{\mathbf{O}}_2=\begin{bmatrix}\dot{O}_{2,11} & \displaystyle (-1)^K
\frac{\Delta \tc^{2K}}{4}\left(\frac{S}{2}-\frac{K\Delta}{4}\left(\frac{2}{i\tc}+\frac{4S}{\Delta}\right)\right)\\
\displaystyle (-1)^K\frac{\Delta}{4\tc^{2K}}\left(\frac{S}{2}+\frac{K\Delta}{4}\left(\frac{2}{i\tc}+\frac{4S}{\Delta}\right)\right) & \dot{O}_{2,22}\end{bmatrix}.
\label{eq:edge-dot-O2}
\end{equation}
Therefore, assuming that $x\in\mathcal{S}_\sigma\cap\mathscr{S}_m^\infty(\delta)$ and $K(x;\epsilon)\le M$ for some fixed constants $\sigma>0$, $\delta>0$, and $M>0$, if $K\ge 1$ then 
\begin{multline}
\epsilon^{(2+\epsilon)/(3\epsilon)}e^{-\lambda/\epsilon}\pu_m=\\
(-1)^K\frac{\Delta \tc^{2K}}{4}-\frac{\rho_0Q_+\phi_+}{1+\rho_0Q_+\phi_+}\cdot\frac{1}{\phi_+}
(\mathbf{F}_0\sigma_+\mathbf{F}_0^{-1})_{12}-\frac{\rho_0Q_-\phi_-}{1+\rho_0Q_-\phi_-}\cdot\frac{1}{\phi_-}(\mathbf{F}_0\sigma_-\mathbf{F}_0^{-1})_{12}+\mathcal{O}(\epsilon).
\end{multline}
Note that as a consequence of taking the matrix $\mathbf{E}_1$ in the form \eqref{eq:edge-E1-E2-sym},  the explicit terms in this formula are independent of the sign $s$ defined in \eqref{eq:edge-s-define}.  On the other hand, for $K=0$ we have 
\begin{equation}
\epsilon^{(2+\epsilon)/(3\epsilon)}e^{-\lambda/\epsilon}\pu_m=\frac{\Delta}{4}-\frac{\rho_0Q_+\phi_+}{1+\rho_0Q_+\phi_+}\cdot\frac{1}{\phi_+}(\mathbf{F}_0\sigma_+\mathbf{F}_0^{-1})_{12}+\mathcal{O}(\epsilon).
\end{equation} 
Therefore, using \eqref{eq:edge-F0-sigmapm-F0inverse-fraction} and \eqref{eq:edge-pert-plus}--\eqref{eq:edge-pert-minus}, we obtain for $K\ge 1$ that
\eq
\begin{split}
\epsilon^{(2+\epsilon)/(3\epsilon)}e^{-\lambda/\epsilon}\pu_m=-i
(-1)^K\tc^{2K}\Bigg(\frac{i\Delta}{4}&+\frac{\rc}{i\Delta}(\rc-S)(\mathcal{H}_K-1)\\
%\frac{(-1)^{m-K}e^{-2X_K/\epsilon}}{1-(-1)^{m-K}e^{-2X_K/\epsilon}}\\
  & -\frac{\rc}{i\Delta}(\rc+S)(\mathcal{H}_{K-1}+1)
%\frac{(-1)^{m-K+1}e^{2X_{K-1}/\epsilon}}{1-(-1)^{m-K+1}e^{2X_{K-1}/\epsilon}}
\Bigg)+\mathcal{O}(\epsilon),
\end{split}
\endeq
and for $K=0$ that
\begin{equation}
\epsilon^{(2+\epsilon)/(3\epsilon)}e^{-\lambda/\epsilon}\pu_m=
-i\left(\frac{i\Delta}{4}+\frac{\rc}{i\Delta}(\rc-S)(\mathcal{H}_0-1)\right)
%\frac{(-1)^me^{-2X_0/\epsilon}}{1-(-1)^me^{-2X_0/\epsilon}}
+\mathcal{O}(\epsilon),
\end{equation}
where
\begin{equation}
\mathcal{H}_n=\mathcal{H}_n(x;\epsilon):=\begin{cases}\coth(\epsilon^{-1}X_n),\quad &n\equiv m\pmod{2},\\
\tanh(\epsilon^{-1}X_n),\quad & n\not\equiv m\pmod{2}.
\end{cases}
\label{eq:Hn-edge}
\end{equation}
%Of course $K=K(x;\epsilon)$ has jump discontinuities along the curves $\Re(\mathfrak{d}(x))=-\tfrac{2n+1}{4}\epsilon\log(\epsilon^{-1})$, $n=0,1,2,3,\dots$, and we should check that any jump discontinuities in the above asymptotic formulae can be absorbed into the error terms (because $\pu_m$ itself has no jump discontinuities).  Choose $n=0,1,2,3,\dots$, and assume that $\Re(\mathfrak{d}(x))=-\tfrac{2n+1}{4}\epsilon\log(\epsilon^{-1})$.  Then neglecting terms that are $\mathcal{O}(\epsilon)$ for such $x$, the condition for matching is: 
%\begin{multline}
%(-1)^{n+1}\tc^{2(n+1)}\left(\frac{\Delta}{4}-\frac{2\rc}{\Delta}(\rc+S)\frac{(-1)^{m-n}e^{2X_n/\epsilon}}{1-(-1)^{m-n}e^{2X_n/\epsilon}}\right)\\
%{}=(-1)^n\tc^{2n}\left(\frac{\Delta}{4}-\frac{2\rc}{\Delta}(\rc-S)
%\frac{(-1)^{m-n}e^{-2X_n/\epsilon}}{1-(-1)^{m-n}e^{-2X_n/\epsilon}}\right),
%\end{multline}
%or put differently, by canceling common factors,
%\begin{equation}
%\tc^2\left[\frac{\Delta}{4}\left(1+\tc^{-2}\right)+\frac{2\rc}{\Delta}(\rc+S)\right]e^{X_n/\epsilon}-
%\left[\frac{\Delta}{4}\left(1+\tc^2\right)+\frac{2\rc}{\Delta}(\rc-S)\right](-1)^{m-n}e^{-X_n/\epsilon}=0.
%\end{equation}
%It is a direct calculation to show that the coefficients of the independent exponentials $e^{\pm X_n/\epsilon}$ both vanish identically.  This proves that
%the asymptotic formulae match up to $\mathcal{O}(\epsilon)$ where $K(x;\epsilon)$ is discontinuous.  
%
It follows from the construction of $\lambda=\lambda(x)$ summarized in \S\ref{section:g-define} that
%Recall that \textcolor{red}{(where does this come from)} 
for real $x>0$ we have $\Im(\lambda)=-\pi$.  Let
\begin{equation}
\mu(x):=\lambda(x)+i\pi.
\label{eq:edge-mu-def}
\end{equation}
Then $\mu$ is a Schwarz-symmetric analytic function of $x$ defined in the sector $\mathcal{S}_0$.
Using \eqref{epsilon}, we then arrive at the following asymptotic formulae for $\pu_m$:  if $K(x;\epsilon)\ge 1$, 
\eq
\begin{split}
m^{-2m/3}&e^{-m\mu}\pu_m \\
=e&^{-1/3-\mu/2}(-1)^{m-K}\tc^{2K}\Bigg(\frac{i\Delta}{4}+
\frac{\rc}{i\Delta}(\rc-S)(\mathcal{H}_K-1) -\frac{\rc}{i\Delta}(\rc+S)(\mathcal{H}_{K-1}+1)
\Bigg)+\mathcal{O}(m^{-1}),
\label{eq:U-edge-K-positive}
\end{split}
\endeq
while in the region where $K(x;\epsilon)=0$,
\begin{equation}
m^{-2m/3}e^{-m\mu}\pu_m=e^{-1/3-\mu/2}(-1)^m\left(\frac{i\Delta}{4}+\frac{\rc}{i\Delta}(\rc-S)
(\mathcal{H}_0-1)
%\frac{(-1)^me^{-2X_0/\epsilon}}{1-(-1)^me^{-2X_0/\epsilon}}
\right)+\mathcal{O}(m^{-1}).
\label{eq:U-edge-K-zero}
\end{equation}
Note that $\mu$, $i\Delta$, $\rc$, $S$, and $X_n$, $n=0,1,2,3,\dots$ are all real-valued for positive real $x$.  Similar calculations starting from \eqref{eq:edge-Vm-exact} yield
\eq
\begin{split}
m&^{2(m-1)/3}e^{m\mu}\pv_m \\
&=e^{1/3+\mu/2}\frac{(-1)^{m-K+1}}{\tc^{2K}}\Bigg(\frac{i\Delta}{4}+
\frac{\rc}{i\Delta}(\rc+S)(\mathcal{H}_K-1) -\frac{\rc}{i\Delta}(\rc-S)(\mathcal{H}_{K-1}+1)
\Bigg)+\mathcal{O}(m^{-1})
\label{eq:V-edge-K-positive}
\end{split}
\endeq
when $K(x;\epsilon)\geq 1$, while
\begin{equation}
m^{2(m-1)/3}e^{m\mu}\pv_m=e^{1/3+\mu/2}(-1)^{m+1}\left(\frac{i\Delta}{4}+\frac{\rc}{i\Delta}(\rc+S)(\mathcal{H}_0-1)\right)+\mathcal{O}(m^{-1})
\label{eq:V-edge-K-zero}
\end{equation}
when $K(x;\epsilon)=0$.

%
%These formulae may be simplified with the introduction of the trigonometric quantities $\mathcal{A}_n$ and $\mathcal{B}_n$, $n=0,1,2,3,\dots$:
%\begin{equation}
%\mathcal{A}_n:=\frac{(-1)^{m-n}e^{-2X_n/\epsilon}}{1-(-1)^{m-n}e^{-2X_n/\epsilon}}=
%\begin{cases}\displaystyle\tfrac{1}{2}\left(\coth(\epsilon^{-1}X_n)-1\right),&\quad n\equiv m\pmod{2}\\
%\displaystyle\tfrac{1}{2}\left(\tanh(\epsilon^{-1}X_n)-1\right),&\quad n\not\equiv m\pmod{2},
%\end{cases}
%\end{equation}
%while
%\begin{equation}
%\mathcal{B}_n:=\frac{(-1)^{m-n}e^{2X_n/\epsilon}}{1-(-1)^{m-n}e^{2X_n/\epsilon}}=
%\begin{cases}
%\displaystyle -\tfrac{1}{2}\left(\coth(\epsilon^{-1}X_n)+1\right),&\quad n\equiv m\pmod{2}\\
%\displaystyle-\tfrac{1}{2}\left(\tanh(\epsilon^{-1}X_n)+1\right),&\quad n\not\equiv m\pmod{2}.
%\end{cases}
%\end{equation}
%

Now we show that we can dispense with the technical device of introducing the integer-valued function $K(x;\epsilon)$.
\begin{theorem}[Asymptotics of $\pu_m$ and $\pv_m$ near an edge of $T$]
Define $\dot{\pu}^{\mathrm{Edge}}=\dot{\pu}^{\mathrm{Edge}}(x;m)$ and 
$\dot{\pv}^{\mathrm{Edge}}=\dot{\pv}^{\mathrm{Edge}}(x;m)$ by
\eq
\dot{\pu}^{\mathrm{Edge}}:=(-1)^me^{-1/3-\mu/2}\left[\frac{i\Delta}{4}+\frac{\rc}{i\Delta}(\rc-S)\sum_{n=0}^\infty(-1)^n\tc^{2n}(\mathcal{H}_n-1)\right]
\label{eq:dotu}
\endeq
and
\eq
\dot{\pv}^{\mathrm{Edge}}:=(-1)^{m+1}e^{1/3+\mu/2}\left[\frac{i\Delta}{4}+\frac{\rc}{i\Delta}(\rc+S)\sum_{n=0}^\infty\frac{(-1)^n}{\tc^{2n}}(\mathcal{H}_n-1)\right],
\label{eq:dotv}
\end{equation}
where $S=S(x)$, $\Delta=\Delta(x)$, $r_*=r_*(x)$, $t_*=t_*(x)$, 
$\mathcal{H}_n=\mathcal{H}_n(x;\epsilon)=\mathcal{H}_n(x;(m-\tfrac{1}{2})^{-1})$, and 
$\mu=\mu(x)$ are 
defined by \eqref{cubic-equation}, \eqref{eq:Delta-define}, \eqref{eq:edge-rc}, 
\eqref{eq:edge-tc}, \eqref{eq:Hn-edge}, and \eqref{eq:edge-mu-def}, respectively.
Recall that $\mathcal{S}_\sigma$, $\mathfrak{d}(x)$, and 
$\mathscr{S}_m^\infty(\delta)$ are defined by \eqref{eq:edge-sector-define}, 
\eqref{eq:d-define}, and \eqref{eq:edge-pole-cheese}, respectively, and that the 
Painlev\'e functions $\pu_m=\pu_m(y)$ and $\pv_m=\pv_m(y)$ are given for positive integer
$m$ in \eqref{backlund-positive}. 
Then, uniformly for $x\in\mathcal{S}_\sigma\cap\mathscr{S}_m^\infty(\delta)$ for fixed $\sigma>0$ and $\delta>0$, and $\Re(\mathfrak{d}(x))\ge -Mm^{-1}\log(m)$ for some fixed $M>0$,
\eq
m^{-2m/3}e^{-m\mu(x)}\pu_m\left(\left(m-\tfrac{1}{2}\right)^{2/3}x\right) = \dot{\pu}^{\mathrm{Edge}}(x;m) + \mathcal{O}(m^{-1}),\quad m\to +\infty
\label{eq:u-edge-approx}
\endeq
and
\eq
m^{2(m-1)/3}e^{m\mu(x)}\pv_m\left(\left(m-\tfrac{1}{2}\right)^{2/3}x\right) = \dot{\pv}^{\mathrm{Edge}}(x;m) + \mathcal{O}(m^{-1}),\quad m\to +\infty.
\label{eq:v-edge-approx}
\endeq
\label{theorem:edge}
\end{theorem}
\begin{proof}
Each $x$ to which the theorem applies corresponds to a finite value of $K(x;\epsilon)$, say $K_0\ge 0$.  The identities written down in Lemma~\ref{lemma-identities} imply that
\begin{equation}
-\tc^2\left[\frac{i\Delta}{4}-\frac{\rc}{i\Delta}(\rc+S)(\mathcal{H}_n+1)\right]=\frac{i\Delta}{4}+\frac{\rc}{i\Delta}(\rc-S)(\mathcal{H}_n-1)
\label{eq:second-identity}
\end{equation}
holds for each $n$.  Using this identity (repeatedly, starting with $n=0$) to rewrite the terms in the sum for $\dot{\pu}^\mathrm{Edge}$ with $n\le K_0-1$ yields
\eq
\begin{split}
(-1)^me^{1/3+\mu/2}\dot{\pu}^\mathrm{Edge} = (-1)^{K_0}\tc^{2K_0}\frac{i\Delta}{4}&+\frac{\rc}{i\Delta}(\rc-S)\sum_{n=K_0}^\infty
(-1)^n\tc^{2n}(\mathcal{H}_n-1)\\
  & - \frac{\rc}{i\Delta}(\rc+S)\sum_{n=0}^{K_0-1}(-1)^{n+1}\tc^{2(n+1)}(\mathcal{H}_n+1).
\end{split}
\endeq
Now we use the fact that $K(x;\epsilon)=K_0$ to see that, for some constant $C>0$ independent of $n$, $|\mathcal{H}_n-1|\le Cm^{K_0-n}$ while $|\mathcal{H}_n+1|\le Cm^{n-K_0+1}$.  It therefore follows (i) that the infinite series converges for $m$ sufficiently large by comparison to a geometric series and (ii) that if $K_0>0$ then
\begin{equation}
(-1)^me^{1/3+\mu/2}\dot{\pu}^\mathrm{Edge} = (-1)^{K_0}\tc^{2K_0}\left[\frac{i\Delta}{4}
+\frac{\rc}{i\Delta}(\rc-S)(\mathcal{H}_{K_0}-1)-\frac{\rc}{i\Delta}(\rc+S)(\mathcal{H}_{K_0-1}+1)\right]+\mathcal{O}(m^{-1}),
\end{equation}
while if $K_0=0$ then
\begin{equation}
(-1)^me^{1/3+\mu/2}\dot{\pu}^\mathrm{Edge} = \frac{i\Delta}{4}
+\frac{\rc}{i\Delta}(\rc-S)(\mathcal{H}_0-1)+\mathcal{O}(m^{-1}).
\end{equation}
By comparison with \eqref{eq:U-edge-K-positive} and \eqref{eq:U-edge-K-zero}, the asymptotic formula \eqref{eq:u-edge-approx} is proved.  The corresponding 
asymptotic formula \eqref{eq:v-edge-approx} for $\pv_m$ is proved in a 
similar manner with the help of the identity 
\begin{equation}
-\frac{1}{\tc^2}\left[\frac{i\Delta}{4}-\frac{\rc}{i\Delta}(\rc-S)(\mathcal{H}_n+1)\right]=\frac{i\Delta}{4}+\frac{\rc}{i\Delta}(\rc+S)(\mathcal{H}_n-1),
\end{equation}
which holds for each $n=0,1,2,\dots$ by the identities in 
Lemma~\ref{lemma-identities}.
\end{proof}
The approximation formula for $\pu_m$ is compared with the corresponding exact expression in Figures~\ref{fig:UEdge1} and \ref{fig:UEdge2}.
\begin{figure}[h]
%\begin{center}
%\includegraphics[width=2.5 in]{U10Edge1.pdf}%
%\hspace{0.5 in}%
%\includegraphics[width=2.5 in]{U20Edge1.pdf}
%\end{center}
\includegraphics{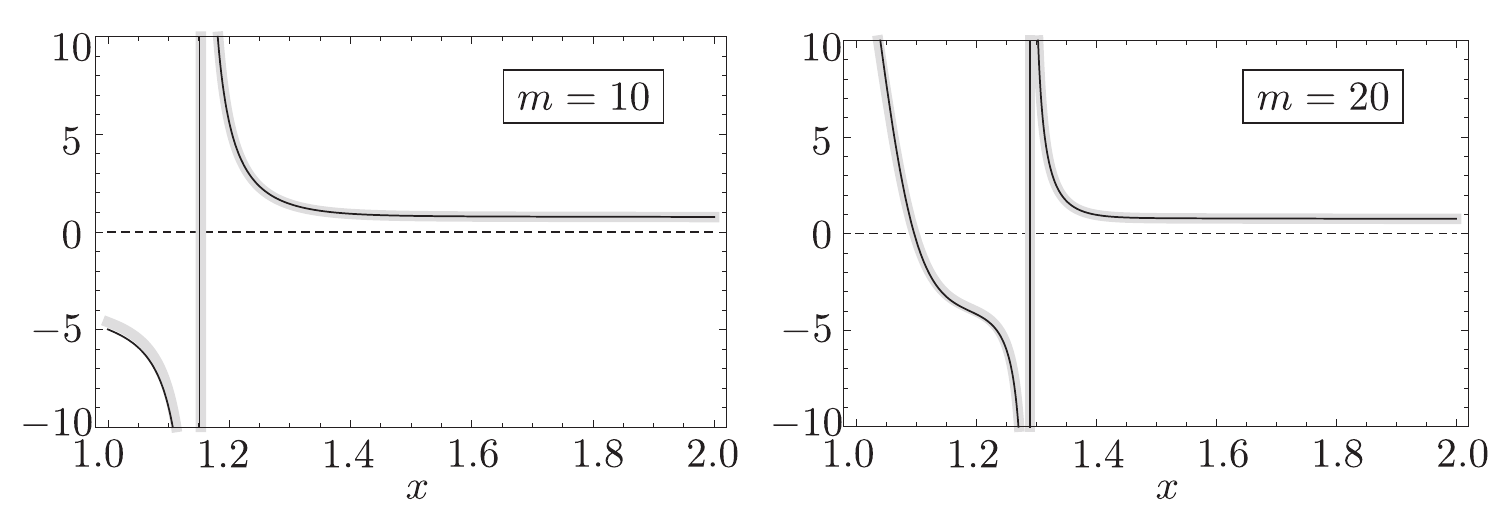}
\caption{The function $m^{-2m/3}e^{-m\mu}\pu_m$ plotted as a function of $x>0$ (thick gray curves) compared with $\dot{\pu}^\mathrm{Edge}$ (thin black curves) for $m=10$ (left) and $m=20$ (right).  The edge occurs at about $x=1.445$.}
\label{fig:UEdge1}
\end{figure}
\begin{figure}[h]
%\begin{center}
%\includegraphics[width=2.5 in]{U10Edge2.pdf}%
%\hspace{0.5 in}%
%\includegraphics[width=2.5 in]{U20Edge2.pdf}
%\end{center}
\includegraphics{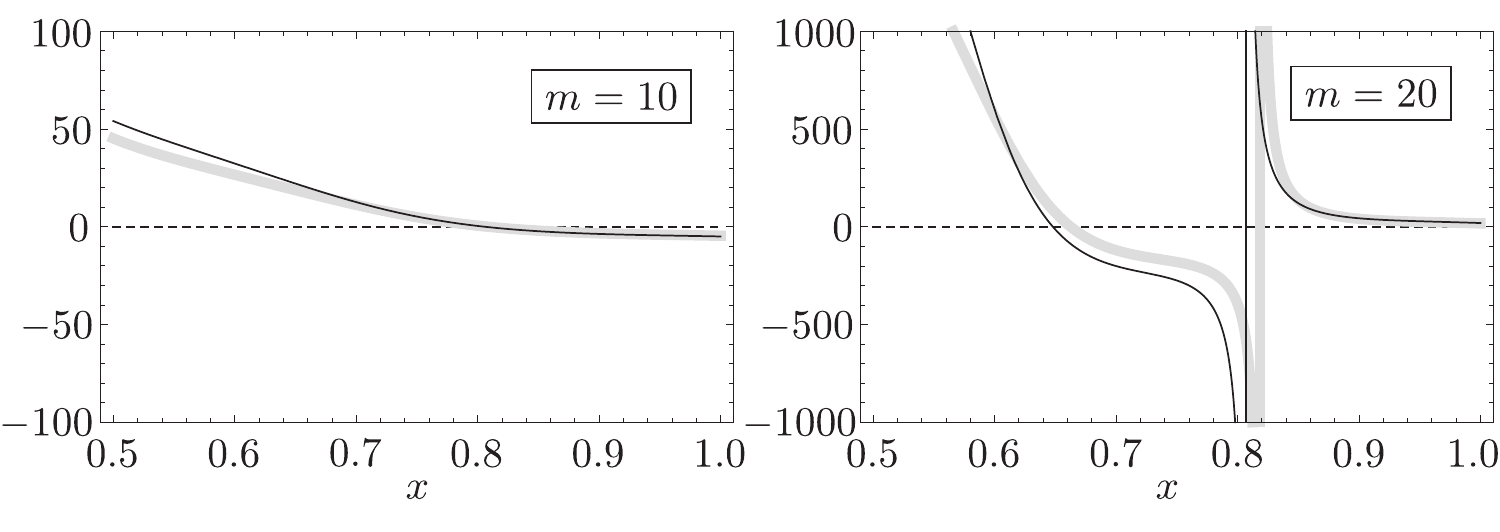}
\caption{The same as in Figure~\ref{fig:UEdge1}, but for $x$ further from the edge inside $T$.}
\label{fig:UEdge2}
\end{figure}

Next, consider the formula \eqref{eq:edge-pp-formula1} for $\pp_m$.  
Comparing with \eqref{eq:edge-Um-exact}, taking into account \eqref{epsilon} along with \eqref{eq:edge-mu-def}, and using Theorem~\ref{theorem:edge}, we see that the denominator in \eqref{eq:edge-pp-formula1} for $\pp_m$ can be written, up to a nonzero factor with modulus asymptotically independent of $m$, as $\dot{\pu}^\mathrm{Edge}+\mathcal{O}(m^{-1})$.  
The spacing between nearest zeros of $\dot{\pu}^\mathrm{Edge}$, for $x$ near $\partial T$ but bounded away from the corners, scales as $m^{-1}$, and therefore the denominator in \eqref{eq:edge-pp-formula1} will be bounded away from zero provided that $x$ is bounded away from each zero of 
$\dot{\pu}^\mathrm{Edge}$ by a distance of $\delta m^{-1}$ for some sufficiently small fixed $\delta>0$.  By analogy with the definition \eqref{eq:edge-pole-cheese} of $\mathscr{S}_m^\infty(\delta)$, we therefore define the set
\begin{equation}
\mathscr{S}_m^{0,\pu}(\delta):=\left\{y\in\mathcal{S}_0:\;|y-x|>\delta\epsilon\;\;\text{for all $x$ such that $\dot{\pu}^\mathrm{Edge}=0$}\right\}.
\label{eq:edge-zero-cheese-U}
\end{equation}
Assuming that $x\in\mathscr{S}_m^{0,\pu}(\delta)$ therefore controls the denominator, and then assuming that $x\in\mathcal{S}_\sigma\cap\mathscr{S}_m^\infty(\delta)$ shows that
%Under the additional assumption that $x$ is bounded away (by $\delta\epsilon$) from all zeros of $\pu_m$,
%$x\in\mathscr{S}_m(\delta)$ for some $\delta>0$ and that $K(x;\epsilon)\le M$ for some $M>0$, 
%the denominator is bounded away from zero, and also 
the matrix product $\mathbf{E}_1\dot{\mathbf{P}}_1$ is $\mathcal{O}(\epsilon)$ 
%for such $x$ 
as a consequence of \eqref{eq:edge-Qplus-Qminus-product}.  Moreover, elsewhere in \eqref{eq:edge-pp-formula1} 
%Elsewhere, 
the matrix coefficient $\mathbf{E}_j$ appears paired with the coefficient $\dot{\mathbf{P}}_j$, and this again implies a symmetry in the indices $s$ and $\overline{s}$ that results in an asymptotic formula independent of the value of $s=\pm$.
The formula we obtain in this way under the assumptions in force on $x$ is, for $K(x;\epsilon)\ge 1$,
\begin{equation}
m^{-1/3}\pp_m=\zc -\frac{1}{2}\rc\left(\mathcal{H}_K+\mathcal{H}_{K-1}
%\mathcal{B}_{K-1}-\mathcal{A}_K
\right) + 
\frac{2\rc(\rc-S)^2(\mathcal{H}_K-1)+2\rc(\rc+S)^2(\mathcal{H}_{K-1}+1)}{\Delta^2-4\rc(\rc-S)(\mathcal{H}_K-1)+4\rc(\rc+S)(\mathcal{H}_{K-1}+1)}+\mathcal{O}(m^{-1}),
\label{eq:p-tiles-K-positive-1}
\end{equation}
while if $K(x;\epsilon)=0$
\begin{equation}
m^{-1/3}\pp_m=\zc-\frac{1}{2}\rc(\mathcal{H}_0-1)+\frac{2\rc(\rc-S)^2(\mathcal{H}_0-1)}{\Delta^2-4\rc(\rc-S)(\mathcal{H}_0-1)}+\mathcal{O}(m^{-1}).
\label{eq:p-tiles-K-zero}
\end{equation}
The formula \eqref{eq:p-tiles-K-positive-1} may be rewritten in the following form:
\eq
\begin{split}
m^{-1/3}\pp_m=\zc-\frac{1}{2}\rc(\mathcal{H}_K+\mathcal{H}_{K-1}) + &
\frac{2\rc (\rc-S)^2(\mathcal{H}_K-1)}{\Delta^2-4\rc(\rc-S)(\mathcal{H}_K-1)}\\
{} & +\frac{2\rc(\rc+S)^2(\mathcal{H}_{K-1}+1)}{\Delta^2+4\rc(\rc+S)(\mathcal{H}_{K-1}+1)} +
\mathcal{O}(m^{-1}).
\label{eq:p-tiles-K-positive-2}
\end{split}
\endeq
Indeed, because the denominators in \eqref{eq:p-tiles-K-positive-1} and \eqref{eq:p-tiles-K-positive-2} are bounded away from zero due to the assumption that $x\in \mathscr{S}_m^{0,\pu}(\delta)$, the difference 
between the explicit terms is proportional to the product 
$(\mathcal{H}_K-1)(\mathcal{H}_{K-1}+1)$ by uniformly bounded factors, and 
this product is proportional to $m^{-1}$ (it is really the identity 
\eqref{eq:edge-Qplus-Qminus-product} in disguise).  Analogous calculations starting instead from \eqref{eq:edge-pq-formula1}
give
\eq
\begin{split}
m^{-1/3}\pq_m=-\zc-\frac{1}{2}\rc(\mathcal{H}_K+\mathcal{H}_{K-1}) + &
\frac{2\rc (\rc+S)^2(\mathcal{H}_K-1)}{\Delta^2-4\rc(\rc+S)(\mathcal{H}_K-1)}\\
{} & +\frac{2\rc(\rc-S)^2(\mathcal{H}_{K-1}+1)}{\Delta^2+4\rc(\rc-S)(\mathcal{H}_{K-1}+1)} +
\mathcal{O}(m^{-1})
\label{eq:q-tiles-K-positive}
\end{split}
\endeq
for $K(x;\epsilon)\geq 1$, and 
\begin{equation}
m^{-1/3}\pq_m=-\zc-\frac{1}{2}\rc(\mathcal{H}_0-1)+\frac{2\rc(\rc+S)^2(\mathcal{H}_0-1)}{\Delta^2-4\rc(\rc+S)(\mathcal{H}_0-1)}+\mathcal{O}(m^{-1})
\label{eq:q-tiles-K-zero}
\end{equation}
for $K(x;\epsilon)=0$.  Here we need to assume that $x\in \mathcal{S}_\sigma\cap\mathscr{S}_m^\infty(\delta)$ as well as $x\in\mathscr{S}_m^{0,\pv}(\delta)$, where
\begin{equation}
\mathscr{S}_m^{0,\pv}(\delta):=\left\{y\in\mathcal{S}_0:\;|y-x|>\delta\epsilon\;\text{for all $x$ such that $\dot{\pv}^\mathrm{Edge}=0$}\right\}.
\label{eq:edge-cheese-zero-V}
\end{equation}
Once again, the device of the integer-valued function $K(x;\epsilon)$ is artificial, as the following result shows.
\begin{theorem}[Asymptotics of $\pp_m$ and $\pq_m$ near an edge of $T$]
Define $\dot{\pp}^{\mathrm{Edge}}=\dot{\pp}^{\mathrm{Edge}}(x;m)$ by 
\begin{equation}
\dot{\pp}^\mathrm{Edge}:=\zc +\sum_{n=0}^\infty\left[-\frac{1}{2}\rc(\mathcal{H}_n+1)+
\frac{2\rc(\rc+S)^2(\mathcal{H}_n+1)}{\Delta^2+4\rc(\rc+S)(\mathcal{H}_n+1)}\right],
\label{eq:dotp-1}
\end{equation}
which can also be written in the form
\begin{equation}
\dot{\pp}^\mathrm{Edge}=\zc+\sum_{n=0}^\infty\left[-\frac{1}{2}\rc(\mathcal{H}_n-1)+\frac{2\rc(\rc-S)^2(\mathcal{H}_n-1)}{\Delta^2-4\rc(\rc-S)(\mathcal{H}_n-1)}\right].
\label{eq:dotp-2}
\end{equation}
Also define $\dot{\pq}^{\mathrm{Edge}}=\dot{\pq}^{\mathrm{Edge}}(x;m)$ by
\begin{equation}
\begin{split}
\dot{\pq}^\mathrm{Edge}&:=-\zc +\sum_{n=0}^\infty\left[-\frac{1}{2}\rc(\mathcal{H}_n+1)+\frac{2\rc(\rc-S)^2(\mathcal{H}_n+1)}{\Delta^2+4\rc(\rc-S)(\mathcal{H}_n+1)}\right] \\
  & = -\zc+\sum_{n=0}^\infty\left[-\frac{1}{2}\rc(\mathcal{H}_n-1)+\frac{2\rc(\rc+S)^2(\mathcal{H}_n-1)}{\Delta^2-4\rc(\rc+S)(\mathcal{H}_n-1)}\right].
\end{split}
\label{eq:dotq}
\end{equation}
Here $S=S(x)$, $\Delta=\Delta(x)$, $z_*=z_*(x)$, $r_*=r_*(x)$, and 
$\mathcal{H}_n=\mathcal{H}_n(x;\epsilon)=\mathcal{H}_n(x;(m-\tfrac{1}{2})^{-1})$ are 
defined by \eqref{cubic-equation}, \eqref{eq:Delta-define}, \eqref{eq:edge-zc}, 
\eqref{eq:edge-rc}, and \eqref{eq:Hn-edge}, respectively.
Recall that $\mathcal{S}_\sigma$, $\mathfrak{d}(x)$, $\mathscr{S}_m^\infty(\delta)$, 
$\mathscr{S}_m^{0,\pu}(\delta)$, and $\mathscr{S}_m^{0,\pv}(\delta)$ are defined by 
\eqref{eq:edge-sector-define}, \eqref{eq:d-define}, \eqref{eq:edge-pole-cheese}, 
\eqref{eq:edge-zero-cheese-U}, and \eqref{eq:edge-cheese-zero-V}, respectively, and 
that the Painlev\'e-II functions $\pp_m=\pp_m(y)$ and $\pq_m=\pq_m(y)$ are given 
for positive integer $m$ in \eqref{log-derivative}.
Then, uniformly for $x\in\mathcal{S}_\sigma\cap\mathscr{S}_m^\infty(\delta)\cap\mathscr{S}_m^{0,\pu}(\delta)$ for fixed $\sigma>0$ and $\delta>0$, and $\Re(\mathfrak{d}(x))\ge -Mm^{-1}\log(m)$ for some fixed $M>0$,
\begin{equation}
m^{-1/3}\pp_m\left(\left(m-\tfrac{1}{2}\right)^{2/3}x\right)=\dot{\pp}^\mathrm{Edge}(x;m)+\mathcal{O}(m^{-1}),\quad m\to +\infty
\label{eq:p-edge-approx}
\end{equation}
while, uniformly for $x\in\mathcal{S}_\sigma\cap\mathscr{S}_m^\infty(\delta)\cap\mathscr{S}_m^{0,\pv}(\delta)$ for fixed $\sigma>0$ and $\delta>0$, and $\Re(\mathfrak{d}(x))\ge -Mm^{-1}\log(m)$ for some fixed $M>0$,
\begin{equation}
m^{-1/3}\pq_m\left(\left(m-\tfrac{1}{2}\right)^{2/3}x\right)=\dot{\pq}^\mathrm{Edge}(x;m)+\mathcal{O}(m^{-1}),\quad m\to +\infty.
\label{eq:q-edge-approx}
\end{equation}
\label{theorem:edge-p-q}
\end{theorem}
\begin{proof}
To prove \eqref{eq:p-edge-approx} we first observe that the exact identity
\begin{equation}
-\frac{1}{2}\rc(\mathcal{H}_n+1)+\frac{2\rc(\rc+S)^2(\mathcal{H}_n+1)}{\Delta^2+4\rc(\rc+S)(\mathcal{H}_n+1)}=
-\frac{1}{2}\rc(\mathcal{H}_n-1)+\frac{2\rc(\rc-S)^2(\mathcal{H}_n-1)}{\Delta^2-4\rc(\rc-S)(\mathcal{H}_n-1)}
\end{equation}
is established by direct calculation, which in particular shows the equivalence of \eqref{eq:dotp-1} and \eqref{eq:dotp-2}.  Moreover, we can also use the identity to write $\dot{\pp}^\mathrm{Edge}$ as
\begin{multline}
\dot{\pp}^\mathrm{Edge}=\zc+\sum_{n=K_0}^\infty\left[-\frac{1}{2}\rc(\mathcal{H}_n-1)+\frac{2\rc(\rc-S)^2(\mathcal{H}_n-1)}{\Delta^2-4\rc(\rc-S)(\mathcal{H}_n-1)}\right]\\
{}+\sum_{n=0}^{K_0-1}\left[-\frac{1}{2}\rc(\mathcal{H}_n+1)
+\frac{2\rc(\rc+S)^2(\mathcal{H}_n+1)}{\Delta^2+4\rc(\rc+S)(\mathcal{H}_n+1)}\right].
\end{multline}
By analogous arguments as in the proof of Theorem~\ref{theorem:edge},
%used to analyze $\dot{\pu}^\mathrm{Edge}$, 
the condition $K(x;\epsilon)=K_0$ implies the convergence of the infinite sum and that if $K_0>0$ then
\begin{multline}
\dot{\pp}^\mathrm{Edge}=\zc-\frac{1}{2}\rc(\mathcal{H}_{K_0}+\mathcal{H}_{K_0-1})
+\frac{2\rc(\rc-S)^2(\mathcal{H}_{K_0}-1)}{\Delta^2-4\rc(\rc-S)(\mathcal{H}_{K_0}-1)} \\{}+
\frac{2\rc(\rc+S)^2(\mathcal{H}_{K_0-1}+1)}{\Delta^2+4\rc(\rc+S)(\mathcal{H}_{K_0-1}+1)} +
\mathcal{O}(m^{-1}),
\end{multline}
while if $K_0=0$ then
\begin{equation}
\dot{\pp}^\mathrm{Edge}=\zc-\frac{1}{2}\rc(\mathcal{H}_0-1)+\frac{2\rc(\rc-S)^2(\mathcal{H}_0-1)}{\Delta^2-4\rc(\rc-S)(\mathcal{H}_0-1)} + \mathcal{O}(m^{-1}).
\end{equation}
The proof of \eqref{eq:p-edge-approx} is then complete upon comparison with \eqref{eq:p-tiles-K-zero} and \eqref{eq:p-tiles-K-positive-2}.
The proof of \eqref{eq:q-edge-approx} follows exactly the same lines.
\end{proof}
The approximation formula for $\pp_m$ is compared with the corresponding exact expression in Figure~\ref{fig:PEdge}.
\begin{figure}[h]
%\begin{center}
%\includegraphics[width=2.5 in]{P10Edge.pdf}%
%\hspace{0.5 in}%
%\includegraphics[width=2.5 in]{P20Edge.pdf}
%\end{center}
\includegraphics{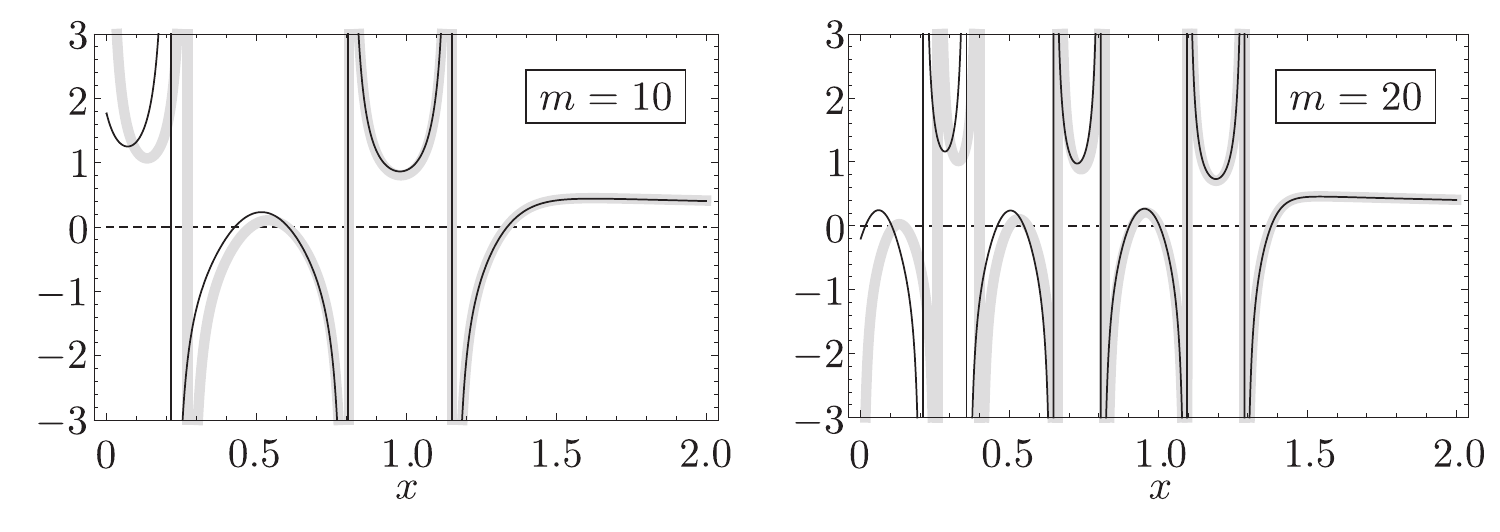}
\caption{The function $m^{-1/3}\pp_m$ plotted as a function of $x>0$ (thick gray curves) compared with $\dot{\pp}^\mathrm{Edge}$ (thin black curves) for $m=10$ (left) and $m=20$ (right).  The edge occurs at about $x=1.445$.}
\label{fig:PEdge}
\end{figure}

\begin{remark}
While such plots as shown in Figures~\ref{fig:UEdge1}--\ref{fig:PEdge} further demonstrate the accuracy of the infinite series formulae presented in Theorems~\ref{theorem:edge} and \ref{theorem:edge-p-q}, it should be mentioned that plots comparing $m^{-2m/3}e^{-m\mu}\pu_m$ with the ``piecewise'' approximations \eqref{eq:U-edge-K-positive}--\eqref{eq:U-edge-K-zero}, or comparing $m^{-1/3}\pp_m$ with the approximations
\eqref{eq:p-tiles-K-positive-1} (or \eqref{eq:p-tiles-K-positive-2}) and \eqref{eq:p-tiles-K-zero}
show poor agreement.  The approximating formulae are indeed asymptotically equivalent to the infinite series formulae (this is the content of the proofs of Theorems~\ref{theorem:edge} and \ref{theorem:edge-p-q}), however
 the constants implicit in the error estimates are evidently much larger for the ``piecewise'' approximations.  It seems that the key mechanism behind the relatively poor accuracy of the latter approximations is that the true locations of the poles and zeros of $\pu_m$ are, for moderate values of $m$, shifted right out of the ``$K$-windows'' to the edge of which they asymptotically converge.  Unless $m$ is extremely large, the true poles and zeros lie in the incorrect window to be visible under the piecewise approximation.  The infinite series formulae appear to circumvent this difficulty by including all of the singularities, even should they not appear in the correct window.  \myendrmk
\end{remark}

\begin{remark}
%\textcolor{red}{
%Give formulae for $\pv_m$ and $\pq_m$?  Carefully check the formulation of the theorem (the theorem should define the holes in the cheese correctly; the asymptotic formula for $\pu_m$ is valid near zeros, but that for $\pp_m$ is not near the zeros of $\pu_m$).  
%Then describe in a brief remark how B\"acklund transformations can be used to fill in the holes.  Just refer to the argument as it is written in the other paper, and point out how it depends on the shifting of the pole lattices as $m$ changes, and how this fails near the corners.}
The role of the condition $x\in\mathscr{S}_m^{0,\pu}(\delta)$ in the proof of validity of the asymptotic formula \eqref{eq:p-edge-approx} for $\pp_m$ and of the condition $x\in\mathscr{S}_m^{0,\pv}(\delta)$ in the proof of the formula \eqref{eq:q-edge-approx} for $\pq_m$ is in each case merely to bound the denominator in the approximation away from zero and therefore to allow error terms in the denominators of the exact formulae \eqref{eq:edge-pp-formula1}--\eqref{eq:edge-pq-formula1} to be expressed as absolute error terms.  These conditions may therefore be dropped at the cost of modifying the corresponding asymptotic formulae to maintain error terms in the denominators.

On the other hand, the condition $x\in\mathscr{S}_m^\infty(\delta)$ plays a much more essential role in the hypotheses of Theorems~\ref{theorem:edge} and \ref{theorem:edge-p-q}, because it bounds $x$ away from values at which the error matrix $\mathbf{E}(z)$ cannot be controlled at all.  However, even this condition can be dropped at the cost of an argument based on the asymptotic analysis of a different Riemann-Hilbert problem.  The main idea is to use the B\"acklund relation $\pu_m=\pv_{m+1}^{-1}$.  Since according to the arguments in \S\ref{section:edge-singularities} the pole lattices shift upon incrementing $m$ by a nonzero fraction of the lattice spacing when $x$ is bounded away from the corners of $T$, the identity 
\begin{equation}
(\mathscr{S}_m^\infty(\delta)\cap\mathcal{S}_\sigma)\cup (\mathscr{S}_{m+1}^\infty(\delta)\cap\mathcal{S}_\sigma)=\mathcal{S}_\sigma
\end{equation}
holds for all $m$ with any $\sigma>0$ as long as $\delta>0$ is sufficiently small.  Therefore, $\pu_m$ can be analyzed for all $x\in\mathcal{S}_\sigma$ by a combination of the results of Theorem~\ref{theorem:edge} applied for degrees $m$ and $m+1$, in the latter case by first approximating $\pv_{m+1}$ and then reciprocating to obtain an asymptotic formula for $\pu_m$.  Similar reasoning applies to all four families of rational Painlev\'e-II functions.  

These arguments are carried out in some detail in a more complicated setting in \cite{Buckingham-Miller-rational-noncrit}.  The method fails, however, when as $x$ approaches a corner point of $T$, the pole lattices fail to shift significantly as $m$ is replaced with $m+1$.  Near the corners, the pole lattice becomes ``frozen'' and the B\"acklund transformation cannot be applied to obtain asymptotics for the rational Painlev\'e-II functions near singularities of the approximation.  This approximation will be developed in \S\ref{section:cusp} below.
\label{remark:cheese}
\myendrmk
\end{remark}

\section{Analysis for $x$ near a corner point of $\partial T$}
\label{section:cusp}
In this section, we suppose that $x$ is close to the corner point $x_c<0$ of $\partial T$, and we derive asymptotic formulae for the rational Painlev\'e-II functions in the limit of large $m$.  Our goal is to 
make the formal rescaling argument given in \S\ref{section:results-summary} completely rigorous and
also to isolate the precise solution of Painlev\'e-I that is relevant.  We will also obtain corresponding asymptotic formulae for the functions $\pu_m$ and $\pv_m$.  We choose to recycle some symbols used in \S\ref{section:edge} to represent analogous objects, and we hope that once these are redefined in the present context there will be no confusion for the reader.

\subsection{The Negative-$x$ Configuration Riemann-Hilbert problem}
As $x_c$ lies on the negative real axis, we begin our analysis with the 
Riemann-Hilbert problem satisfied by ${\bf N}(z;x,\epsilon)$ in the 
Negative-$x$ Configuration as shown in 
\cite[Figure 20]{Buckingham-Miller-rational-noncrit}.  This contour 
arrangement is suitable for asymptotic analysis in the genus-zero region 
as long as $|\arg(-x)|<2\pi/3$.  As illustrated in Figure \ref{corner-break}, 
the local behavior in $z$ of the function $2h(z;x)+\lambda(x)$ (that is of 
use in the genus-zero region) changes dramatically near one of the band 
endpoints as $x\to x_c$.  To handle this change will require not only a 
different parametrix around this band endpoint but also the use of a 
modified $g$-function with generically different band endpoints we label 
$\mathfrak{a}$ and $\mathfrak{b}$.  Therefore we deform the Riemann-Hilbert 
problem if necessary so the points of self-intersection $a$ and $b$ are 
replaced by $\mathfrak{a}$ and $\mathfrak{b}$, respectively (we also redefine 
${\bf N}$ as the solution to this deformed problem).  Then 
${\bf N}(z;x,\epsilon)$ satisfies the normalization condition 
\eq
\lim_{z\to\infty}{\bf N}(z;x,\epsilon)(-z)^{-\sigma_3/\epsilon} = \mathbb{I}
\endeq
and the jump condition $\mathbf{N}_+=\mathbf{N}_-\mathbf{V}^{(\mathbf{N})}$, where 
$\mathbf{V}^{(\mathbf{N})}$ and the jump contour $\Sigma^{(\mathbf{N})}$ are shown in 
Figure~\ref{fig:N-jumps-pi}.  
\begin{figure}[h]
\setlength{\unitlength}{2pt}
\begin{center}
\begin{picture}(100,100)(-50,-50)
\put(-74,42){\framebox{$z$}}
\put(-35,0){\circle*{2}}
\put(-38,-4){$\mathfrak{a}$}
\put(35,0){\circle*{2}}
\put(36,-5){$\mathfrak{b}$}
\thicklines
\put(25,0){\line(1,0){30}}
\put(48,0){\vector(1,0){1}}
\put(57,-1){$\bbm -1 & -ie^{-\theta/\epsilon} \\ 0 & -1 \ebm$}
\put(45,3){$L$}
\qbezier(35,0)(20,20)(35,40)
\put(27.5,22){\vector(0,1){1}}
\put(30,20){$\bbm 1 & 0 \\ ie^{\theta/\epsilon} & 1 \ebm$}
\put(-25,0){\line(1,0){50}}
\put(-25,0){\vector(1,0){28}}
\put(0,-6){$\Sigma$}
\put(-20,8){$\bbm 0 & -ie^{-\theta/\epsilon} \\ -ie^{\theta/\epsilon} & 0 \ebm$}
\qbezier(-35,0)(-20,20)(-35,40)
\put(-27.5,22){\vector(0,1){1}}
\put(-56,20){$\bbm 1 & ie^{-\theta/\epsilon} \\ 0 & 1 \ebm$}
\put(-25,0){\line(-1,0){30}}
\put(-30,0){\vector(-1,0){18}}
\put(-80,-1){$\bbm 1 & 0 \\ ie^{\theta/\epsilon} & 1 \ebm$}
\qbezier(-35,0)(-20,-20)(-35,-40)
\put(-27.5,-22){\vector(0,-1){1}}
\put(-56,-22){$\bbm 1 & ie^{-\theta/\epsilon} \\ 0 & 1 \ebm$}
\qbezier(35,0)(20,-20)(35,-40)
\put(27.5,-22){\vector(0,-1){1}}
\put(30,-22){$\bbm 1 & 0 \\ ie^{\theta/\epsilon} & 1 \ebm$}
\end{picture}
\end{center}
\caption{\emph{The jump matrices $\mathbf{V^{(N)}}(z;x,\epsilon)$ for 
the function ${\bf N}(z;x,\epsilon)$ with $x$ near $x_c$.}}
\label{fig:N-jumps-pi}
\end{figure}
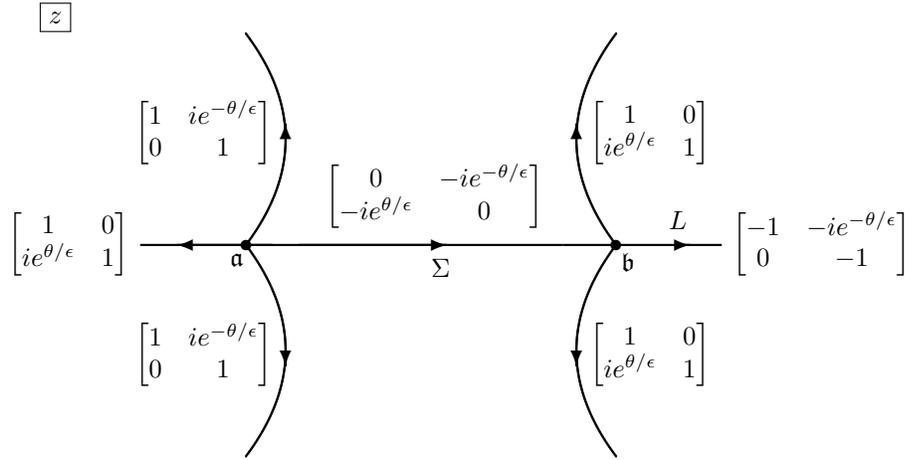

\subsection{The modified $g$-function for $x$ near $x_c$}
Recall the values $a_c$ and $b_c$ given by \eqref{eq:ac-bc-define}.  
Suppose $\mathfrak{a}$ and $\mathfrak{b}$ are close to $a_c$ and $b_c$, respectively,
and recall  the unique function $r(z;\mathfrak{a},\mathfrak{b})$ whose square is the quadratic
$(z-\mathfrak{a})(z-\mathfrak{b})$, whose domain of analyticity in $z$ is $\mathbb{C}\setminus\Sigma$, where $\Sigma$ is the oriented line segment from $z=\mathfrak{a}$ to $z=\mathfrak{b}$, and that satisfies $r(z;\mathfrak{a},\mathfrak{b})=z+\mathcal{O}(1)$ as $z\to\infty$.  Given complex constants $\mathfrak{m}$ and  $\nu$,
and a contour $L$ connecting $\mathfrak{b}$ with $\infty$ without intersecting $\Sigma$ and agreeing with the positive real axis for large $|z|$, let an analytic function $\mathfrak{g}:\mathbb{C}\setminus (\Sigma\cup L)\to\mathbb{C}$ be defined by
\begin{equation}
\mathfrak{g}(z):=\frac{1}{2}\theta(\mathfrak{a};x)+\frac{1}{2}\nu +\int_\mathfrak{a}^z\left[\frac{1}{2}\theta'(\zeta;x)-\frac{3}{2}\left(\zeta+\frac{1}{2}(\mathfrak{a}+\mathfrak{b})+\frac{\mathfrak{m}}{\zeta-\mathfrak{a}}\right)r(\zeta;\mathfrak{a},\mathfrak{b})\right]\,d\zeta,
\label{eq:ghat-define}
\end{equation}
where the path of integration is arbitrary in the simply-connected domain $\mathbb{C}\setminus(\Sigma\cup L)$, and where $\theta$ is defined by \eqref{eq:theta-define}.

This function has the following elementary properties, as are easily verified.  The boundary values
taken by $\mathfrak{g}$ on the segment $\Sigma$ satisfy
\begin{equation}
\mathfrak{g}_+(z)+\mathfrak{g}_-(z)=\theta(z;x)+\nu,\quad z\in\Sigma.
\label{eq:ghat-on-Sigh}
\end{equation}
Similarly, taking the contour $L$ to be oriented in the direction away from $\mathfrak{b}$, 
\begin{equation}
\begin{split}
\mathfrak{g}_+(z)-\mathfrak{g}_-(z)&=\frac{3}{2}\oint\left(\zeta+\frac{1}{2}(\mathfrak{a}+\mathfrak{b})+\frac{\mathfrak{m}}{\zeta-\mathfrak{a}}\right)r(\zeta;\mathfrak{a},\mathfrak{b})\,d\zeta\\
&=-\frac{3\pi i}{8}(\mathfrak{b}-\mathfrak{a})(\mathfrak{b}^2-\mathfrak{a}^2+4\mathfrak{m}),\quad z\in L,
\end{split}
\label{eq:ghat-on-Lh}
\end{equation}
where the contour of integration is a positively-oriented loop that encloses $\Sigma$, and the second line follows by evaluation of the integral by residues at $\zeta=\infty$.  Furthermore, considering the asymptotic behavior of $\mathfrak{g}(z)$ for large $z$ one obtains
\begin{equation}
\mathfrak{g}(z)=\mathfrak{g}_1z +\mathfrak{g}_\ell\log(-z)+\mathfrak{g}_0 + \mathcal{O}(z^{-1}),\quad z\to\infty,
\end{equation}
where
\begin{equation}
\mathfrak{g}_1:=\frac{1}{16}(9\mathfrak{a}^2+6\mathfrak{a}\mathfrak{b}+9\mathfrak{b}^2+8x-24\mathfrak{m}),
\end{equation}
\begin{equation}
\mathfrak{g}_\ell:=\frac{3}{16}(\mathfrak{b}-\mathfrak{a})(\mathfrak{b}^2-\mathfrak{a}^2+4\mathfrak{m}),
\end{equation}
and $\mathfrak{g}_0$ is defined in terms of a convergent integral as
\begin{multline}
\mathfrak{g}_0:=\frac{1}{2}\theta(\mathfrak{a};x)+\frac{1}{2}\nu -\mathfrak{g}_1\mathfrak{a}-\mathfrak{g}_\ell\log(-\mathfrak{a})\\
{}+ \int_\mathfrak{a}^\infty
\left[\frac{1}{2}\theta'(\zeta;x)-\frac{3}{2}\left(\zeta+\frac{1}{2}(\mathfrak{a}+\mathfrak{b})+\frac{\mathfrak{m}}{\zeta-\mathfrak{a}}\right)r(\zeta;\mathfrak{a},\mathfrak{b})-\mathfrak{g}_1-\frac{\mathfrak{g}_\ell}{\zeta}\right]\,d\zeta.
\end{multline}

\begin{lemma}
There exist unique functions $\mathfrak{b}=\mathfrak{b}(x,\mathfrak{a})$ and $\mathfrak{m}=\mathfrak{m}(x,\mathfrak{a})$, analytic in $(x,\mathfrak{a})$ near $(x_c,a_c)$, for which $\mathfrak{b}(x_c,a_c)=b_c$ and $\mathfrak{m}(x_c,a_c)=0$ hold and such that 
\begin{equation}
\mathfrak{g}_1=0\quad\text{and}\quad\mathfrak{g}_\ell=1
\end{equation}
both hold as identities in $(x,\mathfrak{a})$ near $(x_c,a_c)$.  Also, $\mathfrak{b}_\mathfrak{a}(x_c,a_c)=\mathfrak{m}_\mathfrak{a}(x_c,a_c)=0$ and $\mathfrak{b}_x(x_c,a_c)=-6^{-2/3}/2$, while $\mathfrak{m}_x(x_c,a_c)=1/6$.
\label{lemma:g1gell}
\end{lemma}
\begin{proof}
From the equation $\mathfrak{g}_1=0$, $\mathfrak{m}$ may be eliminated in favor of $x$, $\mathfrak{a}$, and $\mathfrak{b}$:
\begin{equation}
\mathfrak{m}=\frac{3}{8}\mathfrak{a}^2+\frac{1}{4}\mathfrak{a}\mathfrak{b}+\frac{3}{8}\mathfrak{b}^2+\frac{1}{3}x.
\label{eq:m-1}
\end{equation}
Thus, the equation $\mathfrak{g}_\ell=1$ becomes a cubic equation in $\mathfrak{b}$ with coefficients depending analytically on $(x,\mathfrak{a})$:
\begin{equation}
P(\mathfrak{b};x,\mathfrak{a}):=\frac{5}{2}\mathfrak{b}^3-\frac{3}{2}\mathfrak{a}\mathfrak{b}^2
+\left(\frac{4}{3}x-\frac{1}{2}\mathfrak{a}^2\right)\mathfrak{b} -\left(\frac{16}{3}+\frac{1}{2}\mathfrak{a}^3 +\frac{4}{3}x\mathfrak{a}\right)=0.
\end{equation}
It is a direct calculation to confirm that $P(b_c;x_c,a_c)=0$.  Moreover,
\begin{equation}
P_\mathfrak{b}(b_c;x_c,a_c)=\frac{15}{2}b_c^2-3a_cb_c+\frac{4}{3}x_c-\frac{1}{2}a_c^2=32\left(\frac{2}{9}\right)^{1/3}\neq 0,
\end{equation}
from which the Implicit Function Theorem yields the desired function $\mathfrak{b}=\mathfrak{b}(x,\mathfrak{a})$.  Then, from \eqref{eq:m-1} one obtains
\begin{equation}
\mathfrak{m}=\mathfrak{m}(x,\mathfrak{a})=\frac{3}{8}\mathfrak{a}^2+\frac{1}{4}\mathfrak{a}\mathfrak{b}(x,\mathfrak{a})+\frac{3}{8}\mathfrak{b}(x,\mathfrak{a})^2+\frac{1}{3}x,
\label{eq:m-in-terms-of-a-b-x}
\end{equation}
and by direct calculation using $\mathfrak{b}(x_c,a_c)=b_c$ it then follows that $\mathfrak{m}(x_c,a_c)=0$.  Finally, the claimed values of the derivatives follow from the formulae $\mathfrak{b}_\mathfrak{a}(x_c,a_c)=-P_\mathfrak{a}(b_c;x_c,a_c)/P_\mathfrak{b}(b_c;x_c,a_c)$ and $\mathfrak{b}_x(x_c,a_c)=-P_x(b_c;x_c,a_c)/P_\mathfrak{b}(b_c;x_c,a_c)$ and from the chain rule applied 
to \eqref{eq:m-in-terms-of-a-b-x}.
\end{proof}

Taking $\mathfrak{b}=\mathfrak{b}(x,\mathfrak{a})$ and $\mathfrak{m}=\mathfrak{m}(x,\mathfrak{a})$
as in Lemma~\ref{lemma:g1gell}, and further uniquely choosing $\nu$ as a function of $(x,\mathfrak{a})$ so that $\mathfrak{g}_0=0$, we obtain a function $\mathfrak{g}(z)=\mathfrak{g}(z;x,\mathfrak{a})$ satisfying \eqref{eq:ghat-on-Sigh}, the jump condition
\begin{equation}
\mathfrak{g}_+(z;x,\mathfrak{a})-\mathfrak{g}_-(z;x,\mathfrak{a})=-2\pi i,\quad z\in L
\end{equation}
(which follows from \eqref{eq:ghat-on-Lh} upon using $\mathfrak{g}_\ell=1$), and the asymptotic condition
\begin{equation}
\mathfrak{g}(z;x,\mathfrak{a})=\log(-z)+\mathcal{O}(z^{-1}),\quad z\to\infty.
\end{equation}
Associated with $\mathfrak{g}$ is the function $\mathfrak{h}$ defined as 
\begin{equation}
\mathfrak{h}(z)=\mathfrak{h}(z;x,\mathfrak{a}):=\frac{1}{2}\theta(z;x)-\mathfrak{g}(z;x,\mathfrak{a}).
\label{eq:h-corner-define}
\end{equation}
It follows from \eqref{eq:h-corner-define} and \eqref{eq:ghat-define} that
\begin{equation}
\mathfrak{h}(z;x,\mathfrak{a})+\frac{1}{2}\nu(x,\mathfrak{a})=\frac{3}{2}\int_{\mathfrak{a}}^z
\left(\zeta+\frac{1}{2}(\mathfrak{a}+\mathfrak{b}(x,\mathfrak{a}))+\frac{\mathfrak{m}(x,\mathfrak{a})}{\zeta-\mathfrak{a}}\right)r(\zeta;\mathfrak{a},\mathfrak{b}(x,\mathfrak{a}))\,d\zeta.
\label{eq:hath-sum}
\end{equation}

It is also easy to check that $\mathfrak{g}$ and $\mathfrak{h}$ are continuous with respect to $(x,\mathfrak{a})$ in the topology of uniform convergence on compact sets in the $z$-plane that are disjoint from $\Sigma\cup L$, and that
\begin{equation}
\mathfrak{g}(z;x_c,a_c)=g(z;x_c),\quad\nu(x_c,a_c)=\lambda(x_c),\quad\text{and}\quad
\mathfrak{h}(z;x_c,a_c)=h(z;x_c),
\label{eq:g-h-nu-critical}
\end{equation}
that is, the new $g$-function $\mathfrak{g}$ coincides with the old $g$-function $g$ (the one described in \S\ref{section:g-define} used to study the rational Painlev\'e-II functions for $x$ outside of $T$ as well as for $x$ near a smooth point of $\partial T$) in the limit $(x,\mathfrak{a})\to (x_c,a_c)$.
%The function $\mathfrak{g}$ will be used to analyze the rational 
%Painlev\'e-II functions under the assumption that $x$ is close to the corner 
%point $x_c$ of $T$.  It is introduced as follows.  We take the 
%Riemann-Hilbert problem for the matrix $\mathbf{N}(z;x,\epsilon)$ in the 
%``negative $x$ configuration'' as defined in 
%\cite{Buckingham-Miller-rational-noncrit} where the point $a$ is replaced 
%with $\mathfrak{a}$ and the point $b$ is replaced by 
%$\mathfrak{b}(x,\mathfrak{a})$.  

At this point $\mathfrak{g}(z;x,\mathfrak{a})$ and $\nu(x,\mathfrak{a})$ are 
completely specified assuming $\mathfrak{a}$ is given.  In Lemma 
\ref{lemma:conformal} below we will determine $\mathfrak{a}$ as a function of 
$x$ (therefore giving $\mathfrak{g}$ as a function of $z$ and $x$ and 
$\nu$ as a function of $x$).  For now, we make the change of variables 
\begin{equation}
\mathbf{O}(z):=e^{-\nu(x,\mathfrak{a})\sigma_3/(2\epsilon)}{\mathbf N}(z;x,\epsilon)e^{-\mathfrak{g}(z;x,\mathfrak{a})\sigma_3/\epsilon}e^{\nu(x,\mathfrak{a})\sigma_3/(2\epsilon)}.
\end{equation}
Now
\eq
\lim_{z\to\infty}{\bf O}(z;x,\epsilon)=\mathbb{I}, 
\endeq
and the domain of analyticity is the same for both $\mathbf{O}(z)$ and 
${\mathbf N}(z)$.  The jump conditions satisfied by $\mathbf{O}(z)$ 
are illustrated in Figure~\ref{fig:hatO-jumps-corner}.
Note that we take the angles between contours meeting at $z=\mathfrak{a}$ and $z=\mathfrak{b}$ to be locally as indicated in the figure (we do not directly specify the angles involving the segment $\Sigma$).  
\begin{figure}[h]
\setlength{\unitlength}{2pt}
\begin{center}
\begin{picture}(100,100)(-50,-50)
\put(-74,42){\framebox{$z$}}
\put(-35,0){\circle*{2}}
\put(-44,4){$\tfrac{2}{5}\pi$}
\put(-44,-6){$\tfrac{2}{5}\pi$}
\put(-34,1){$\mathfrak{a}$}
\put(35,0){\circle*{2}}
\put(35,4){$\tfrac{2}{3}\pi$}
\put(35,-6){$\tfrac{2}{3}\pi$}
\put(29,1){$\mathfrak{b}$}
\thicklines
\put(25,0){\line(1,0){30}}
\put(48,0){\vector(1,0){1}}
\put(57,-1){$\bbm 1 & ie^{-(2\mathfrak{h}_++\nu)/\epsilon} \\ 0 & 1 \ebm$}
\put(45,3){$L$}
\qbezier(35,0)(20,20)(35,40)
\put(27.5,22){\vector(0,1){1}}
\put(30,20){$\bbm 1 & 0 \\ ie^{(2\mathfrak{h}+\nu)/\epsilon} & 1 \ebm$}
\put(-25,0){\line(1,0){50}}
\put(-25,0){\vector(1,0){28}}
\put(0,-7){$\Sigma$}
\put(-10,8){$\bbm 0 & -i \\ -i & 0 \ebm$}
%
%\qbezier(-35,0)(-10,20)(-35,40)
%\put(-22.5,22){\vector(0,1){1}}
%%\qbezier(-35,0)(-41.5568,20.1798)(-53,40)
\put(-35,0){\line(-1,3){14}}
\put(-35,0){\vector(-1,3){8}}
\put(-40,25){$\bbm 1 & ie^{-(2\mathfrak{h}+\nu)/\epsilon} \\ 0 & 1 \ebm$}
\put(-25,0){\line(-1,0){30}}
\put(-30,0){\vector(-1,0){18}}
\put(-91,-1){$\bbm 1 & 0 \\ ie^{(2\mathfrak{h}+\nu)/\epsilon} & 1 \ebm$}
%
%\qbezier(-35,0)(-10,-20)(-35,-40)
%\put(-22.5,-22){\vector(0,-1){1}}
\put(-35,0){\line(-1,-3){14}}
\put(-35,0){\vector(-1,-3){8}}
\put(-40,-27){$\bbm 1 & ie^{-(2\mathfrak{h}+\nu)/\epsilon} \\ 0 & 1 \ebm$}
\qbezier(35,0)(20,-20)(35,-40)
\put(27.5,-22){\vector(0,-1){1}}
\put(30,-22){$\bbm 1 & 0 \\ ie^{(2\mathfrak{h}+\nu)/\epsilon} & 1 \ebm$}
\end{picture}
\end{center}
\caption{The jump matrices for $\mathbf{O}(z)$ for 
$x$ near $x_c$.}
\label{fig:hatO-jumps-corner}
\end{figure}
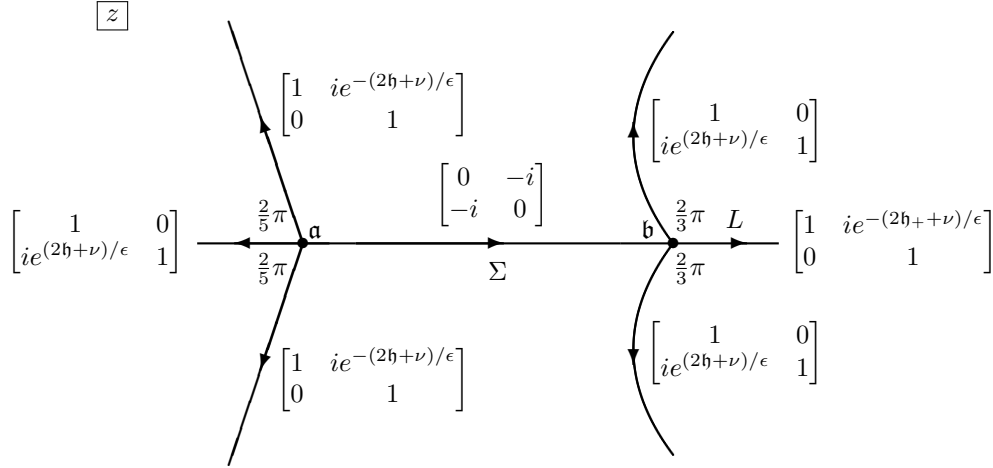
It follows from \eqref{eq:g-h-nu-critical}, continuity 
of $2\mathfrak{h}(z)+\nu$, and the results of \cite[Section 3]{Buckingham-Miller-rational-noncrit} that for $x$ near $x_c$ and $\mathfrak{a}$ near $a_c$, all triangular jump matrices for $\mathbf{O}(z)$ converge to the identity matrix as $\epsilon\to 0$ as long as fixed neighborhoods of the limiting endpoints $(a_c,b_c)$ of $\Sigma$ are excluded.  Moreover, the convergence is exponentially fast and uniform outside of the aforementioned neighborhoods.  This suggests that to approximate $\mathbf{O}(z)$ accurately, it will be necessary only to deal with the constant jump across $\Sigma$ as well as fixed-size neighborhoods of $a_c$ and $b_c$.

\subsection{The outer parametrix}
To deal with the jump of $\mathbf{O}$ across the segment $\Sigma$ we simply adapt the solution of Riemann-Hilbert Problem~\ref{rhp:outer-model-edge} with $K=0$ to the present situation.  Replacing $a(x)$ with $\mathfrak{a}$ and $b(x)$ with $\mathfrak{b}(x,\mathfrak{a})$, we use 
\eqref{eq:edge-beta-define} to define a function $\beta(z;x,\mathfrak{a})$ analytic for $z\in\mathbb{C}\setminus\Sigma$, and then by the formula \eqref{Odot-gen0} we obtain the matrix function $\dot{\mathbf{O}}^{(\mathrm{out})}(z)$.  

This outer parametrix has the following properties.  Firstly, $\dot{\mathbf{O}}^{(\mathrm{out})}(z)$ is analytic for $z\in\mathbb{C}\setminus\Sigma$.  Secondly, $\dot{\mathbf{O}}^{(\mathrm{out})}(z)-\mathbb{I}$ has a convergent Laurent expansion for sufficiently large $|z|$.
Thirdly, for $z\in\Sigma$, the boundary values are related by $\dot{\mathbf{O}}^{(\mathrm{out})}_+(z) =
\dot{\mathbf{O}}^{(\mathrm{out})}_-(z)(-i\sigma_1)$.  Fourthly, $\dot{\mathbf{O}}^{(\mathrm{out})}(z)$ is independent of $\epsilon$ and is uniformly bounded for $z$ bounded away from $a_c$ and $b_c$ (assuming that $x-x_c$ and $\mathfrak{a}-a_c$ are sufficiently small).  Further properties of $\dot{\mathbf{O}}^{(\mathrm{out})}(z)$ concerning its behavior near the points $\mathfrak{a}$ and $\mathfrak{b}(x,\mathfrak{a})$ will be developed below in the discussion of inner parametrices. 

\subsection{The inner (Airy) parametrix near $z=b_c$}
We closely follow \cite[Section 3.6.2]{Buckingham-Miller-rational-noncrit} using the ``Negative-$x$ Configuration''.  Let $\mathbb{D}_\mathfrak{b}$ denote a disk of sufficiently small radius independent of $\epsilon$ that contains the point $z=\mathfrak{b}(x,\mathfrak{a})$ as long as 
$x-x_c$ and $\mathfrak{a}-a_c$ are sufficiently small.  Because $2\mathfrak{h}(z)+\nu$ vanishes at $z=\mathfrak{b}(x,\mathfrak{a})$ like $(z-\mathfrak{b})^{3/2}$, making the substitutions of $\mathfrak{a}$ for $a(x)$, $\mathfrak{b}(x,\mathfrak{a})$ for $b(x)$, $\mathfrak{h}$ for $h$, and $\nu$ for $\lambda$, the construction of \cite[Section 3.6.2]{Buckingham-Miller-rational-noncrit} yields a matrix denoted $\dot{\mathbf{O}}^{(\mathfrak{b})}(z)$ that satisfies exactly the same jump conditions as does $\mathbf{O}(z)$ within the disk $\mathbb{D}_\mathfrak{b}$, and that satisfies the uniform estimate
\begin{equation}
\dot{\mathbf{O}}^{(\mathfrak{b})}(z)\dot{\mathbf{O}}^{(\mathrm{out})}(z)^{-1}=\mathbb{I}+\mathcal{O}(\epsilon),\quad z\in\partial\mathbb{D}_\mathfrak{b}.
\label{eq:corner-airy-match}
\end{equation}
No further details will be required for our analysis, as the dominant source of error terms will come from a neighborhood of the other endpoint $z=\mathfrak{a}$.  Approximation of $\mathbf{O}(z)$ near this point is the topic we take up next.
\subsection{The inner (Painlev\'e-I tritronqu\'ee) parametrix near $z=a_c$}

\subsubsection{Conformal coordinate near $z=a_c$}
Let $q$ denote the principal branch of the square root $(\mathfrak{a}-z)^{1/2}$.  The function $\mathfrak{h}(z;x,\mathfrak{a})+\tfrac{1}{2}\nu(x,\mathfrak{a})$ can be written as an odd analytic function of $q$ using \eqref{eq:hath-sum}:
\begin{equation}
\mathfrak{h}(z;x,\mathfrak{a})+\frac{1}{2}\nu(x,\mathfrak{a})=f(q;x,\mathfrak{a}):=-3\int_0^q
\left(\mathfrak{m}(x,\mathfrak{a})-\frac{1}{2}(3\mathfrak{a}+\mathfrak{b}(x,\mathfrak{a}))w^2+w^4\right)(\mathfrak{b}(x,\mathfrak{a})-\mathfrak{a}+w^2)^{1/2}\,dw.
\end{equation}
The integrand is analytic and even near $w=0$ and this implies the claimed behavior of $f$.  The first several generally nonzero Taylor coefficients of $f$ are:
\begin{equation}
\begin{split}
f'(0;x,\mathfrak{a})&=  -3\mathfrak{m}(x,\mathfrak{a})(\mathfrak{b}(x,\mathfrak{a})-\mathfrak{a})^{1/2},\\
\frac{1}{3!}f'''(0;x,\mathfrak{a})&=\frac{1}{2}(\mathfrak{b}(x,\mathfrak{a})-\mathfrak{a})^{-1/2}
\left[(\mathfrak{b}(x,\mathfrak{a})-\mathfrak{a})(3\mathfrak{a}+\mathfrak{b}(x,\mathfrak{a}))-\mathfrak{m}(x,\mathfrak{a})\right],\\
\frac{1}{5!}f^{(5)}(0;x,\mathfrak{a})&=\frac{3}{40}(\mathfrak{b}(x,\mathfrak{a})-\mathfrak{a})^{-3/2}
\left[\mathfrak{m}(x,\mathfrak{a})+2(\mathfrak{b}(x,\mathfrak{a})-\mathfrak{a})(3\mathfrak{a}+
\mathfrak{b}(x,\mathfrak{a}))-8(\mathfrak{b}(x,\mathfrak{a})-\mathfrak{a})^2\right],\\
\frac{1}{7!}f^{(7)}(0;x,\mathfrak{a})&=-\frac{3}{112}(\mathfrak{b}(x,\mathfrak{a})-\mathfrak{a})^{-5/2}
\left[8(\mathfrak{b}(x,\mathfrak{a})-\mathfrak{a})^2+(\mathfrak{b}(x,\mathfrak{a})-\mathfrak{a})(3\mathfrak{a}+\mathfrak{b}(x,\mathfrak{a}))+\mathfrak{m}(x,\mathfrak{a})\right].
\end{split}
\label{eq:fderivs}
\end{equation}
Since $\mathfrak{m}(x_c,a_c)=0$ and $\mathfrak{b}(x_c,a_c)=b_c=-3a_c$, we see that
\begin{equation}
f'(0;x_c,a_c)=0,\quad \frac{1}{3!}f'''(0;x_c,a_c)=0,\quad\text{but}\quad
\frac{1}{5!}f^{(5)}(0;x_c,a_c)=-\frac{6^{5/6}}{5}<0.
\label{eq:fderivs-criticality}
\end{equation}
Therefore, while $f$ generally vanishes linearly at $q=0$, when $x=x_c$ and $\mathfrak{a}=a_c$ it vanishes there to higher (quintic) order.  
For future reference we also record here the value
\begin{equation}
\frac{1}{7!}f^{(7)}(0;x_c,a_c)=-\frac{3^{7/6}}{2^{11/6}7}.
\label{eq:f7-criticality}
\end{equation}
From \eqref{eq:fderivs} and recalling Lemma~\ref{lemma:g1gell}, we also see that
\begin{equation}
f_\mathfrak{a}'(0;x_c,a_c)=0\quad\text{but}\quad
\frac{1}{3!}f_\mathfrak{a}'''(0;x_c,a_c)=\left(\frac{243}{2}\right)^{1/6}\neq 0
\label{eq:fa-derivs-criticality}
\end{equation}
and
\begin{equation}
f_x'(0;x_c,a_c)=-6^{-1/6}.
\label{eq:fx-deriv-criticality}
\end{equation}

We wish to introduce a conformal map $z\mapsto W$ defined in a disk in the $z$-plane containing the point $z=a_c$ such that $\mathfrak{h}+\tfrac{1}{2}\nu$ takes a simple form in terms of $W$.  Precisely, we will show that $W$ can be found so that
\begin{equation}
\mathfrak{h}(z;x,\mathfrak{a})+\frac{1}{2}\nu(x,\mathfrak{a})=-\frac{4}{5}(-W)^{5/2}-s(-W)^{1/2},
\label{eq:basic-h-equation}
\end{equation}
where on the right-hand side the power functions denote principal branches defined for $|\arg(-W)|<\pi$.  We wish for this equation to hold as an identity along each of the jump contours for $\mathbf{O}$ near $z=\mathfrak{a}$ for which the corresponding jump matrix depends on $z$.

The complex parameter $s$ is intended to allow the degeneration of the model function on the right-hand side of \eqref{eq:basic-h-equation} from a generic square-root vanishing to a $\tfrac{5}{2}$-power vanishing as occurs on the left-hand side as $(x,\mathfrak{a})\to (x_c,a_c)$.  Therefore, we should expect that $s$ will need to depend on $(x,\mathfrak{a})$ in order to guarantee the existence of the conformal map $z\mapsto W$.  In fact, $\mathfrak{a}$ will also need to depend on $x$, as the following result indicates.
\begin{lemma}
There exist analytic functions $\mathfrak{a}=\mathfrak{a}(x)$ and $s=s(x)$, well-defined in a neighborhood of $x=x_c$ and satisfying $\mathfrak{a}(x_c)=a_c$ and $s(x_c)=0$, such that a conformal map $z\mapsto W$ exists in a neighborhood of $z=a_c$ for $x-x_c$ sufficiently small that guarantees that \eqref{eq:basic-h-equation} holds on the three jump contours for $\mathbf{O}$ for which the jump matrix depends on $z$.  The conformal map takes $z=\mathfrak{a}$ to $W=0$.  
Denoting by $W_c$ the mapping $W$ in the special case that $x=x_c$, 
we have $W_c'(a_c)=3^{1/3}/2^{7/15}>0$.
% and $W''(a_c)=-3^{2/3}/(2^{17/15}7)$.  
 Also, $s'(x_c)=2^{1/5}/3^{1/3}>0$, so that locally the map $x\mapsto s$ is a dilation of a neighborhood of $x_c$.
\label{lemma:conformal}
\end{lemma}
The proof of Lemma~\ref{lemma:conformal} is given in \S\ref{sec:lemma:conformal}.  From now on, we consider $\mathfrak{a}$ not as a fixed parameter, but rather as depending on $x$ near $x_c$ according to Lemma~\ref{lemma:conformal}.  Similarly, $\mathfrak{b}=\mathfrak{b}(x):=\mathfrak{b}(x,\mathfrak{a}(x))$, $\mathfrak{m}=\mathfrak{m}(x):=\mathfrak{m}(x,\mathfrak{a}(x))$, and  $s=s(x)$ with $s(x_c)=0$ and $s'(x_c)>0$.  In this situation, $\mathfrak{h}$ becomes a function of $z$ parametrized by $x$ near $x_c$, and the constant $\nu$ becomes a function of $x$ alone satisfying
\begin{equation}
\nu(x_c)=\frac{1}{3}+2\log(-a_c)=\frac{1}{3}-\log(6^{2/3}).
\label{eq:corner-nu}
\end{equation}
We may therefore consider the partial derivative of $\mathfrak{h}$ with respect to $x$, denoted $\mathfrak{h}_x(z)=\mathfrak{h}_x(z;x)$.  This function is analytic in $z$ for $z\in\mathbb{C}\setminus\Sigma$ and satisfies the asymptotic condition $\mathfrak{h}_x(z)=\tfrac{1}{2}\theta_x(z)-\mathfrak{g}_x(z)=\tfrac{1}{2}z+\mathcal{O}(z^{-1})$ as $z\to\infty$.  On the cut $\Sigma$, the relation $\mathfrak{h}_{x+}(z;x)+\mathfrak{h}_{x-}(z;x)=-\nu'(x)$ holds.  Finally, $\mathfrak{h}_x$ is bounded at the endpoints of $\Sigma$.  It is not difficult to see that these conditions imply that $\mathfrak{h}_x(z;x)$ necessarily has the form
\begin{equation}
\mathfrak{h}_x(z;x)=\frac{1}{2}r(z;\mathfrak{a}(x),\mathfrak{b}(x))-\frac{1}{2}\nu'(x),
\end{equation}
and then by imposing the normalization condition for large $z$ one finds that
\begin{equation}
\nu'(x)=-\frac{1}{2}(\mathfrak{a}(x)+\mathfrak{b}(x)),\quad\text{and in particular}\quad
\nu'(x_c)=-6^{-1/3}.
\label{eq:corner-nu-prime}
\end{equation}

\subsubsection{Tritronqu\'ee parametrix}
\label{section:tritronquee-parametrix}
Consider the rescalings
\begin{equation}
\xi:=\epsilon^{-2/5}W\quad\text{and}\quad t:=\epsilon^{-4/5}s.
\end{equation}
Let $\mathbb{D}_\mathfrak{a}$ be a disk of small fixed radius containing the point $z=\mathfrak{a}$ (which is close to $z=a_c$ for $x$ near $x_c$).
Under the conformal mapping $z\mapsto W$ and the above rescaling of $W$ to obtain $\xi$, 
the disk $\mathbb{D}_\mathfrak{a}$ is mapped to a disk of large radius proportional to $\epsilon^{-2/5}$ centered at the origin in the $\xi$-plane.  Choosing the contours of the original problem near $z=\mathfrak{a}$ so that their images in the $\xi$-plane are
straight rays from the origin, the exact jump conditions for $\mathbf{O}(z)$ may be represented in the $\xi$-plane in terms of the function $\phi(\xi)$ alone, where
\begin{equation}
\phi(\xi)=\phi(\xi;t):=\frac{4}{5}(-\xi)^{5/2}+t(-\xi)^{1/2}.
\label{eq:phidef}
\end{equation}
The exact jump conditions for $\mathbf{O}(z)$ near $z=\mathfrak{a}$ are illustrated in the $\xi$-plane in Figure~\ref{fig:O-jumps-xi}.
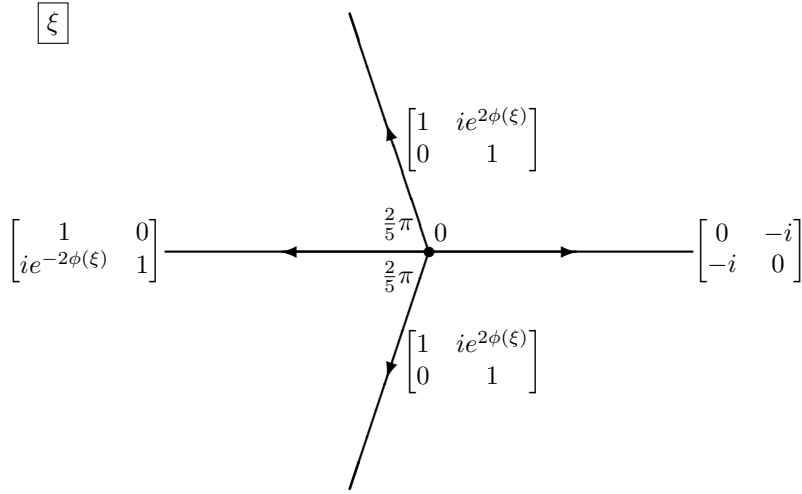
\begin{figure}[h]
\setlength{\unitlength}{2pt}
\begin{center}
\begin{picture}(100,100)(-50,-50)
\put(-74,42){\framebox{$\xi$}}
\put(0,0){\circle*{2}}
\put(-9,4){$\tfrac{2}{5}\pi$}
\put(-9,-6){$\tfrac{2}{5}\pi$}
\put(1,2){$0$}
\thicklines
\put(0,0){\line(1,0){50}}
\put(0,0){\vector(1,0){28}}
%\put(0,-7){$\Sigma$}
\put(50,-1){$\bbm 0 & -i \\ -i & 0 \ebm$}
%
%\qbezier(-35,0)(-10,20)(-35,40)
\put(0,0){\line(-1,3){15}}
\put(-7,21){\vector(-1,3){1}}
\put(-5,20){$\bbm 1 & ie^{2\phi(\xi)} \\ 0 & 1 \ebm$}
\put(0,0){\line(-1,0){50}}
\put(0,0){\vector(-1,0){28}}
\put(-80,-1){$\bbm 1 & 0 \\ ie^{-2\phi(\xi)} & 1 \ebm$}
%
%\qbezier(-35,0)(-10,-20)(-35,-40)
\put(0,0){\line(-1,-3){15}}
\put(-7,-21){\vector(-1,-3){1}}
\put(-5,-22){$\bbm 1 & ie^{2\phi(\xi)} \\ 0 & 1 \ebm$}
\end{picture}
\end{center}
\caption{The jump matrix for $\mathbf{O}(z)$ rendered in the $\xi$-plane.}
\label{fig:O-jumps-xi}
\end{figure}

We wish to find a certain exact solution of these jump conditions designed to match well onto the outer parametrix $\dot{\mathbf{O}}^{(\mathrm{out})}(z)$ at the boundary of the disk.  To do this, we first observe that for $z\in\mathbb{D}_\mathfrak{a}$ the outer parametrix can be expressed in terms of 
$\xi$ as
\begin{equation}
\dot{\mathbf{O}}^{(\mathrm{out})}(z)=\mathbf{F}^{(\mathfrak{a})}(z)\epsilon^{\sigma_3/10}(-\xi)^{\sigma_3/4}\mathbf{M}^{-1}.
\label{eq:corner-o-dot-out-local}
\end{equation}
Here $\mathbf{M}=\mathbf{M}^{-1}$ is defined in \eqref{Odot-gen0}, and 
$\mathbf{F}^{(\mathfrak{a})}(z)$ is a holomorphic matrix-valued function for 
$z\in\mathbb{D}_\mathfrak{a}$ that is independent of $\epsilon$ and satisfies 
$\det(\mathbf{F}^{(\mathfrak{a})}(z))=1$.  (An explicit formula for $\mathbf{F}^{\mathfrak{a}}(z)$ in terms of the conformal mapping $W$ of Lemma~\ref{lemma:conformal} can be obtained directly from \eqref{eq:corner-o-dot-out-local} with $\xi=\epsilon^{-2/5}W$.)  To match onto the outer parametrix therefore means that we will seek a solution of the following parametrix Riemann-Hilbert problem.
\begin{rhp}[Tritronqu\'ee parametrix]
Seek a $2\times 2$ matrix-valued function $\mathbf{T}(\xi)=\mathbf{T}(\xi;t)$ with the following properties:
\begin{itemize}
\item[]\textbf{\textit{Analyticity:}}  $\mathbf{T}(\xi;t)$ is analytic in the four sectors $0<\arg(-\xi)<2\pi/5$,
$-2\pi/5<\arg(-\xi)<0$, $2\pi/5<\arg(-\xi)<\pi$, and $-\pi<\arg(-\xi)<-2\pi/5$, and is H\"older continuous in each sector up to the boundary.
\item[]\textbf{\textit{Jump condition:}}  The boundary values taken along the rays common to the boundary of adjacent sectors (taken with outward orientation) are related as follows:
\begin{equation}
\mathbf{T}_+(\xi;t)=\mathbf{T}_-(\xi;t)\begin{bmatrix}0 & -i\\-i & 0\end{bmatrix},\quad \arg(\xi)=0,
\end{equation}
\begin{equation}
\mathbf{T}_+(\xi;t)=\mathbf{T}_-(\xi;t)\begin{bmatrix}1 & 0\\ie^{-2\phi(\xi;t)} & 1\end{bmatrix},
\quad\arg(-\xi)=0,
\label{eq:negative-exponent}
\end{equation}
\begin{equation}
\mathbf{T}_+(\xi;t)=\mathbf{T}_-(\xi;t)\begin{bmatrix}1 & ie^{2\phi(\xi;t)}\\0 & 1\end{bmatrix},
\quad \arg(-\xi)=\pm\frac{2\pi}{5},
\label{eq:positive-exponent}
\end{equation}
where $\phi$ is defined by \eqref{eq:phidef}.
\item[]\textbf{\textit{Normalization:}}  The matrix $\mathbf{T}(\xi;t)$ satisfies the condition
\begin{equation}
\lim_{\xi\to\infty}\mathbf{T}(\xi;t)\mathbf{M}(-\xi)^{-\sigma_3/4}=\mathbb{I},
\end{equation}
with the limit being uniform with respect to direction.
\end{itemize}
\label{rhp:tritronquee}
\end{rhp}
From the solution of this Riemann-Hilbert problem (if it exists, given $t\in\mathbb{C}$), we obtain
a local parametrix for $\mathbf{O}(z)$ expected to be valid for $z\in\mathbb{D}_\mathfrak{a}$ as follows:
\begin{equation}
\mathbf{O}^{(\mathfrak{a})}(z):=\mathbf{F}^{(\mathfrak{a})}(z)\epsilon^{\sigma_3/10}\mathbf{T}(\epsilon^{-2/5}W(z;x);\epsilon^{-4/5}s(x)),\quad z\in\mathbb{D}_\mathfrak{a}.
\label{eq:O-parametrix-fraka}
\end{equation}

The matrix $\mathbf{T}(\xi;t)$ can be further characterized with the help of Fredholm theory applied to Riemann-Hilbert Problem~\ref{rhp:tritronquee}.  To apply Fredholm theory, we first introduce a related matrix $\widetilde{\mathbf{T}}(\xi;t)$ defined in terms of $\mathbf{T}(\xi;t)$ as follows:
\begin{equation}
\widetilde{\mathbf{T}}(\xi;t):=\begin{cases}
\mathbf{T}(\xi;t),&\quad |\xi|<1,\\
\mathbf{T}(\xi;t)\mathbf{M}(-\xi)^{-\sigma_3/4},&\quad |\xi|>1.
\end{cases}
\label{eq:T-tilde}
\end{equation}
The matrix $\widetilde{\mathbf{T}}(\xi;t)$ satisfies the conditions of a similar Riemann-Hilbert problem 
with the differences being (i) there is an additional jump discontinuity of $\widetilde{\mathbf{T}}(\xi;t)$ across the unit circle (oriented clockwise) with jump matrix $\mathbf{M}(-\xi)^{-\sigma_3/4}$, (ii) $\widetilde{\mathbf{T}}(\xi;t)$ has no jump across the positive real axis for $\xi>1$, (iii) on the remaining three rays of the jump contour for $\mathbf{T}$ the jump matrix for $\widetilde{\mathbf{T}}$ for $|\xi|>1$ is the jump matrix for $\mathbf{T}$ conjugated by $\mathbf{M}(-\xi)^{-\sigma_3/4}$, and (iv) the normalization condition becomes $\widetilde{\mathbf{T}}(\xi;t)\to\mathbb{I}$ as $\xi\to\infty$.  It is by now a standard construction (see Muskehelishvili \cite{Muskhelishvili} and Deift \cite{Deift} for general theory, and see \cite[Appendix A]{KMM} for specific information about the H\"older spaces most useful in the present application) to associate an identity-normalized Riemann-Hilbert problem such as that satisfied by $\widetilde{\mathbf{T}}$ with an inhomogeneous linear system of singular integral equations (formulated on a suitable function space of H\"older-continuous matrix-valued functions on the jump contour) for which the linear operator acting on the unknown is Fredholm with zero index.  The fact that the parameter $t$ appears analytically in this operator immediately implies that the solution may fail to exist only at isolated points in the $t$-plane, and these singularities are poles of finite order.  The representation of $\widetilde{\mathbf{T}}(\xi;t)$ made available via this approach combined with the fact that the jump matrices for $\widetilde{\mathbf{T}}$ decay to $\mathbb{I}$ as $\xi\to\infty$ along each ray faster than any negative power of $|\xi|$ shows that in fact $\mathbf{T}(\xi;t)\mathbf{M}(-\xi)^{-\sigma_3/4}=\widetilde{\mathbf{T}}(\xi;t)$ has an asymptotic expansion as $\xi\to\infty$ in descending integer powers of $\xi$; there exists a sequence of matrix functions $\{\mathbf{T}_p(t)\}_{p=1}^\infty$ such that, for any integer $P\ge 0$,
\begin{equation}
\mathbf{T}(\xi;t)\mathbf{M}(-\xi)^{-\sigma_3/4}=\mathbb{I}+\sum_{p=1}^P\mathbf{T}_p(t)\xi^{-p} + \mathcal{O}(\xi^{-(P+1)}),\quad \xi\to\infty.
\label{eq:T-asymptotic-series}
\end{equation}
The coefficient matrices $\mathbf{T}_p(t)$ are analytic in $t$ except at the singularities of the solution where they have at worst poles of finite order.  The error term is uniform for $t$ in any compact set that does not contain any of these singularities.  It also follows from the representation of the solution that the asymptotic series \eqref{eq:T-asymptotic-series} is differentiable term-by-term with respect to both $\xi$ and (away from singularities) $t$. 

Two additional observations regarding Riemann-Hilbert Problem~\ref{rhp:tritronquee} are the following.  Firstly, since the jump matrices all have unit determinant, then so does $\mathbf{T}(\xi;t)$ when it exists, and this further implies that $\mathbf{T}(\xi;t)^{-1}$ has the same singularities as does $\mathbf{T}(\xi;t)$ itself.  The condition $\det(\mathbf{T}(\xi;t))=1$ applied to the series \eqref{eq:T-asymptotic-series} also implies that
\begin{equation}
\mathrm{tr}(\mathbf{T}_1(t))=0.
\label{eq:trace-T1-zero}
\end{equation}
Secondly, if $\mathbf{T}(\xi;t)$ is a solution of Riemann-Hilbert Problem~\ref{rhp:tritronquee}, then so is $\mathbf{T}(\xi^*;t^*)^*$, so the index-zero condition implies that whenever $t$ is such that $\mathbf{T}$ exists, then so is $t^*$,
and $\mathbf{T}(\xi;t)=\mathbf{T}(\xi^*;t^*)^*$.  This identity further implies that all of the coefficients in the expansion \eqref{eq:T-asymptotic-series} satisfy $\mathbf{T}_p(t)=\mathbf{T}_p(t^*)^*$, i.e., the matrix entries are all Schwarz-symmetric meromorphic functions of $t$.

The elements of the coefficient matrices $\mathbf{T}_p(t)$ satisfy a number of ordinary differential equations.  Indeed, observe that the matrix $\mathbf{L}(\xi;t):=\mathbf{T}(\xi;t)e^{\phi(\xi;t)\sigma_3}$ satisfies jump conditions that are independent of both $\xi$ and $t$ (the transformation has the effect of replacing the exponential factors $e^{\pm 2\phi(\xi;t)}$ in \eqref{eq:negative-exponent}--\eqref{eq:positive-exponent} with $1$).  This implies that the matrices
\begin{equation}
\mathbf{A}(\xi;t):=\frac{\partial\mathbf{L}}{\partial\xi}(\xi;t)\mathbf{L}(\xi;t)^{-1}
\quad\text{and}\quad
\mathbf{U}(\xi;t):=\frac{\partial\mathbf{L}}{\partial t}(\xi;t)\mathbf{L}(\xi;t)^{-1}
\label{eq:A-U-def}
\end{equation}
are both entire functions of $\xi$.  We may compute their asymptotic expansions as $\xi\to\infty$ with the help of \eqref{eq:T-asymptotic-series}.  These are as follows:
\begin{equation}
\begin{split}
\mathbf{U}(\xi;t)&=-\xi\sigma_++\sigma_--[\mathbf{T}_1(t),\sigma_+] + \mathcal{O}(\xi^{-1}),\\
\mathbf{A}(\xi;t)&=-2\xi^2\sigma_+ + 2\xi\left(\sigma_--[\mathbf{T}_1(t),\sigma_+]\right)+
2[\mathbf{T}_1(t),\sigma_+]\mathbf{T}_1(t)+2[\mathbf{T}_1(t),\sigma_-]\\
&\quad\quad\quad{}-2[\mathbf{T}_2(t),\sigma_+]-\frac{1}{2}t\sigma_+ + \xi^{-1}\mathbf{A}_1(t) + \xi^{-2}\mathbf{A}_2(t)+\mathcal{O}(\xi^{-3}),
\end{split}
\end{equation}
where $[\mathbf{A},\mathbf{B}]:=\mathbf{A}\mathbf{B}-\mathbf{B}\mathbf{A}$,  
and with the help of \eqref{eq:trace-T1-zero} we have
\begin{equation}
A_{1,21}(t)=\frac{1}{2}t + 2T_{1,11}(t)^2 - 2\det(\mathbf{T}_1(t)) + 2(T_{2,22}(t)-T_{2,11}(t))+4T_{1,21}(t)T_{2,21}(t).
\label{eq:A1-21}
\end{equation}
But since $\mathbf{U}(\xi;t)$ and $\mathbf{A}(\xi;t)$ are entire in $\xi$, they are equal to the polynomial terms in the expansions
\begin{equation}
\begin{split}
\mathbf{U}(\xi;t)&=-\xi\sigma_++\sigma_--[\mathbf{T}_1(t),\sigma_+],\\
\mathbf{A}(\xi;t)&=-2\xi^2\sigma_+ + 2\xi\left(\sigma_--[\mathbf{T}_1(t),\sigma_+]\right)+
2[\mathbf{T}_1(t),\sigma_+]\mathbf{T}_1(t)+2[\mathbf{T}_1(t),\sigma_-]
-2[\mathbf{T}_2(t),\sigma_+]-\frac{1}{2}t\sigma_+,
\end{split}
\label{eq:U-and-A}
\end{equation}
and we further establish the identity $A_{1,21}(t)=0$.   According to the definitions \eqref{eq:A-U-def}, 
$\mathbf{L}(\xi;t)$ is a simultaneous fundamental solution matrix of the Lax pair of differential equations $\mathbf{L}_\xi=\mathbf{A}\mathbf{L}$ and $\mathbf{L}_t=\mathbf{U}\mathbf{L}$.  This overdetermined system is therefore compatible, meaning that the coefficient matrices $\mathbf{A}$ and $\mathbf{U}$ satisfy the zero-curvature condition $\mathbf{U}_\xi-\mathbf{A}_t +[\mathbf{U},\mathbf{A}]=\mathbf{0}$.  With the help of \eqref{eq:trace-T1-zero} and $A_{1,21}(t)=0$ with $A_{1,21}(t)$ given by \eqref{eq:A1-21}, the matrix elements of $\mathbf{U}(\xi;t)$ and $\mathbf{A}(\xi;t)$ can be written in terms of just three unknown functions of $t$:
\begin{equation}
\begin{split}
H=H(t)&:=T_{1,21}(t),\\
Y=Y(t)&:=T_{1,21}(t)^2 + T_{1,22}(t)-T_{1,11}(t),\\
Z=Z(t)&:=2T_{2,21}(t)+2T_{1,12}(t)-2T_{1,21}(t)^3-4T_{1,21}(t)T_{1,22}(t)+2T_{1,21}(t)T_{1,11}(t),
\end{split}
\label{eq:HYZdef}
\end{equation}
namely
\begin{equation}
\mathbf{U}(\xi;t)=\begin{bmatrix}H& Y-H^2-\xi\\1 & -H\end{bmatrix}
\end{equation}
and
\begin{equation}
\mathbf{A}(\xi;t)=\begin{bmatrix}
2\xi H+Z+2HY & -2\xi^2+2\xi(Y-H^2)-t-2Y^2-2H^2Y-2HZ\\
2\xi+2Y & -2\xi H-Z-2HY
\end{bmatrix}.
\label{eq:corner-A-rewrite}
\end{equation}
Upon separating the coefficients of different powers of $\xi$, the zero-curvature condition then yields the following three differential equations (and no further relations):
\begin{equation}
\begin{split}
H'(t)&=-Y(t),\\
Y'(t)&=Z(t),\\
Z'(t)&=6Y(t)^2+t.
\end{split}
\label{eq:PI-system}
\end{equation}
Elimination of $Z(t)$ shows that $Y(t)$ solves the Painlev\'e-I equation
\eq
Y''(t)=6Y(t)^2+t.
\label{eq:Painleve-I}
\endeq
Now, setting to zero the other elements of the matrix coefficient $\mathbf{A}_1(t)$ allows $T_{3,21}(t)$ and the difference $T_{3,22}(t)-T_{3,11}(t)$ to be explicitly expressed in terms of elements of the matrices $\mathbf{T}_1(t)$ and $\mathbf{T}_2(t)$ (only two further relations appear from the equation $\mathbf{A}_1(t)=\mathbf{0}$ because $\mathrm{tr}(\mathbf{A}_1(t))=0$).  Using these identities along with
$A_{1,21}(t)=0$, \eqref{eq:trace-T1-zero}, and the definitions \eqref{eq:HYZdef}, the equation $A_{2,21}(t)=0$ yields the additional identity
\begin{equation}
H(t)=\frac{1}{2}Z(t)^2-2Y(t)^3-tY(t),
\label{eq:Hamiltonian}
\end{equation}
which is easily checked to be consistent with \eqref{eq:PI-system}.
Therefore, $H$ is the Hamiltonian function associated with the solution $Y$ of the Painlev\'e-I equation.  

The large-$t$ asymptotic behavior of the general solution of the Painlev\'e-I equation was studied by Kapaev \cite{Kapaev:2004} by means of the Deift-Zhou steepest descent method applied to a Riemann-Hilbert problem that is equivalent to a generalization of Riemann-Hilbert Problem~\ref{rhp:tritronquee} to allow two independent Stokes constants (this also requires including two additional jump rays with angles $\arg(-\xi)=\pm 4\pi/5$).  Kapaev proves that in the special case of the Stokes constants in which his problem corresponds with Riemann-Hilbert Problem~\ref{rhp:tritronquee}, the function $Y$ ($Y=y_3=y_{-2}$ in Kapaev's notation) has the asymptotic behavior
\begin{equation}
Y(t)=-\left(-\frac{t}{6}\right)^{1/2}+\mathcal{O}(t^{-2}),\quad t\to\infty,\quad |\arg(-t)|\le \frac{4}{5}\pi-\delta
\label{eq:tritronquee-asymptotics}
\end{equation}
for any $\delta>0$ however small.  The validity of this asymptotic formula in such a large sector of the $t$-plane is sufficient to uniquely identify the solution $Y(t)$ of the Painlev\'e-I equation \cite[Remark 2.3]{Kapaev:2004}.  It is called the (real) tritronqu\'ee solution.  More generally, there is a one-parameter family of Schwarz-symmetric ``tronqu\'ee'' solutions of $Y''(t)=2=6Y(t)^2+t$ having the same asymptotic description \eqref{eq:tritronquee-asymptotics} but with $t$ restricted to the smaller sector $|\arg(-t)|\le \pi/5-\delta$.  In the remaining sectors of the complex $t$-plane there are (double) poles accumulating at $t=\infty$; the general solution of the Painlev\'e-I equation has poles near $t=\infty$ in all directions.  In the limit of large $|t|$ in any sector in which the solution is not asymptotically pole-free, the solution is asymptotically described by a Weierstra\ss\ elliptic function with modulus depending on $\arg(t)$.  

Dubrovin, Grava, and Klein \cite{Dubrovin:2009} conjectured that $Y(t)$ is analytic for $|\arg(-t)|<4\pi/5$, i.e., that the pole-free nature of $Y(t)$ that holds for large $|t|$ due to \eqref{eq:tritronquee-asymptotics} actually extends to all $|t|$ in the indicated sector.  This conjecture has recently been proven by Costin, Huang, and Tanveer \cite{CostinHT}.  The poles of $Y$ are obviously values of $t$ for which the solution of Riemann-Hilbert Problem~\ref{rhp:tritronquee} is itself singular (fails to exist).  Conversely, if $t\in\mathbb{C}$ is a value at which $Y(\cdot)$ is analytic, then it follows from the differential equations \eqref{eq:PI-system} that the same is true for $H(\cdot)$ and $Z(\cdot)$, and hence the matrix $\mathbf{A}(\xi;t)$ given by \eqref{eq:corner-A-rewrite} also exists as a quadratic polynomial in $\xi$.  Therefore, at such a value of $t$ there exist canonical solutions $\mathbf{L}(\xi;t)$ of the linear system $\mathbf{L}_\xi=\mathbf{A}\mathbf{L}$, and it follows that the matrix $\mathbf{T}(\xi;t)$ satisfying the conditions of Riemann-Hilbert Problem~\ref{rhp:tritronquee} also exists.  In other words, the singularities of $\mathbf{T}(\xi;t)$ are precisely the poles of the tritronqu\'ee solution $Y(t)$.

It is well-known that every pole of every solution $Y(t)$ of the Painlev\'e-I equation $Y''(t)=6Y(t)^2+t$
is a double pole.  Indeed, if $t_0$ is a pole of $Y(t)$, then by matching the (necessarily) dominant
terms $Y''(t)$ and $-6Y(t)^2$ one easily obtains this result.  By continuing the argument to higher order one sees that, at any pole $t_0$, $Y$ has a Laurent expansion of the form
\begin{equation}
Y(t)=\frac{1}{(t-t_0)^2} -\frac{t_0}{10}(t-t_0)^2 -\frac{1}{6}(t-t_0)^3 + \mathcal{O}((t-t_0)^4),\quad t\to t_0.
\end{equation}
Substituting this expansion into \eqref{eq:Hamiltonian} shows that the associated Hamiltonian has the expansion
\begin{equation}
H(t)=\frac{1}{t-t_0}+\mathcal{O}(1),
\end{equation}
i.e., the Hamiltonian necessarily has simple poles, all of residue $1$.

Fix a compact set $\mathcal{K}$ in the $t$-plane that contains no poles of the tritronqu\'ee solution $Y(t)$.  Uniformly for $t\in \mathcal{K}$ (i.e., for $s(x)\in\epsilon^{4/5}\mathcal{K}$), we then have 
\begin{equation}
\mathbf{T}(\xi;t)\mathbf{M}(-\xi)^{-\sigma_3/4}=\mathbb{I}+\mathcal{O}(\xi^{-1}),\quad \xi\to\infty.
\end{equation}
Then, since $z\in\partial\mathbb{D}_\mathfrak{a}$ implies that $\xi$ is proportional to $\epsilon^{-2/5}$, the definition \eqref{eq:O-parametrix-fraka} implies that (since $\mathbf{F}^{(\mathfrak{a})}(z)$ is independent of $\epsilon$ and has unit determinant)
\begin{equation}
\dot{\mathbf{O}}^{(\mathfrak{a})}(z)\dot{\mathbf{O}}^{(\mathrm{out})}(z)^{-1}=
\mathbf{F}^{(\mathfrak{a})}(z)\epsilon^{\sigma_3/10}\left(\mathbb{I}+\mathcal{O}(\epsilon^{2/5})\right)
\epsilon^{-\sigma_3/10}\mathbf{F}^{(\mathfrak{a})}(z)^{-1}=
\mathbb{I}+\mathcal{O}(\epsilon^{1/5}),\quad \epsilon\to 0, \quad z\in\partial\mathbb{D}_\mathfrak{a}.
\label{eq:match-on-partial-Da}
\end{equation}
This estimate also holds uniformly for $t\in \mathcal{K}$.

%Let $y(v)$ be a solution of the Painlev\'e I equation \eqref{Painleve-I}.  
%If $y(v)$ has a pole at $v=v_0$ then by matching the (necessarily) dominant 
%$y''(v)$ and $-6y^2$ terms in \eqref{Painleve-I}, we see that $v_0$ must be a 
%double pole.  We want to understand the behavior of the Hamiltonian
%\eq
%\mathscr{H}(y,v):=\frac{1}{2}(y'(v))^2+vy(v)-2y(v)^3
%\endeq
%near the pole $v_0$.  Assume a Laurent expansion for $y$:
%\eq
%y(v) = \sum_{k=-2}^\infty c_k(v-v_0)^k.
%\endeq
%By plugging this series into the differential equation \eqref{Painleve-I}, 
%we can identify the coefficients, yielding
%\eq
%y(v) = \frac{1}{(v-v_0)^2}+\frac{v_0}{10}(v-v_0)^2+\frac{1}{6}(v-v_0)^3+\mathcal{O}\left((v-v_0)^4\right).
%\endeq
%Plugging this expansion into $\mathscr{H}(y,v)$ yields 
%\eq
%\mathscr{H}(y,v) = \frac{-1}{v-v_0} + \mathcal{O}(1),
%\endeq
%so all poles of the Hamiltonian are simple with residue $-1$.
%

\subsection{The global parametrix and computation of error terms}
The global parametrix for $\mathbf{O}(z)$ is defined as follows:
\begin{equation}
\dot{\mathbf{O}}(z):=\begin{cases}
\dot{\mathbf{O}}^{(\mathfrak{b})}(z),&\quad z\in \mathbb{D}_\mathfrak{b},\\
\dot{\mathbf{O}}^{(\mathfrak{a})}(z),&\quad z\in\mathbb{D}_\mathfrak{a},\\
\dot{\mathbf{O}}^{(\mathrm{out})}(z),&\quad z\in\mathbb{C}\setminus\overline{\mathbb{D}_\mathfrak{a}\cup\mathbb{D}_\mathfrak{b}}.
\end{cases}
\label{eq:global-parametrix}
\end{equation}
The error in approximating $\mathbf{O}(z)$ by its global parametrix $\dot{\mathbf{O}}(z)$ is quantified by introducing the error $\mathbf{E}(z):=\mathbf{O}(z)\dot{\mathbf{O}}(z)^{-1}$.  This matrix
is analytic for $z\in\mathbb{C}\setminus\Sigma^{(\mathbf{E})}$, where the jump contour $\Sigma^{(\mathbf{E})}$ consists of (i) all arcs of the jump contour for $\mathbf{O}(z)$ outside of the disks $\mathbb{D}_\mathfrak{a}$ and $\mathbb{D}_\mathfrak{b}$ with the exception of the segment $\Sigma$ (because the outer parametrix $\dot{\mathbf{O}}^{(\mathrm{out})}(z)$ and $\mathbf{O}(z)$ satisfy the same jump condition on $\Sigma$) and (ii) the two circles $\partial\mathbb{D}_\mathfrak{a}$ and $\partial\mathbb{D}_\mathfrak{b}$, both of which we take to be oriented in the clockwise (negative) direction.  Note that $\Sigma^{(\mathbf{E})}\cap(\mathbb{D}_\mathfrak{a}\cup\mathbb{D}_\mathfrak{b})=\emptyset$ because the inner parametrices exactly satisfy the jump conditions of
$\mathbf{O}(z)$ within these disks.  We assume that the arcs of $\Sigma^{(\mathbf{E})}$ that coincide with arcs of the jump contour for $\mathbf{O}(z)$ also inherit their orientation, and note further that the contour $\Sigma^{(\mathbf{E})}$ will be taken to be independent of $x$ near $x_c$.  Note also that
$\mathbf{E}(z)\to\mathbb{I}$ as $z\to\infty$ as this holds for both $\mathbf{O}(z)$ and $\dot{\mathbf{O}}^{(\mathrm{out})}(z)$.

\begin{lemma}
Let $\mathcal{K}$ be a compact subset of the complex $t$-plane that does not contain any poles of the tritronqu\'ee solution $Y(t)$ and 
suppose that $t=\epsilon^{4/5}s(x)\in \mathcal{K}$. 
The error $\mathbf{E}(z)$ satisfies the conditions of a Riemann-Hilbert problem of small-norm type in the limit $\epsilon\to 0$
with jump matrix $\mathbf{V}^{(\mathbf{E})}(z)$ satisfying the conditions $\mathbf{V}^{(\mathbf{E})}-\mathbb{I}\in L^2(\Sigma^{(\mathbf{E})})\cap L^\infty(\Sigma^{(\mathbf{E})})$ and the estimates $\|\mathbf{V}^{(\mathbf{E})}-\mathbb{I}\|_{L^\infty(\Sigma^{(\mathbf{E})})}=\mathcal{O}(\epsilon^{1/5})$ and
$\|z^2(\mathbf{V}^{(\mathbf{E})}-\mathbb{I})\|_{L^1(\Sigma^{(\mathbf{E})})}=\mathcal{O}(\epsilon^{1/5})$, both holding uniformly for $t\in \mathcal{K}$.
\label{lemma-corner-small-norm}
\end{lemma}
\begin{proof}  It suffices to analyze the jump matrix $\mathbf{V}^{(\mathbf{E})}(z)$ defined on
$\Sigma^{(\mathbf{E})}$.
On the arcs of $\Sigma^{(\mathbf{E})}$ outside of the two disks $\mathbb{D}_\mathfrak{a}$ and $\mathbb{D}_\mathfrak{b}$, we have
\begin{equation}
\mathbf{V}^{(\mathbf{E})}(z)=\mathbf{E}_-(z)^{-1}\mathbf{E}_+(z)=\dot{\mathbf{O}}^{(\mathrm{out})}(z)
\mathbf{V}^{(\mathbf{O})}(z)\dot{\mathbf{O}}^{(\mathrm{out})}(z)^{-1},
\end{equation}
where $\mathbf{V}^{(\mathbf{O})}(z)$ denotes the jump matrix for $\mathbf{O}(z)$.  But since $s$ is small, and hence $2\mathfrak{h}+\nu$ is close to $2h+\lambda$ outside of the two disks, $\mathbf{V}^{(\mathbf{O})}(z)-\mathbb{I}$ is exponentially small in the limit $\epsilon\to 0$ and rapidly decaying as $z\to\infty$.  Since $\dot{\mathbf{O}}^{(\mathrm{out})}(z)$ and its inverse are both independent of $\epsilon$ and bounded outside of the disks, it follows that $\mathbf{V}^{(\mathbf{E})}-\mathbb{I}$ is exponentially small in $\epsilon$ and rapidly decaying as $z\to\infty$ on all arcs of $\Sigma^{(\mathbf{E})}$ outside of the two disks.  

Next consider $z\in\partial\mathbb{D}_\mathfrak{b}$.  Given the clockwise orientation, the jump across this circle is $\mathbf{V}^{(\mathbf{E})}(z)=\dot{\mathbf{O}}^{(\mathfrak{b})}(z)\dot{\mathbf{O}}^{(\mathrm{out})}(z)^{-1}$ which is uniformly $\mathbb{I}+\mathcal{O}(\epsilon)$ according to \eqref{eq:corner-airy-match}.  
Finally, consider $z\in\partial\mathbb{D}_\mathfrak{a}$.  Here a similar argument using instead 
\eqref{eq:match-on-partial-Da} shows that $\mathbf{V}^{(\mathbf{E})}(z)-\mathbb{I}=\mathcal{O}(\epsilon^{1/5})$, this being the dominant contribution to $\|\mathbf{V}^{(\mathbf{E})}-\mathbb{I}\|_{L^\infty(\Sigma^{(\mathbf{E})})}$.
\end{proof}

This result is related to a rather well-developed theory of small-norm Riemann-Hilbert problems, some useful elements of which are described in \cite[Appendix B]{Buckingham-Miller-rational-noncrit}.  
%\textcolor{red}{The following sentence is probably not necessary here (but perhaps reserve for the edge section) because the contour $\Sigma^{(\mathbf{E})}$ can be taken to be independent of $x$.}  Here we are avoiding certain technicalities related to extra steps needed in principle to prove the uniformity of operator norm bounds on Cauchy projections, but rather than describe these steps we refer to the companion paper cited above in which they are carried out in detail.  
Lemma~\ref{lemma-corner-small-norm} has several consequences.  One of them is the existence of an asymptotic (for large $z$) representation of the form
\begin{equation}
\mathbf{E}(z)=\mathbb{I}+z^{-1}\mathbf{E}_1 +z^{-2}\mathbf{E}_2 + o(z^{-1}),\quad z\to\infty.
\end{equation}
Here the coefficients $\mathbf{E}_1$ and $\mathbf{E}_2$ will depend on both $x$ and $\epsilon$.  Further consequences include asymptotic (for small $\epsilon$) formulae for the moments $\mathbf{E}_1$ and $\mathbf{E}_2$:  for $p=1,2$,
\begin{equation}
\mathbf{E}_p=-\frac{1}{2\pi i}\int_{\Sigma^{(\mathbf{E})}}(\mathbf{V}^{(\mathbf{E})}(z)-\mathbb{I})z^{p-1}\,dz -\frac{1}{2\pi i}\int_{\Sigma^{(\mathbf{E})}}\mathcal{C}_-^{\Sigma^{(\mathbf{E})}}[\mathbf{V}^{(\mathbf{E})}(\cdot)-\mathbb{I}](z)(\mathbf{V}^{(\mathbf{E})}(z)-\mathbb{I})z^{p-1}\,dz + \mathcal{O}(\epsilon^{3/5}).
\label{eq:corner-moment-formulas}
\end{equation}
Here we use the Cauchy projection operator that, given an admissible (as 
defined in \cite[Appendix B]{Buckingham-Miller-rational-noncrit}) oriented 
contour $\gamma$, is defined as 
\eq
\mathcal{C}_-^{\gamma}[f](\zeta):=\mathop{\lim_{z\to\zeta}}_{z\text{ on the right of }\zeta}\frac{1}{2\pi i}\int_\gamma\frac{f(w)dw}{w-z}, \quad z\in\mathbb{C}\backslash \gamma.
\endeq

Now, as part of the proof of Lemma~\ref{lemma-corner-small-norm} it was shown that $\mathbf{V}^{(\mathbf{E})}(z)-\mathbb{I}$ is dominated by its behavior for $z\in\partial\mathbb{D}_\mathfrak{a}$, with all other contributions being $\mathcal{O}(\epsilon)$.  Therefore, without any loss of the order of accuracy, the integration contour $\Sigma^{(\mathbf{E})}$ in \eqref{eq:corner-moment-formulas} 
could be replaced with $\partial\mathbb{D}_\mathfrak{a}$ (oriented negatively), and similarly $\mathcal{C}_-^{\Sigma^{(\mathbf{E})}}$ can be replaced by 
$\mathcal{C}_-^{\partial\mathbb{D}_\mathfrak{a}}$.  
Furthermore, for $z\in\partial\mathbb{D}_\mathfrak{a}$, we have the formula $\mathbf{V}^{(\mathbf{E})}(z)=
\dot{\mathbf{O}}^{(\mathfrak{a})}(z)\dot{\mathbf{O}}^{(\mathrm{out})}(z)^{-1}$.  In \eqref{eq:match-on-partial-Da} the term $\mathbb{I}+\mathcal{O}(\epsilon^{2/5})$ is shorthand for the
series \eqref{eq:T-asymptotic-series} evaluated for $\xi=\epsilon^{-2/5}W(z)$, where $W$ is the conformal map described by Lemma~\ref{lemma:conformal}.  Therefore, a more precise version of \eqref{eq:match-on-partial-Da} is
\begin{multline}
\mathbf{V}^{(\mathbf{E})}(z)-\mathbb{I}=\frac{\epsilon^{1/5}H(t)}{W(z)}\mathbf{F}^{(\mathfrak{a})}(z)\sigma_-\mathbf{F}^{(\mathfrak{a})}(z)^{-1} \\{}+ 
\frac{\epsilon^{2/5}(H(t)^2-Y(t))}{2W(z)}\mathbf{F}^{(\mathfrak{a})}(z)\sigma_3
\mathbf{F}^{(\mathfrak{a})}(z)^{-1}+\mathcal{O}(\epsilon^{3/5}),\quad z\in\partial\mathbb{D}_\mathfrak{a}.
\label{eq:match-on-partial-Da-better}
\end{multline}
Without any change of the order of the error term above, we may consider $\mathbf{F}^{(\mathfrak{a})}(z)$ and $W(z)$ to be evaluated for $x=x_c$ (and hence also $\mathfrak{a}=a_c$ and $\mathfrak{b}=b_c$) as the difference amounts to a contribution of order $x-x_c=\mathcal{O}(\epsilon^{4/5})$.  
In this degenerate situation, we have 
\begin{equation}
\mathbf{F}^{(\mathfrak{a})}(z)=\mathbf{M}\cdot (b_c-z)^{-\sigma_3/4}\left(\frac{z-a_c}{W_c(z)}\right)^{\sigma_3/4},
\end{equation}
which implies that
\begin{equation}
\mathbf{F}^{(\mathfrak{a})}(z)\sigma_-\mathbf{F}^{(\mathfrak{a})}(z)^{-1}=k(z)
\mathbf{M}\sigma_-\mathbf{M}^{-1}\quad\text{and}\quad
\mathbf{F}^{(\mathfrak{a})}(z)\sigma_3\mathbf{F}^{(\mathfrak{a})}(z)^{-1}=\sigma_1,
\end{equation}
where $k(z)$ is the scalar function analytic at $z=a_c$ given by
\begin{equation}
k(z):=\frac{1}{2}(b_c-z)^{1/2}
\left(\frac{W_c(z)}{z-a_c}\right)^{1/2},
\end{equation}
and where $\mathbf{M}\sigma_-\mathbf{M}^{-1}=i\sigma_2+\sigma_3$.
Clearly, computing the second integral term in \eqref{eq:corner-moment-formulas} up to terms of order $\epsilon^{3/5}$ amounts to substituting for $\mathbf{V}^{(\mathbf{E})}(\cdot)-\mathbb{I}$ from the
leading term of \eqref{eq:match-on-partial-Da-better}; since then the constant matrix $\mathbf{M}\sigma_-\mathbf{M}^{-1}$ will appear squared, the fact that $\sigma_-^2=\mathbf{0}$ means that only the first integral term in \eqref{eq:corner-moment-formulas} is actually needed.  Therefore, inserting \eqref{eq:match-on-partial-Da-better} into \eqref{eq:corner-moment-formulas} under these considerations leads to the formula
\begin{equation}
\mathbf{E}_p=-\frac{\epsilon^{1/5}H(t)}{2\pi i}\bbm 1 & 1 \\ -1 & -1\ebm
\oint_{\partial\mathbb{D}_\mathfrak{a}}\frac{k(z)z^{p-1}\,dz}{W_c(z)} -
\frac{\epsilon^{2/5}(H(t)^2-Y(t))}{4\pi i}\sigma_1\oint_{\partial\mathbb{D}_\mathfrak{a}}\frac{z^{p-1}\,dz}{W_c(z)} + \mathcal{O}(\epsilon^{3/5}).
\end{equation}
Taking into account the negative orientation of the circle $\partial\mathbb{D}_\mathfrak{a}$ and the fact that $W_c(z)$ has a unique simple zero at $z=a_c\in\mathbb{D}_\mathfrak{a}$, we then obtain
\begin{equation}
\mathbf{E}_p=\epsilon^{1/5}\frac{k(a_c)a_c^{p-1}}{W_c'(a_c)}H(t)\bbm 1 & 1 \\ -1 & -1\ebm
+\epsilon^{2/5}\frac{a_c^{p-1}}{2W_c'(a_c)}(H(t)^2-Y(t))\sigma_1 + \mathcal{O}(\epsilon^{3/5}).
\end{equation}
Finally, using $k(a_c)=\tfrac{1}{2}(b_c-a_c)^{1/2}W_c'(a_c)^{1/2}$ and recalling the value of $W_c'(a_c)$ from Lemma~\ref{lemma:conformal} and the values of $a_c$ and $b_c$ from
\eqref{eq:ac-bc-define}, we obtain
\begin{equation}
\begin{split}
\mathbf{E}_1&=\epsilon^{1/5}\frac{2^{1/15}}{3^{1/3}}H(t)\bbm 1 & 1\\-1 & -1\ebm +\epsilon^{2/5}\frac{1}{2^{8/15}3^{1/3}}(H(t)^2-Y(t))\sigma_1+\mathcal{O}(\epsilon^{3/5}),\\
\mathbf{E}_2&=-\epsilon^{1/5}\frac{1}{2^{4/15}3^{2/3}}H(t)\bbm 1 & 1\\-1 & -1\ebm -\epsilon^{2/5}\frac{1}{2^{13/15}3^{2/3}}(H(t)^2-Y(t))\sigma_1+\mathcal{O}(\epsilon^{3/5}).
\end{split}
\end{equation}

\subsection{Asymptotic formulae for the rational Painlev\'e-II functions}
We begin with the exact formula $\mathbf{O}(z)=\mathbf{E}(z)\dot{\mathbf{O}}(z)$ that expresses the matrix $\mathbf{O}(z)$, whose moments at $z=\infty$ encode the rational Painlev\'e-II functions, in terms of the explicit global parametrix $\dot{\mathbf{O}}(z)$ and the error matrix $\mathbf{E}(z)$.  This implies the following exact formulae for the rational Painlev\'e-II functions 
%(adapted from \eqref{eq:edge-Um-exact}--\eqref{eq:edge-pq-formula1} with the substitutions $\lambda\to\nu$ and $\mathbf{P}_j\to\mathbf{0}$):
\begin{equation}
\epsilon^{(2+\epsilon)/(3\epsilon)}e^{-\nu/\epsilon}\pu_m=\dot{O}_{1,12} + E_{1,12},
\label{eq:corner-pu-exact}
\end{equation}
\begin{equation}
\epsilon^{-(2-\epsilon)/(3\epsilon)}e^{\nu/\epsilon}\pv_m=\dot{O}_{1,21}+E_{1,21},
\label{eq:corner-pv-exact}
\end{equation}
\begin{equation}
\epsilon^{1/3}\pp_m=\dot{O}_{1,22}+E_{1,22} -\frac{\dot{O}_{2,12}+E_{2,12}+E_{1,11}\dot{O}_{1,12}+
E_{1,12}\dot{O}_{1,22}}{\dot{O}_{1,12}+E_{1,12}},
\label{eq:corner-pp-exact}
\end{equation}
and
\begin{equation}
\epsilon^{1/3}\pq_m=-\dot{O}_{1,11}-E_{1,11}+\frac{\dot{O}_{2,21}+E_{2,21}+E_{2,12}\dot{O}_{1,11}+E_{1,22}\dot{O}_{1,21}}{\dot{O}_{1,21}+E_{1,21}}.
\label{eq:corner-pa-exact}
\end{equation}
The derivation of \eqref{eq:corner-pu-exact}--\eqref{eq:corner-pa-exact} is 
given in \cite[Equations (3-49)--(3-52)]{Buckingham-Miller-rational-noncrit} 
with the substitutions $\lambda\to\nu$, 
$\dot{\bf O}^\text{(out)}\to\dot{\bf O}$, and $g\to\mathfrak{g}$.  
Here, the matrices $\dot{\mathbf{O}}_1$ and $\dot{\mathbf{O}}_2$ are the coefficients in the large-$z$ expansion
\begin{equation}
\dot{\mathbf{O}}(z)=\mathbb{I}+z^{-1}\dot{\mathbf{O}}_1 + z^{-2}\dot{\mathbf{O}}_2 + \mathcal{O}(z^{-3}),\quad z\to\infty.
\end{equation}
Since $\dot{\mathbf{O}}(z)=\dot{\mathbf{O}}^{(\mathrm{out})}(z)$ for $|z|$ 
sufficiently large in magnitude, the coefficients $\dot{\bf O}_1$ and 
$\dot{\bf O}_2$ have essentially already been computed in \eqref{eq:edge-dot-O1}--\eqref{eq:edge-dot-O2}; the only task remaining is to set $K=0$ in those formulae to obtain
\begin{equation}
\dot{\mathbf{O}}_1=\frac{1}{4}(\mathfrak{b}-\mathfrak{a})\sigma_1=6^{-1/3}\sigma_1+\mathcal{O}(\epsilon^{4/5}) \quad \text{and} \quad
\dot{O}_{2,12}=\dot{O}_{2,21}=\frac{1}{8}(\mathfrak{b}^2-\mathfrak{a}^2) = 6^{-2/3} +
\mathcal{O}(\epsilon^{4/5}),
\end{equation}
where the error terms are uniform for $x-x_c=\mathcal{O}(\epsilon^{4/5})$.  Therefore,
\begin{equation}
\epsilon^{(2+\epsilon)/(3\epsilon)}e^{-\nu/\epsilon}\pu_m=\frac{1}{6^{1/3}}+\epsilon^{1/5}\frac{2^{1/15}}{3^{1/3}}H
\left(\frac{2^{1/15}}{3^{1/3}}\frac{x-x_c}{\epsilon^{4/5}}\right) +\mathcal{O}(\epsilon^{2/5}),
\end{equation}
\begin{equation}
\epsilon^{-(2-\epsilon)/(3\epsilon)}e^{\nu/\epsilon}\pv_m=\frac{1}{6^{1/3}}-\epsilon^{1/5}\frac{2^{1/15}}{3^{1/3}}H\left(\frac{2^{1/15}}{3^{1/3}}\frac{x-x_c}{\epsilon^{4/5}}\right) +\mathcal{O}(\epsilon^{2/5}),
\end{equation}
\begin{equation}
\epsilon^{1/3}\pp_m=-\frac{1}{6^{1/3}}-\epsilon^{2/5}\frac{2^{7/15}}{3^{1/3}}Y\left(\frac{2^{1/15}}{3^{1/3}}\frac{x-x_c}{\epsilon^{4/5}}\right)+\mathcal{O}(\epsilon^{3/5}),
\end{equation}
and
\begin{equation}
\epsilon^{1/3}\pq_m=\frac{1}{6^{1/3}}+\epsilon^{2/5}\frac{2^{7/15}}{3^{1/3}}Y\left(\frac{2^{1/15}}{3^{1/3}}\frac{x-x_c}{\epsilon^{4/5}}\right)+\mathcal{O}(\epsilon^{3/5}).
\end{equation}
Recalling \eqref{epsilon} and \eqref{eq:corner-nu}--\eqref{eq:corner-nu-prime}, we have the following theorem.
\begin{theorem}[Asymptotics of $\pu_m$, $\pv_m$, $\pp_m$, and $\pq_m$ near a corner 
of $T$]
Let $Y(t)$ be the real tritronqu\'ee solution of the Painlev\'e-I equation 
\eqref{eq:Painleve-I} specified by its asymptotic expansion 
\eqref{eq:tritronquee-asymptotics}, and let $H(t)$ be the associated Hamiltonian 
(see \eqref{eq:Hamiltonian}).  Also let $\mathcal{K}$ be a compact set in the $t$-plane that does not contain any poles of $Y(t)$, and recall $\pu_m=\pu_m(y)$, 
$\pv_m=\pv_m(y)$, $\pp_m=\pp_m(y)$, and $\pq_m=\pq_m(y)$ are defined for 
positive integer $m$ in 
\eqref{backlund-positive}
%, \eqref{backlund-negative}, 
and \eqref{log-derivative}, 
while $x_c:=-(9/2)^{2/3}$.  Then the formulae
%\begin{multline}
%m^{-2m/3}e^{-m/3}e^{m(x-x_c)/6^{1/3}}\pu_m\\
%{}=e^{-1/2}\left[\frac{1}{6^{1/3}}+\frac{1}{m^{1/5}}
%\frac{2^{1/15}}{3^{1/3}}H\left(\frac{2^{1/15}}{3^{1/3}}m^{4/5}(x-x_c)\right)+\mathcal{O}\left(\frac{1}{m^{2/5}}\right)\right],
%\end{multline}
\begin{equation}
\begin{split}
\left(\frac{m}{6}\right)^{-2m/3}e^{1/2-m/3}e^{m(x-x_c)/6^{1/3}}&\pu_m\left(\left(m-\tfrac{1}{2}\right)^{2/3}x\right) \\
& =1+\frac{2^{6/15}}{m^{1/5}}H\left(\frac{2^{1/15}}{3^{1/3}}m^{4/5}(x-x_c)\right)+\mathcal{O}\left(\frac{1}{m^{2/5}}\right),
\end{split}
\label{eq:corner-pu-asymp-final}
\end{equation}
%\begin{multline}
%m^{2(m-1)/3}e^{m/3}e^{-m(x-x_c)/6^{1/3}}\pv_m\\
%{}=e^{1/2}
%\left[\frac{1}{6^{1/3}}-\frac{1}{m^{1/5}}\frac{2^{1/15}}{3^{1/3}}H\left(\frac{2^{1/15}}{3^{1/3}}m^{4/5}(x-x_c)\right) +\mathcal{O}\left(\frac{1}{m^{2/5}}\right)\right],
%\end{multline}
\begin{equation}
\begin{split}
\left(\frac{m}{6}\right)^{2(m-1)/3}e^{m/3-1/2}e^{-m(x-x_c)/6^{1/3}}&\pv_m\left(\left(m-\tfrac{1}{2}\right)^{2/3}x\right) \\
 & =1-\frac{2^{6/15}}{m^{1/5}}H\left(\frac{2^{1/15}}{3^{1/3}}m^{4/5}(x-x_c)\right) +\mathcal{O}\left(\frac{1}{m^{2/5}}\right),
\label{eq:corner-pv-asymp-final}
\end{split}
\end{equation}
\begin{equation}
m^{-1/3}\pp_m\left(\left(m-\tfrac{1}{2}\right)^{2/3}x\right)=-\frac{1}{6^{1/3}}-\frac{1}{m^{2/5}}\frac{2^{7/15}}{3^{1/3}}Y\left(\frac{2^{1/15}}{3^{1/3}}m^{4/5}(x-x_c)\right)+\mathcal{O}\left(\frac{1}{m^{3/5}}\right),
\label{eq:corner-pp-asymp-final}
\end{equation}
and
\begin{equation}
m^{-1/3}\pq_m\left(\left(m-\tfrac{1}{2}\right)^{2/3}x\right)=\frac{1}{6^{1/3}}+\frac{1}{m^{2/5}}\frac{2^{7/15}}{3^{1/3}}Y\left(\frac{2^{1/15}}{3^{1/3}}m^{4/5}(x-x_c)\right)+\mathcal{O}\left(\frac{1}{m^{3/5}}\right)
\label{eq:corner-pq-asymp-final}
\end{equation}
all hold in the limit $m\to+\infty$ uniformly for 
\begin{equation}
\frac{2^{1/15}}{3^{1/3}}m^{4/5}(x-x_c)\in \mathcal{K}.
\end{equation}
\label{theorem:corner}
\end{theorem}
Note that, since $H'(t)=-Y(t)$, these formulae are formally consistent with the exact identities
\eqref{log-derivative} under the rescaling \eqref{x}.  The accuracy of these approximations is illustrated 
in Figures~\ref{fig:Um-corner} and \ref{fig:Pm-corner}. 
\begin{figure}[h]
%\begin{center}
%\includegraphics[width=2.5 in]{U10Corner.pdf}%
%\hspace{0.5 in}%
%\includegraphics[width=2.5 in]{U40Corner.pdf}
%\end{center}
\includegraphics{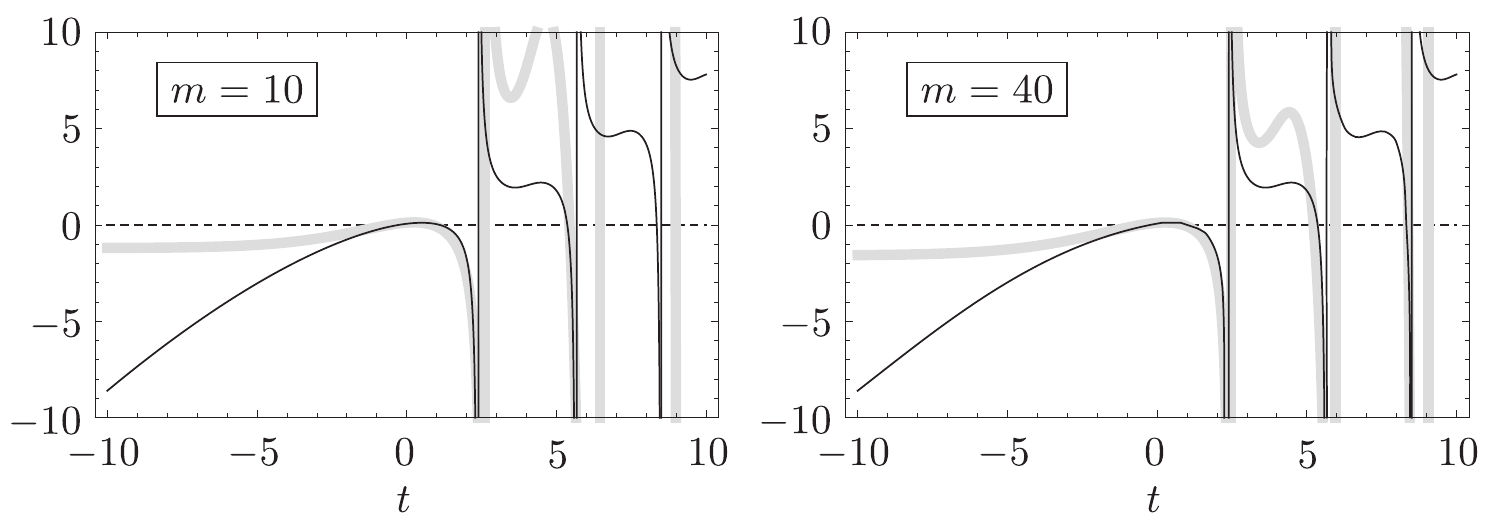}
\caption{The function $2^{-6/15}m^{1/5}((m/6)^{-2m/3}e^{1/2-m/3}e^{m(x-x_c)/6^{1/3}}\pu_m-1)$ (thick gray curves) and the tritronqu\'ee Hamiltonian $H$ (thin black curves) both plotted as functions of $t=2^{1/15}3^{-1/3}m^{4/5}(x-x_c)$ for $m=10$ (left) and $m=40$ (right).  The function $H$ was numerically computed using the ``pole field solver'' code of Fornberg and Weideman \cite{FornbergW11}.}
\label{fig:Um-corner}
\end{figure}
\begin{figure}[h]
%\begin{center}
%\includegraphics[width=2.5 in]{P10Corner.pdf}%
%\hspace{0.5 in}%
%\includegraphics[width=2.5 in]{P40Corner.pdf}
%\end{center}
\includegraphics{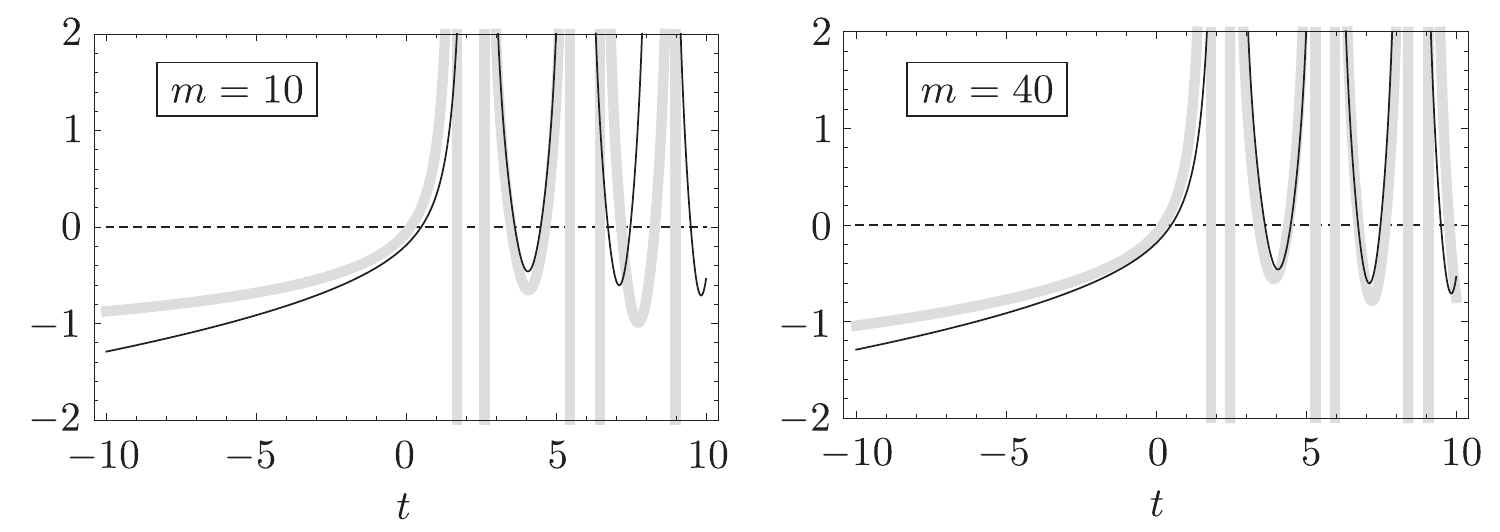}
\caption{The function $-2^{-7/15}3^{1/3}m^{2/5}(m^{-1/3}\pp_m+6^{-1/3})$ (thick gray curves) and the tritronqu\'ee solution $Y$ (thin black curves) both plotted as functions of $t=2^{1/15}3^{-1/3}m^{4/5}(x-x_c)$ for $m=10$ (left) and $m=40$ (right).  The function $Y$ was numerically computed using the ``pole field solver'' code of Fornberg and Weideman \cite{FornbergW11}.}
\label{fig:Pm-corner}
\end{figure}
It is interesting that, according to \eqref{eq:corner-pp-asymp-final}--\eqref{eq:corner-pq-asymp-final}, a function with only simple poles like $\pp_m$ or $\pq_m$ can be accurately approximated by a function like $Y$ having only double poles.  
This is possible because the approximation is only valid on sets that avoid the poles, and from Figure~\ref{fig:Pm-corner} one can see that on the scale where $t$ is fixed (i.e., $x-x_0=\mathcal{O}(m^{-4/5})$) there are pairs of simple poles of $\pp_m$ with opposite residues of $\pm 1$ (visible in Figure~\ref{fig:Pm-corner} as nearby pairs of vertical gray lines) that converge in the limit $m\to\infty$ to a common limit, namely a (double) pole of $Y$ (see also 
Figure \ref{tritronquee-vs-p10-p40}).  The 
approximations  guaranteed by Theorem~\ref{theorem:corner} are simply not valid where these coalescence events are occurring.  

The same phenomenon occurs at the level of the functions $\pu_m$ and $\pv_m$, where simple zeros and simple poles converge in pairs to merge and form the simple poles of $H$ (these do not cancel since a zero of $\pu_m$ or $\pv_m$ corresponds to a non-zero 
value of $H$ for sufficiently large $m$, as can be seen from 
\eqref{eq:corner-pu-asymp-final} and \eqref{eq:corner-pv-asymp-final}).  The fact that when $x$ is close to $x_c$ every neighborhood of a pole of $\pu_m$ ultimately (as $m\to\infty$) contains a simple zero as well means that the technique described in Remark~\ref{remark:cheese} of exploiting the B\"acklund transformations \eqref{backlund-positive}--\eqref{backlund-negative} to analyze the reciprocal $1/\pu_m$ and hence obtain asymptotics without the need of excluding ``holes'' near the poles of $Y$ does not apply in the case of the analysis near the corner points of $\partial T$.

\appendix
\section{Proofs of Key Lemmas}
\label{sec:Appendix}
\renewcommand{\theequation}{A-\arabic{equation}}
\subsection{Proof of Lemma~\ref{lemma-S-univalent}}
\label{sec:lemma-S-univalent}
\begin{proof}
By the Argument Principle, it suffices to show that $S$ maps the boundary of its domain of analyticity onto a simple closed curve on the Riemann sphere in a $1$-to-$1$ fashion.  Consider $x$ traversing the (open) contour labeled $\circled{1}$ in Figure~\ref{fig:Sigma-S}.  As $x$ is real, writing $S(x)=u+iv$ with $u,v\in\mathbb{R}$, \eqref{cubic-equation} splits into two real equations:
\begin{equation}
3u^3-9uv^2+4xu+8=0\quad\text{and}\quad 9u^2v-3v^3+4xv=0.
\label{eq:uv-eqns}
\end{equation}
Because $S$ has a jump across $\Sigma_S$, by Schwarz symmetry $v\neq 0$.  Dividing the second equation by $v$ and substituting into the first gives the relation
\begin{equation}
6u^3+6uv^2=8.
\end{equation}
Introducing polar coordinates by $u=\rho\cos(\vartheta)$ and $v=\rho\sin(\vartheta)$, $\rho\ge 0$, this relation reads
\begin{equation}
6\rho^3\cos(\vartheta)=8.
\end{equation}
We therefore have $\cos(\vartheta)>0$, so the image of the arc $\circled{1}$ in polar form reads
\begin{equation}
\rho=\left(\frac{4}{3}\right)^{1/3}(\sec(\vartheta))^{1/3}.
\label{eq:S-1-polar-form}
\end{equation}
Since from the second equation in \eqref{eq:uv-eqns} we have $4x=3v^2-9u^2$, with the use of \eqref{eq:S-1-polar-form} we get a formula for $x$ as a function of $\vartheta$:
\begin{equation}
x=3^{1/3}4^{2/3}(\sec(\vartheta))^{2/3}(1-4(\cos(\vartheta))^2).
\end{equation}
Note that $\vartheta=0$ gives $x=x_c$, while $\vartheta=\pi/3$ gives $x=0$.  Moreover, by a direct calculation one sees that 
\begin{equation}
\frac{dx}{d\vartheta}=3^{1/3}4^{2/3}\frac{2}{3}\sin(\vartheta)(\cos(\vartheta))^{1/3}(8+(\sec(\vartheta))^2)>0,\quad 0<\vartheta<\frac{\pi}{3}.
\end{equation}
Hence the arc $\circled{1}$ is mapped onto its image, given by the polar equation \eqref{eq:S-1-polar-form} with $0<\vartheta<\pi/3$ (obviously a simple arc), in a $1$-to-$1$ fashion.  By Schwarz symmetry of $S$ and the rotational symmetry \eqref{eq:S-rotational-symmetry}, the remaining five arcs illustrated in Figure~\ref{fig:Sigma-S} are also mapped to their images in a $1$-to-$1$ fashion, and their images are all simple arcs subtending the remaining five sectors of the complex $S$-plane each with opening angle $\pi/3$.  Since $S$ is continuous along the curve $\circled{1}\to\circled{2}\to\circled{3}\to\circled{4}\to\circled{5}\to\circled{6}$, its image is therefore a simple closed curve.  $S$ maps the domain $\mathbb{C}\setminus\Sigma_S$ to the interior of this latter curve.  See Figure~\ref{fig:S-plane}.
\end{proof}

\subsection{Proof of Lemma~\ref{lemma:d-univalent}}
\label{sec:lemma:d-univalent}
\begin{proof}
The fact that $\mathfrak{d}(x^*)=\mathfrak{d}(x)^*$ holds if $x\in\mathcal{S}_0$ follows directly from 
\eqref{eq:c-reflection} using the definition \eqref{eq:d-define}.  To prove the univalence, we begin with the integral formula \eqref{eq:c-function-define} and we introduce the change of integration variable given by $2\zeta=S(x)+\Delta(x)w$ (recall that $\Delta(x):=-4i/(-3S(x))^{1/2}$ holds as an identity for $x\in\mathcal{S}_0$ with the principal branch implied) to yield
\begin{equation}
\mathfrak{d}(x)=\frac{3\Delta(x)^3}{16}\int_{-1}^{-2S(x)/\Delta(x)}\left(w+\frac{2S(x)}{\Delta(x)}\right)
r(w;-1,1)\,dw-\frac{i\pi}{2}.
\end{equation}
Eliminating $\Delta$ in favor of $S$, we obtain 
$\mathfrak{d}(x)=I(t)$ where $t:=i(\tfrac{3}{4})^{1/2}(-S(x))^{3/2}$ (principal branch) and where
\begin{equation}
\label{proofs-I-def}
I(t):=-\frac{2}{t}\int_{-1}^{t}(w-t)r(w;-1,1)\,dw-\frac{i\pi}{2}.  
\end{equation}
This function has the reflection symmetry $I(-t^*)=I(t)^*$ (Schwarz symmetry through the imaginary axis).  
Now, $t$ is obviously analytic and univalent on $\mathcal{S}_0$, as can be easily seen with the help of Lemma~\ref{lemma-S-univalent}.  The corresponding conformal image $\tau$ is illustrated in Figure~\ref{fig:t-plane}.  
\begin{figure}[h]
\includegraphics{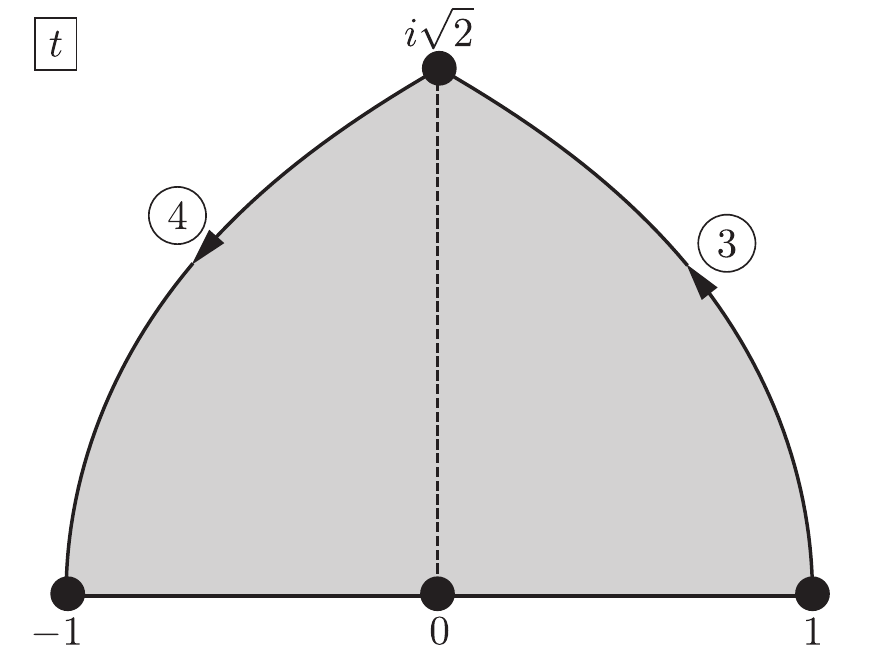}
\caption{The conformal image $\tau$ of the sector $\mathcal{S}_0$ (parametrized by $x$) in the $t$-plane, where $t:=i(\tfrac{3}{4})^{1/2}(-S(x))^{3/2}$.  The boundary curve labeled $\circled{3}$ (corresponding to Figures~\ref{fig:Sigma-S} and \ref{fig:S-plane}) has the parametrization $t=(\sec(\tfrac{2}{3}\vartheta))^{1/2}e^{i\vartheta}$, $0<\vartheta<\pi/2$, and the domain $\tau$ is symmetric through the imaginary axis.}
\label{fig:t-plane}
\end{figure}
It is now obvious that $\mathfrak{d}$ is analytic on $\mathcal{S}_0$, and by composition it remains to prove that $I:\tau\to\mathbb{C}$ is univalent.  

By the Argument Principle, it is sufficient to show that $I$ maps $\partial\tau$ to a simple closed curve on the Riemann sphere in a $1$-to-$1$ fashion.  Note that by Leibniz' rule,
\begin{equation}
I'(t)=\frac{2}{t^2}\int_{-1}^t(w-t)r(w;-1,1)\,dw +\frac{2}{t}\int_{-1}^tr(w;-1,1)\,dw =\frac{2}{t^2}\int_{-1}^t
wr(w;-1,1)\,dw.
\end{equation}
Moreover, the resulting integrand $wr(w;-1,1)$ has an obvious antiderivative, and therefore
\begin{equation}
I'(t)=\frac{2}{t^2}\int_{-1}^twr(w;-1,1)\,dw=\frac{2}{t^2}\int_{-1}^t\frac{d}{dw}\left(\frac{1}{3}r(w;-1,1)^3\right)\,dw=\frac{2r(t;-1,1)^3}{3t^2}.
\label{eq:I-prime}
\end{equation}
By using the fact that $r(\cdot;-1,1)$ is positive imaginary on the top edge of its branch cut, it is easy to see that the interval $0<t<1$ on the boundary of $\tau$, oriented left-to-right, is mapped by $I$ onto
the imaginary half-line from $i\pi/2$ to $i\infty$, oriented downward, and since $I'(t)\neq 0$ for $0<t<1$ (again evaluated along the top edge of the branch cut for $r$) this mapping is $1$-to-$1$.  

We now claim that if the arc of $\partial\tau$ labeled $\circled{3}$ in Figure~\ref{fig:t-plane} is parametrized by 
$t=t(\vartheta):=(\sec(\tfrac{2}{3}\vartheta))^{1/2}e^{i\vartheta}$, $0<\vartheta<\pi/2$, then 
\begin{equation}
\Re\left(\frac{d}{d\vartheta}I(t(\vartheta))\right)<0,\quad 0<\vartheta<\frac{\pi}{2},
\label{eq:inequality-real-part}
\end{equation}
and
\begin{equation}
\Im\left(\frac{d}{d\vartheta}I(t(\vartheta))\right)<0,\quad 0<\vartheta<\frac{\pi}{2}.
\label{eq:inequality-imaginary-part}
\end{equation}
To prove \eqref{eq:inequality-real-part}, we use $r(w;-1,1)^2=w^2-1$ and \eqref{eq:I-prime} with the chain rule to obtain
\begin{equation}
\frac{d}{d\vartheta}I(t(\vartheta))=\frac{2}{3}r(t(\vartheta);-1,1)\frac{t(\vartheta)^2-1}{t(\vartheta)^2}t'(\vartheta).
\end{equation}
But $t(\vartheta)$ lies in the first quadrant, and it is easy to check that $r(\cdot;-1,1)$ maps the open first quadrant to itself.  Also, 
\begin{multline}
\frac{t(\vartheta)^2-1}{t(\vartheta)^2}t'(\vartheta)=-\frac{2(\cos(\tfrac{1}{3}\vartheta))^2\sin(\tfrac{1}{3}\vartheta)}{3(\cos(\tfrac{2}{3}\vartheta))^{3/2}}\left(9-32(\sin(\tfrac{1}{3}\vartheta))^2+32(\sin(\tfrac{1}{3}\vartheta))^4\right) \\{}+ i\frac{2(\sin(\tfrac{1}{3}\vartheta))^2\cos(\tfrac{1}{3}\vartheta)}{3(\cos(\tfrac{2}{3}\vartheta))^{3/2}}\left(9-32(\cos(\tfrac{1}{3}\vartheta))^2+32(\cos(\tfrac{1}{3}\vartheta))^4\right),
\end{multline}
and since the quadratic $9-32u+32u^2$ has complex roots, it follows that the above expression lies in the second quadrant of the complex plane.  This proves \eqref{eq:inequality-real-part}, since the product of quantities in the first and second quadrants lies in the left half-plane.
From \eqref{eq:inequality-real-part} it follows that
\begin{equation}
\frac{d}{d\vartheta}I(t(\vartheta))=-\left(\left[\frac{d}{d\vartheta}I(t(\vartheta))\right]^2\right)^{1/2},
\label{eq:root-squared}
\end{equation}
where the principal branch is meant.  But again using $r(w;-1,1)^2=w^2-1$,
\begin{equation}
\begin{split}
\Im\left(\left[\frac{d}{d\vartheta}I(t(\vartheta))\right]^2\right)&=\Im\left(\frac{4}{9}\frac{(t(\vartheta)^2-1)^3}{t(\vartheta)^4}t'(\vartheta)^2\right)\\ & = \frac{4}{27}\sin(\tfrac{2}{3}\vartheta)(\sec(\tfrac{2}{3}\vartheta))^3
\left(3\sin(\tfrac{2}{3}\vartheta)+2\sin(2\vartheta)\right)^2>0,\quad 0<\vartheta<\frac{\pi}{2}.
\end{split}
\end{equation}
Since the principal-branch square root of a quantity in the upper half-plane lies in the first quadrant, it follows from \eqref{eq:root-squared} that $d I(t(\vartheta))/d\vartheta$ lies in the third quadrant, and therefore both \eqref{eq:inequality-real-part} \emph{and} \eqref{eq:inequality-imaginary-part} hold true.
Now, by the reflection symmetry property $I(-t^*)=I(t)^*$ the image of the terminal point of the arc labeled $\circled{3}$ in Figure~\ref{fig:t-plane} (corresponding to $\vartheta=\pi/2$) lies on the real axis in the $I$-plane.  Using this information and the 
inequalities \eqref{eq:inequality-real-part}--\eqref{eq:inequality-imaginary-part} it follows that the arc labeled $\circled{3}$ in Figure~\ref{fig:t-plane} is mapped in a $1$-to-$1$ fashion onto a simple arc that begins at $I=i\pi/2$, terminates at a strictly negative real value of $I$, and otherwise satisfies $\Re(I)<0$ and $\Im(I)>0$.  Applying again the reflection symmetry of $I$ shows that the whole boundary of $\tau$ is mapped onto a simple closed curve on the Riemann sphere in a $1$-to-$1$ fashion.  The image of $\tau$ is illustrated in Figure~\ref{fig:I-plane}.  
This completes the proof of univalence of $I:\tau\to\mathbb{C}$ and hence of $\mathfrak{d}:\mathcal{S}_0\to\mathbb{C}$.
\end{proof}

\subsection{Proof of Lemma~\ref{lemma:ell}}
\label{sec:lemma:ell}
\begin{proof}
By Schwarz symmetry of $\ell$, it suffices to show that $\Im(\ell(x))$ varies from $0$ to $\pi/4$ as $x$ varies along $\partial T$ from $x_e$ (the real midpoint of $\partial T\cap\mathcal{S}_0$) to $x_ce^{-2\pi i/3}$ (the upper corner).  Since $e^{\ell(x)}$ is analytic and nonvanishing for $x\in\mathcal{S}_0$, the increment of the argument of $e^{\ell(x)}$ along any closed curve in $\mathcal{S}_0$ must vanish.  We use the closed curve $C$ consisting of the following oriented arcs:
%We can make use of the analysis in \S\ref{section:opening-angle} to
%determine how $\Im(\ell(x))$ varies along the edge of $\partial T$ with $|\arg(x)|<\pi/3$.  
%The function $e^{\ell(x)}$ is analytic and non-vanishing for $x\in\mathcal{S}_0$, a fact that enables the
%calculation of the increment of the argument of $e^{\ell(x)}$ as $x$ varies along $\partial T$ from the point $x=x_e>0$ where $\partial T$ intersects the positive real axis to the upper corner $x=x_ce^{-2\pi i/3}$.  Indeed, consider the domain $\Omega$ bounded by the simple closed curve $\partial\Omega$ consisting of the following oriented arcs:
\begin{itemize}
\item $C_1$ being the real interval $[x_e,M]$ for $M\gg 1$, oriented left-to-right,
\item $C_2$ being parametrized by $Me^{i\phi}$ with $\phi$ increasing from $\phi=0$ to $\phi=\pi/3$,
\item $C_3$ being parametrized by $te^{i\pi/3}$ with $t$ decreasing from $t=M$ to 
$t=|x_c|+\delta$ for $\delta\ll 1$,
\item $C_4$ being parametrized by $x_ce^{-2\pi i/3}+\delta e^{i\phi}$ with $\phi$ decreasing from $\pi/3$ to an angle $\phi_c(\delta)$ so that $C_4$ terminates on
$\partial T$, and,
\item $C_5$ being the arc of $\partial T$ that closes $C$, oriented toward $x_e$.
\end{itemize}
Note that, according to Proposition~\ref{prop:corner} and the basic reflection symmetry of $T$ through the lines with slopes $0$ and $\pm\sqrt{3}$, the angle $\phi_c(\delta)$ tends to $\pi/3-4\pi/5=-7\pi/15$ as $\delta\downarrow 0$.  
%According to the Argument Principle, 
Since
the net increment of $\Im(\ell(x))$ as $x$ traverses $C$ is zero, 
%or put differently, 
the increment as $x$ varies backwards along the edge $C_5$ is given by the sum of the increments as $x$ varies along $C_1,\dots,C_4$ in the direction of their given orientation.  We calculate the latter in the limit $\delta\downarrow 0$ and $M\uparrow\infty$.

It is easy to see that $\Im(\ell(x))=0$ holds identically for all $x\in C_1$ as this contour lies on the positive real axis.  Along $C_2$ we calculate the increment of $\Im(\ell(x))$ by considering the limit $M\to\infty$.  In this situation, the fact that $S(x)\sim -2/x$ for large $|x|$
shows that $\arg(i\Delta)$ increases by $\pi/6$ along $C_2$, while (since $r(\zc(x);x)^2=S(x)^2-4/(3S(x))$ holds) $\arg(r(\zc(x);x)^{-5/2})$ increases by $-5\pi/12$ along $C_2$.  Therefore the increment of $\Im(\ell(x))$ as $x$ varies along $C_2$ in the limit $M\to\infty$ is $-\pi/4$.  It can be checked that $\Im(\ell(x))$ is again constant as $x$ varies along $C_3$.  To calculate the increment of $\Im(\ell(x))$ along $C_4$, we consider the limit $\delta\downarrow 0$ in which $i\Delta(x)$ tends to a fixed nonzero limiting value uniformly for $x\in C_4$; therefore it remains to consider the variation of $r(\zc(x);x)^{-5/2}$ along this arc.  It can be shown by analyzing $S(x)$ for $x$ near the corner $x_ce^{-2\pi i/3}$ that $r(\zc(x);x)^2$ is locally proportional to $(x-x_ce^{-2\pi i/3})^{1/2}$, and hence $\arg(r(\zc(x);x)^{-5/2})$ increases by $\pi/2$ along $C_4$ in the limit $\delta\downarrow 0$.  Therefore the increment of $\Im(\ell(x))$ as $x$ varies along $C_4$ is also $\pi/2$ in this limit.  Combining these results shows that the increment of $\Im(\ell(x))$ as $x$ varies along $\partial T$ from $x=x_e>0$ to $x=x_ce^{-2\pi i/3}$ (backwards along $C_5$ in the limit $\delta\downarrow 0$) is exactly $\pi/4$.
\end{proof}

\subsection{Proof of Lemma~\ref{lemma:conformal}}
\label{sec:lemma:conformal}
\begin{proof}
By analogy with the introduction of $q$ as the principal branch square root of $\mathfrak{a}-z$, we denote by $Q$ the principal branch of $(-W)^{1/2}$.  Thus $|\arg(-W)|<\pi$ and $|\arg(Q)|<\pi/2$.  The relation
\eqref{eq:basic-h-equation} obviously implies the equation
\begin{equation}
f(q;x,\mathfrak{a})=-\frac{4}{5}Q^5-sQ.
\label{eq:rewritten-h-equation}
\end{equation}
We want to determine $\mathfrak{a}$ and $s$ as functions of $x$ near $x_c$ so that this equation has a solution $Q=Q(q;x)$ that is a univalent analytic function of $q$ near $q=0$.  Since both $f$ and the quintic polynomial on the right-hand side of \eqref{eq:rewritten-h-equation} are odd, we will impose the condition that $Q$ be an odd function of $q$.  Also, since both $q$ and $Q$ denote principal branches of square roots, it will be necessary that $Q$ map the right half-plane to itself.  Given such a $Q$, we obtain $W=-Q((\mathfrak{a}(x)-z)^{1/2};x)^2$, which by oddness and univalence of $Q$ will clearly be a univalent function of $z$ taking $z=\mathfrak{a}(x)$ to $W=0$.

We first note that if $x=x_c$, $\mathfrak{a}=a_c$, and $s=0$, then it is easy to solve for the function $Q$, which we denote by $Q_c$ in this degenerate situation.  Indeed, according to \eqref{eq:fderivs-criticality}, we can write 
\begin{equation}
f(q;x_c,a_c)=%-\frac{6}{5}(-a_c)^{1/2}
-\frac{6^{5/6}}{5}q^5F(q^2),
\end{equation}
where $F(\cdot)$ is an analytic function of its argument satisfying 
\begin{equation}
F(q^2)=1+\frac{f^{(7)}(0;x_c,a_c)}{42f^{(5)}(0;x_c,a_c)}q^2+\mathcal{O}(q^4)=1+\frac{3^{1/3}5}{2^{8/3}7}q^2+\mathcal{O}(q^4),\quad q\to 0.
\end{equation}
We may therefore solve for $Q_c(q)$ by taking an analytic fifth root:
\begin{equation}
Q_c(q)=\frac{3^{1/6}}{2^{7/30}}
%\left(\frac{3}{2}\right)^{1/5}(-a_c)^{1/10}
qF(q^2)^{1/5},
\end{equation}
where the last factor is an even analytic function that equals $1$ for $q=0$.  Clearly, $Q_c$ is an odd analytic function of $q$ that is univalent for sufficiently small $q$ because 
\begin{equation}
Q_c'(0)=
%\left(\frac{3}{2}\right)^{1/5}(-a_c)^{1/10}=
\frac{3^{1/6}}{2^{7/30}}>0.  
\label{eq:Yc-prime-zero}
\end{equation}
It is also a Schwarz-symmetric function, so since $Q_c'(0)>0$, $Q_c$ locally maps the right half-plane to itself as desired.  
%For future reference, we note that also
%\begin{equation}
%Y'''_c(0)=\frac{3^{3/2}}{2^{19/10}7}.
%\label{eq:Yc-triple-prime-zero}
%\end{equation}
Since $W_c(z):=-Q_c((a_c-z)^{1/2})^2$, the conformal map $W_c$ satisfies
\begin{equation}
W_c(a_c)=0\quad\text{and}\quad W_c'(a_c)=Q_c'(0)^2=\frac{3^{1/3}}{2^{7/15}}>0.
%,\quad\text{and}\quad
%W_c''(a_c)=-\frac{2}{3}Y_c'(0)Y_c'''(0)=-\frac{3^{2/3}}{2^{17/15}7}.
\end{equation}

Next, we consider how $(\mathfrak{a},s,Q(\cdot))$ should evolve as $x$ moves away from $x=x_c$.  Holding $q$ fixed at some small value and assuming that $(\mathfrak{a},s,Q(q))$ are differentiable with respect to $x$ near $x_c$, we obtain from \eqref{eq:rewritten-h-equation} that
\begin{equation}
\dot{Q}(q)=-\frac{f_x(q)+f_\mathfrak{a}(q)\dot{\mathfrak{a}} + Q(q)\dot{s}}{4Q(q)^4+s},
\label{eq:Ydot-a-s}
\end{equation}
where
\begin{equation}
\dot{Q}(q):=\frac{\partial Q}{\partial x}(q;x),\quad\dot{\mathfrak{a}}:=\frac{d\mathfrak{a}}{dx}(x),\quad
\dot{s}:=\frac{ds}{dx}(x).
\end{equation}
The roots of the denominator of \eqref{eq:Ydot-a-s} are $Q(q)=\Upsilon, i\Upsilon,-\Upsilon,-i\Upsilon$, where $\Upsilon:=(-\tfrac{1}{4}s)^{1/4}$.  If the right-hand side of \eqref{eq:Ydot-a-s} is to be an analytic function of small $q$, we will need to require that the numerator also vanishes at the four corresponding (under the inverse $Q^{-1}(\cdot)$ to the univalent function $Q(\cdot)$) values of $q$.
We therefore insist that
\begin{equation}
\begin{split}
f_x(Q^{-1}(\Upsilon))+f_\mathfrak{a}(Q^{-1}(\Upsilon))\dot{\mathfrak{a}}+\Upsilon \dot{s}&=0,\\
f_x(Q^{-1}(i\Upsilon))+f_\mathfrak{a}(Q^{-1}(i\Upsilon))\dot{\mathfrak{a}}+i\Upsilon\dot{s}&=0
\end{split}
\label{eq:adot-sdot-system}
\end{equation}
(that the numerator of \eqref{eq:Ydot-a-s} also vanishes at the points $q$ corresponding to $-\Upsilon$ and $-i\Upsilon$ follows automatically by the oddness of both $f_x(\cdot)$ and $f_\mathfrak{a}(\cdot)$ and the presumed oddness of $Q(\cdot)$ and hence of its inverse).  Supposing for the moment that $s\neq 0$ (and hence $\Upsilon\neq 0$), solving \eqref{eq:adot-sdot-system} for $\dot{\mathfrak{a}}$ and $\dot{s}$ gives
\begin{equation}
\begin{split}
\dot{\mathfrak{a}}&=-\frac{f_x(Q^{-1}(i\Upsilon))-if_x(Q^{-1}(\Upsilon))}{f_\mathfrak{a}(Q^{-1}(i\Upsilon))-if_\mathfrak{a}(Q^{-1}(\Upsilon))},\\
\dot{s}&=\frac{1}{\Upsilon}\cdot\frac{f_\mathfrak{a}(Q^{-1}(\Upsilon))f_x(Q^{-1}(i\Upsilon))-f_\mathfrak{a}(Q^{-1}(i\Upsilon))f_x(Q^{-1}(\Upsilon))}{f_\mathfrak{a}(Q^{-1}(i\Upsilon))-if_\mathfrak{a}(Q^{-1}(\Upsilon))}.
\end{split}
\label{eq:adot-sdot-formulas1}
\end{equation}
After using \eqref{eq:adot-sdot-formulas1} to eliminate $\dot{\mathfrak{a}}$ and $\dot{s}$ from the right-hand side of \eqref{eq:Ydot-a-s}, we intend to view \eqref{eq:Ydot-a-s} and \eqref{eq:adot-sdot-formulas1} as a first-order differential equation
on an appropriate space of triples $(\mathfrak{a},s,Q(\cdot))$ for which we have the initial condition $(a_c,0,Q_c(\cdot))$ when $x=x_c$.

To properly formulate this problem as a dynamical system, 
pick
a radius $\rho>0$ sufficiently small 
%that $Y_c(q)$ is analytic and $|Y_c'(q)/Y_c'(0)-1|\le 1/2$ for $|q|\le 2\rho$, a condition ensuring univalence of $Y_c$ for $|q|\le 2\rho$.  Then 
and consider
the Banach space $\mathcal{B}_0$ consisting of functions $Q:\mathbb{C}\to\mathbb{C}$ analytic for $|q|<2\rho$ and continuous for $|q|\le 2\rho$ for which 
\eq
Q(-q)=-Q(q)\quad\text{and}\quad\|Q\|_0:=\sup_{|q|=2\rho}|Q(q)|<\infty.
\endeq
(That $\|\cdot\|_0$ defines a norm on $\mathcal{B}_0$ follows from the Maximum Modulus Principle.)
Clearly $Q_c(\cdot)\in\mathcal{B}_0$ if $\rho$ is sufficiently small.  The space we will work in is $(\mathfrak{a},s,Q(\cdot))\in\mathcal{B}:=\mathbb{C}^2\times\mathcal{B}_0$, which is itself a Banach space with norm $\|(\mathfrak{a},s,Q(\cdot)\|:=|\mathfrak{a}|+|s|+\|Q\|_0$.  Let $D\subset\mathcal{B}$ be a ball of sufficiently small radius $\delta>0$ in $\mathcal{B}$ centered at $(a_c,0,Q_c(\cdot))$.
We claim that equations \eqref{eq:Ydot-a-s} (with $\dot{\mathfrak{a}}$ and $\dot{s}$ eliminated) and \eqref{eq:adot-sdot-formulas1} define a vector field on $D$, that is, $(\dot{\mathfrak{a}},\dot{s},\dot{Q}(\cdot))$ is continuous in $x$ near $x_c$ and Lipschitz as a function of $(\mathfrak{a},s,Q(\cdot))\in D$.

The key to proving the claim is to rewrite $\dot{\mathfrak{a}}$ and $\dot{s}$ given by \eqref{eq:adot-sdot-formulas1} in a form that makes the rest of the proof straightforward.  
%to show that if $Y\in\mathcal{B}_0$, then $\dot{\mathfrak{a}}$ and $\dot{s}$
%are analytic functions of $\mathfrak{a}$ and $s$ including where the denominators vanish, i.e., the singularities are removable.  Indeed, applying 
With the help of the Lagrange-B\"urmann inversion formula, we have
\begin{equation}
\begin{split}
f_x(Q^{-1}(i\Upsilon))-if_x(Q^{-1}(\Upsilon))&=\frac{1}{2\pi i}\oint_{|q|=\rho}
\frac{f_x(q)Q'(q)\,dq}{Q(q)-i\Upsilon} -\frac{1}{2\pi i}\oint_{|q|=\rho}\frac{if_x(q)Q'(q)\,dq}{Q(q)-\Upsilon}
\\ &=-\frac{4\Upsilon^3}{\pi}X_0(\mathfrak{a},s,Q(\cdot);x),
\end{split}
\label{eq:fx-diff}
\end{equation}
where for any integer $p$ we define
\begin{equation}
X_p(\mathfrak{a},s,Q(\cdot);x):=\oint_{|q|=\rho}\frac{f_x(q;x,\mathfrak{a})Q(q)^pQ'(q)\,dq}{4Q(q)^4+s},
\end{equation}
and where the contour of integration is a positively oriented circle of radius $\rho$ that encloses all four roots of $4Q(q)^4+s$.  To obtain the second representation from the first, one makes the substitution $q\mapsto -q$, exploits oddness of $Q$ and $f_x$, and averages the resulting equivalent formula with the original one, using the identity $4\Upsilon^4+s=0$ to simplify the denominator.  In exactly the same way,
\begin{equation}
f_\mathfrak{a}(Q^{-1}(i\Upsilon))-if_\mathfrak{a}(Q^{-1}(\Upsilon))=-\frac{4\Upsilon^3}{\pi}A_0(\mathfrak{a},s,Q(\cdot);x),
\label{eq:fa-diff}
\end{equation}
where for any integer $p$ we define
\begin{equation}
A_p(\mathfrak{a},s,Q(\cdot);x):=
\oint_{|q|=\rho}\frac{f_\mathfrak{a}(q;x,\mathfrak{a})Q(q)^pQ'(q)\,dq}{4Q(q)^4+s}.
\end{equation}
Similarly,
\begin{equation}
\begin{split}
f_x(Q^{-1}(i\Upsilon))+if_x(Q^{-1}(\Upsilon))&=\frac{4\Upsilon}{\pi}X_2(\mathfrak{a},s,Q(\cdot);x),\\
f_\mathfrak{a}(Q^{-1}(i\Upsilon))+if_\mathfrak{a}(Q^{-1}(\Upsilon))&=\frac{4\Upsilon}{\pi}A_2(\mathfrak{a},s,Q(\cdot);x).
\end{split}
\label{eq:fx-fa-sum}
\end{equation}
Using \eqref{eq:fx-diff}--\eqref{eq:fx-fa-sum},
\eq
\begin{split}
f_\mathfrak{a}(Q^{-1}(\Upsilon))f_x(&Q^{-1}(i\Upsilon))-f_\mathfrak{a}(Q^{-1}(i\Upsilon))f_x(Q^{-1}(\Upsilon))\\
&=\frac{i}{2}\left[f_x(Q^{-1}(i\Upsilon))+if_x(Q^{-1}(\Upsilon))\right]
\left[f_\mathfrak{a}(Q^{-1}(i\Upsilon))-if_\mathfrak{a}(Q^{-1}(\Upsilon))\right]\\
&\quad\quad{}-\frac{i}{2}\left[f_\mathfrak{a}(Q^{-1}(i\Upsilon))+if_\mathfrak{a}(Q^{-1}(\Upsilon))\right]
\left[f_x(Q^{-1}(i\Upsilon))-if_x(Q^{-1}(\Upsilon))\right]\\
&=-\frac{8i\Upsilon^4}{\pi^2}X_2(\mathfrak{a},s,Q(\cdot);x)A_0(\mathfrak{a},s,Q(\cdot);x) + \frac{8i\Upsilon^4}{\pi^2}A_2(\mathfrak{a},s,Q(\cdot);x)X_0(\mathfrak{a},s,Q(\cdot);x).
\end{split}
\endeq
Finally, applying these results to the representations of $\dot{\mathfrak{a}}$ and $\dot{s}$ given in  \eqref{eq:adot-sdot-formulas1} yields
\begin{equation}
\begin{split}
\dot{\mathfrak{a}}&=-\frac{X_0(\mathfrak{a},s,Q(\cdot);x)}
{A_0(\mathfrak{a},s,Q(\cdot);x)},\\
\dot{s}&=\frac{2i}{\pi}X_2(\mathfrak{a},s,Q(\cdot);x)-
\frac{2i}{\pi}\cdot\frac{A_2(\mathfrak{a},s,Q(\cdot);x)X_0(\mathfrak{a},s,Q(\cdot);x)}
{A_0(\mathfrak{a},s,Q(\cdot);x)}.
\end{split}
\label{eq:adot-sdot-formulas2}
\end{equation}

Recall that $f_x(q;x,\mathfrak{a})$ and $f_\mathfrak{a}(q;x,\mathfrak{a})$ are analytic in all three arguments near $q=0$, $x=x_c$, and $\mathfrak{a}=a_c$.  Note also that taking $\delta$ sufficiently small bounds $s\in\mathbb{C}$ and $Q-Q_c\in\mathcal{B}_0$ and hence bounds $4Q(q)^4+s$ away from zero for $|q|=\rho$.  It therefore follows easily that if $\mathcal{I}=\mathcal{I}(\mathfrak{a},s,Q(\cdot);x)$ represents any of the four integrals $X_0$, $X_2$, $A_0$, or $A_2$, then $\mathcal{I}$ is continuous in $x$ near $x_c$ and for each such $x$ satisfies an estimate of the form
\begin{multline}
|\mathcal{I}(\mathfrak{a}_2,s_2,Q_2(\cdot);x)-\mathcal{I}(\mathfrak{a}_1,s_1,Q_1(\cdot);x)|\le c\left(|\mathfrak{a}_2-\mathfrak{a}_1| + |s_2-s_1| 
\vphantom{+ \sup_{|q|=\rho}|Q_2(q)-Q_1(q)| + \sup_{|q|=\rho}|Q'_2(q)-Q'_1(q)|}
\right.\\
\left.{}+ \sup_{|q|=\rho}|Q_2(q)-Q_1(q)| + \sup_{|q|=\rho}|Q'_2(q)-Q'_1(q)|\right)
\end{multline}
for some $c>0$ depending only on $\delta$.  Using the Cauchy Integral Formula to express $Q_j'(q)$ with $|q|=\rho$ in terms of values of $Q_j(q)$ with $|q|=2\rho$ and applying the Maximum Modulus Principle shows that, for some different constant $c$,
\begin{equation}
|\mathcal{I}(\mathfrak{a}_2,s_2,Q_2(\cdot);x)-\mathcal{I}(\mathfrak{a}_1,s_1,Q_1(\cdot);x)|\le c\|(\mathfrak{a}_2,s_2,Q_2(\cdot))-(\mathfrak{a}_1,s_1,Q_2(\cdot))\|.
\end{equation}
It follows from \eqref{eq:adot-sdot-formulas2} that if the denominator $A_0(\mathfrak{a},s,Q(\cdot);x)$ is bounded away from zero on $D$ for all $x$ near $x_c$, then $\dot{\mathfrak{a}}$ and $\dot{s}$ will be continuous in $x$ and will satisfy estimates of the form
\begin{equation}
\begin{split}
|\dot{\mathfrak{a}}(\mathfrak{a}_2,s_2,Q_2(\cdot);x)-\dot{\mathfrak{a}}(\mathfrak{a}_1,s_1,Q_1(\cdot);x)|&\le c\|(\mathfrak{a}_2,s_2,Q_2(\cdot))-(\mathfrak{a}_1,s_1,Q_1(\cdot))\|,\\
|\dot{s}(\mathfrak{a}_2,s_2,Q_2(\cdot);x)-\dot{s}(\mathfrak{a}_1,s_1,Q_1(\cdot);x)|&\le c\|(\mathfrak{a}_2,s_2,Q_2(\cdot))-(\mathfrak{a}_1,s_1,Q_1(\cdot))\|
\end{split}
\label{eq:adot-sdot-Lipschitz}
\end{equation}
for all $(\mathfrak{a}_j,s_j,Q_j(\cdot))\in D$ for $j=1,2$.  But by a residue calculation at $q=0$ using \eqref{eq:fa-derivs-criticality},
\begin{equation}
A_0(a_c,0,Q_c(\cdot);x_c)=
\frac{i\pi}{2Q_c'(0)^3}\left(\frac{243}{2}\right)^{1/6}=i\pi 3^{1/3}2^{-7/15}\neq 0,
\end{equation}
which therefore implies \eqref{eq:adot-sdot-Lipschitz} as long as $\delta$ is sufficiently small.

Now consider \eqref{eq:Ydot-a-s} with $\dot{\mathfrak{a}}$ and $s$ eliminated in terms of the integrals $X_0$, $X_2$, $A_0$, and $A_2$ by means of \eqref{eq:adot-sdot-formulas2}.  The denominator is an analytic function of $q$ having either four simple roots or one four-fold root at $q=0$ (for $s=0$ only).  The numerator is an analytic function of $q$ that by construction vanishes at the four simple roots of the denominator when $s\neq 0$, and therefore the fraction $\dot{Q}$ is analytic in an $s$-independent neighborhood of the origin $q=0$ as long as $s\neq 0$.  But by the Argument Principle and continuity of $\dot{\mathfrak{a}}$ and $\dot{s}$ with respect to $s$, the numerator of $\dot{Q}$ vanishes at the origin to fourth order when $s=0$, and therefore $\dot{Q}$ is analytic in $q$ even if $s=0$.  Moreover, $\dot{Q}$ is an odd function of $q$ because $Q\in\mathcal{B}_0$ is odd.  Therefore, $(\mathfrak{a},s,Q(\cdot))\in D$ implies that $\dot{Q}\in\mathcal{B}_0$ as a function of $q$.  Using the continuity of $\dot{\mathfrak{a}}$ and $s$ in $x$ near $x_c$ and the Lipschitz estimates \eqref{eq:adot-sdot-Lipschitz} then shows that $\dot{Q}$ is continuous from $x$ near $x_c$ in $\mathbb{C}$ to $\mathcal{B}_0$, and that, for some $c$,
\begin{equation}
\|\dot{Q}(\mathfrak{a}_2,s_2,Q_2(\cdot);x)-\dot{Q}(\mathfrak{a}_1,s_1,Q_1(\cdot);x)\|_0\le
c\|(\mathfrak{a}_2,s_2,Q_2(\cdot))-(\mathfrak{a}_1,s_1,Q_1(\cdot))\|
\label{eq:Ydot-Lipschitz}
\end{equation}
holds for all $(\mathfrak{a}_j,s_j,Q_j(\cdot))\in D$ for $j=1,2$. 

Therefore, the triple $(\dot{\mathfrak{a}},\dot{s},\dot{Q}(\cdot))$ constitutes a map from $((\mathfrak{a},s,Q(\cdot)),x)\in D\times (x_c-\delta,x_c+\delta)$ to $\mathcal{B}$ that is continuous in the second argument and Lipschitz in the first.  It follows from local existence and uniqueness theory for differential equations in Banach spaces that for $x-x_c$ sufficiently small there exists a unique solution $(\mathfrak{a}(x),s(x),Q(\cdot;x))\in D$ of the first-order differential equations \eqref{eq:Ydot-a-s} and \eqref{eq:adot-sdot-formulas1} with initial condition $(\mathfrak{a}(x_c),s(x_c),Q(\cdot;x_c))=(a_c,0,Q_c(\cdot))$.  It follows further that, if $\mathfrak{a}=\mathfrak{a}(x)$ and $s=s(x)$, then, for all $x$ near $x_c$, $W(z;x):=-Q((\mathfrak{a}(x)-z)^{1/2};x)^2$ is a univalent map satisfying \eqref{eq:basic-h-equation} on the three contours near $z=\mathfrak{a}$ for which the jump matrix of $\mathbf{O}$ depends on $z$.
Indeed, oddness of $Q$ means that the right-half $q$-plane will be mapped under $Q$ to a curve through the origin symmetric with respect to reflection through the origin.  Since $Q'(q)$ is controlled by the Cauchy Integral Formula, the slope of this curve will be nearly vertical for $x$ close enough to $x_c$.  Upon going back to $W$ by $W=-Q^2$, the region to the right of the boundary curve will be mapped to a single branch cut (both halves agree because the curve is symmetric) emanating from $W=0$ to the right, and its slope will be small near the origin if $x$ is close to $x_c$.  Therefore the mapping $z\mapsto W$ will be well-defined and invertible on the three rays  of the jump contour for which the jump matrix depends on $z$ (we only have to exclude the contour along which $\mathbf{O}_+(z)=\mathbf{O}_-(z)(-i\sigma_1)$ which may indeed interfere with the branch cut, but which does not matter because the jump is independent of $z$).

By residue calculations at $q=0$, we have $A_2(a_c,0,Q_c(\cdot);x_c)=0$ and
(using \eqref{eq:fx-deriv-criticality} and \eqref{eq:Yc-prime-zero}) we calculate $X_0(a_c,0,Q_c(\cdot);x_c)=-i\pi 2^{-14/15}3^{-1/3}$.  This gives the claimed value of $s'(x_c)$ and hence completes the proof of the lemma.
\end{proof}

\end{document}